\documentclass[reqno,a4paper,10pt]{amsart}
\title{Dynamical properties of spatial discretizations of a generic homeomorphism}
\date{\today}
\usepackage[latin1]{inputenc}
\usepackage[english]{babel}
\usepackage[T1]{fontenc}
\usepackage{amsfonts}
\usepackage{amsmath}
\usepackage{amssymb}
\usepackage{stmaryrd}
\usepackage{amsthm}
\usepackage{enumerate}
\usepackage{tikz}
\usepackage[pdftex,colorlinks=true,linkcolor=blue,citecolor=blue,urlcolor=blue]{hyperref}
\usepackage{float}
\usepackage{changepage}
\usepackage{caption}

\author{Pierre-Antoine Guihéneuf}
\DeclareRobustCommand{\SkipTocEntry}[5]{}
\newtheorem{lemme}{Lemma}%[part]
\newtheorem{theoreme}[lemme]{Theorem}
\newtheorem{prop}[lemme]{Proposition}
\newtheorem{coro}[lemme]{Corollary}
\newtheorem{conj}[lemme]{Conjecture}
\newtheorem*{theorem}{Theorem}
\newtheorem*{propo}{Proposition}
\theoremstyle{definition}
\newtheorem{definition}[lemme]{Definition}
\newtheorem{notation}[lemme]{Notation}
\theoremstyle{remark}
\newtheorem*{ques}{Question}
\newtheorem{rem}[lemme]{Remark}

\newtheorem{ex}[lemme]{Example}

\newcounter{counconst}
\newenvironment{constat}{\bigskip\noindent\stepcounter{counconst}\fontfamily{cmss}\selectfont \thecounconst)}{\bigskip}

\newcommand{\N}{\mathbf{N}}

\newcommand{\R}{\mathbf{R}}
\newcommand{\Sn}{\mathfrak{S}}
\newcommand{\T}{\mathbf{T}}

\newcommand{\V}{\mathcal{V}}

\newcommand{\Z}{\mathbf{Z}}
\newcommand{\varep}{\varepsilon}
\newcommand{\Hom}{\mathrm{Homeo}}

\newcommand{\Leb}{\mathrm{Leb}}

\newcommand{\ud}{\mathrm{d}}
\newcommand{\Prb}{\mathcal{P}}
\newcommand{\tT}{\mathcal{T}}
\newcommand{\Th}{\operatorname{tanh}}
\newcommand{\card}{\operatorname{Card}}
\makeatletter

\newskip\@bigflushglue \@bigflushglue = -100pt plus 1fil

\def\bigcentering{\let\\\@centercr\rightskip\@bigflushglue%
\leftskip\@bigflushglue
\parindent\z@\parfillskip\z@skip}

\makeatother

\begin{document}

\begin{abstract}
This paper concerns the link between the dynamical behaviour of a dynamical system and the dynamical behaviour of its numerical simulations. Here, we model numerical truncation as a spatial discretization of the system. Some previous works on well chosen examples (see e.g. cite{Gamb-dif}) show that the dynamical behaviours of dynamical systems and of their discretizations can be quite different. We are interested in generic homeomorphisms of compact manifolds. So our aim is to tackle the following question: can the dynamical properties of a generic homeomorphism be detected on the spatial discretizations of this homeomorphism?% We study both conservative and dissipative homeomorphisms (i.e. with or without the assumption that the homeomorphism preserves a given measure).

We will prove that the dynamics of a single discretization of a generic conservative homeomorphism does not depend on the homeomorphism itself, but rather on the grid used for the discretization. Therefore, dynamical properties of a given generic conservative homeomorphism cannot be detected using a single discretization. Nevertheless, we will also prove that some dynamical features of a generic conservative homeomorphism (such as the set of the periods of all periodic points) can be read on a sequence of finer and finer discretizations.% The dissipative case is somehow opposite: we will see that the dynamics of a generic dissipative homeomorphism is well-approximated by the dynamics of any fine enough discretization of this homeomorphism.
\end{abstract}

\email{pierre-antoine.guiheneuf@math.u-psud.fr}
\address{Laboratoire de mathématiques CNRS UMR 8628\\Université Paris-Sud 11, Bât. 425\\91405 Orsay Cedex FRANCE}
\subjclass[2010]{37M05, 37M25, 37M99, 37B99, 37A05}
\keywords{Conservative generic homeomorphism, dissipative generic homeomorphism, dynamical property, discretization}

\maketitle

\tableofcontents

\section{Introduction}

When one tries to simulate a discrete time dynamical system with a computer, calculations are performed with a finite numerical precision. Numerical errors made at each iteration may add up, so that after a while, the numerically calculated orbit of a point will have nothing in common with the actual one. Nevertheless, a numerically calculated orbit is close to an actual orbit at any time, thus one can hope that the collective behaviour of numerically calculated orbits provides information about the collective behaviour of actual orbits. Interesting properties of dynamical systems are often \emph{asymptotic} (i.e. concern the behaviour of the system in the long term) on the one hand, and \emph{collective} (i.e. concern the collective behaviour of orbits or even the action of the system on open sets and probability measures) on the other. The purpose of our study is to tackle the question: can such properties be detected on numerical simulations?

We will model numerical truncation as a spatial discretization\footnote{See page \pageref{pagefin} about the accuracy of this model.}. Consider a dynamical system whose phase space is a manifold $X$ and whose evolution law is given by a map $f$ from $X$ into itself. Roughly speaking, a numerical simulation of the system with a precision $10^{-N}$ replaces the continuous phase space $X$ by a discrete space $E_N$ made of points of $X$ whose coordinates are decimal numbers with at most $N$ decimal places, and replaces the map $f$ by its discretization $f_N: E_N \to E_N$ that maps $x\in E_N$ to the point (or one of the points) of $E_N$ nearest $f(x)$. We can clarify a bit the question raised in the previous paragraph:

\begin{ques}
Which dynamical properties of a homeomorphism $f$ can be read on the dynamics of its discretizations $(f_N)_{N\geq 0}$?
\end{ques}

Questions about discretizations of dynamical systems have mainly been studied empirically, using numerical simulations of well-chosen examples (see e.g. \cite{Lanf-inf,MR907194,MR1169615,MR929137,MR835874}). For instance, a particular attention has been paid to the obtaining of physical measures by numerical calculations: see e.g. \cite{MR0478237,MR534103}, \cite{Boya-comp, Gora-why} or \cite{MR2279269}. These papers explain why in some cases absolutely continuous invariant measures of the initial system can be obtained from numerical simulations, for example when the system is uniquely ergodic. However there are situations where (serious) problems arise: in a short article, J.-M. Gambaudo and C. Tresser \cite{Gamb-dif} \label{refG}give the examples of two simple homeomorphisms of $[0,1]^2$ which both have attractors that attract most of the orbits. Nevertheless, those attractors are undetectable in practice, simply because the connected components of their basins are much too small. Thus, the actual dynamics of these homeomorphisms does not have much in common with that observed on simulations.

In the early 90s, a significant work has been done by  a group of researchers including among others P. Diamond\label{DiaDia}, P. Kloeden, V. Kozyakin, J. Mustard and A. Pokrovskii. Their point of view lies between theory and practice: for instance in \cite{MR1265228,MR1331572,MR1354569,MR1387977} they define some dynamical properties that are robust under discretization. Nevertheless, the precision where one can detect these properties on discretizations might be huge, so that these notions could be unusable in practice. On the other hand, in \cite{MR1353178,MR1392078,MR1400185,MR1481914,MR1445902} the authors empirically remark that some quantities related to the discretizations of dynamical systems in dimension 1, such as the proportion of points in the basin attraction of the fixed point 0, the distribution of the lengths of cycles etc., are similar to same quantities for random applications with an attractive centre.

From a more theoretical point of view, S. Luzzatto and P. Pilarczyk conducted recently a quite interesting study of the modeling of a computer discretization by multivalued in \cite{MR2776399} (see also \cite{MR1403460} and \cite{MR1387977}). Note that this modeling is very different from that used by most authors and which is described above. Their paper includes a discussion about the relationship between the continuous and the discrete: at first glance properties of the original system can be deduced from a finite number of discretizations only when these properties are robust, which is relatively rare. To overcome this difficulty the authors define what they call finite resolution properties, in particular such properties can be verified in finite time by computational methods.

In \cite{Ghys-vari}, \'E. Ghys made a complete study of the dynamical behaviour of discretizations of the Anosov automorphism of the torus $\T^2$ given by $(x,y)\mapsto (y,x+y)$ (see also \cite{MR1176587}). \'E. Ghys notes that the dynamics of these discretizations (he shows that they are permutations with remarkably small orders) has nothing to do with the actual dynamics of the automorphism, which is a paradigm of chaotic dynamics. The very specific arithmetical properties of this example let one hope that such a difference between the actual dynamics of a system and the dynamics of its discretizations might be exceptional. This hope is supported by more recent results of S. Hayashi in \cite{Haya-obs}, roughly speaking, he shows that every $C^1$-diffeomorphism of a compact manifold with no positive Lyapunov exponent can be approximated by another diffeomorphism whose attractors are ``observable in practise''. Furthermore, T. Miernowski made a fairly comprehensive study of discretizations of circle homeomorphisms \cite{MR2279269}. He basically shows that the dynamical properties of a typical (generic or prevalent) circle homeomorphism/diffeomorphism can be read on the dynamics of its discretizations. All these results suggest to tackle the above question not for \emph{arbitrary} homeomorphisms but rather for \emph{typical} homeomorphisms; it is the point of view we will adopt in this article.
\bigskip

In this paper we study dynamical properties of discretizations of generic homeomorphisms in the sense of Baire. We establish properties for both \emph{dissipative}, i.e. arbitrary homeomorphisms of $X$, and \emph{conservative} homeomorphisms, i.e. homeomorphisms of $X$ that preserve a given good probability measure.

The results we obtained concern generic homeomorphisms of a compact manifold (with boundary) of dimension $n\geq 2 $. In this introduction we state the main results for the torus $\T^n$ and for ``uniform'' discretizations  on $\T^n$. The general framework in which the results are valid is defined in section \ref{0..2}.

Let $n\geq 2$. For all $N\in\N$, let $E_N$ be the finite subset of the torus $\T^n=\R^n/\Z^n$ made of the points whose coordinates are decimal numbers with at most $N$ decimal places:
\[E_N=\left\{\left(\frac{i_1}{10^N},\dots,\frac{i_n}{10^N}\right) \in \T^n \ \Big| \ (i_1,\dots,i_n)\in\Z^n\right \}.\]
Let $P_N$ be a projection of $\T^n$ on $E_N$ minimizing the euclidian distance: for $x\in \T^n$, $P_N(x)$ is one of the points of $E_N$ nearest $x$. Let $f_N:E_N\to E_N$ be the discretization of the homeomorphism $f:\T^n\to\T^n$ according to $E_N$ defined by $f_N=P_N\circ f$.
\bigskip

We will prove many results, concerning various aspects of the dynamics of the discretizations. In the rest of this introduction, we try to organize those results according to some ``lessons'':

\begin{constat}
The dynamics of a single discretization of a generic homeomorphism has in general nothing to do with the dynamics of the initial homeomorphism.
\end{constat}

Given a finite set $E$, the dynamics of any finite map $\sigma : E\to E$ is quite simple: given $x\in E$, the orbit $(\sigma^k(x))_k$ is preperiodic. Therefore, the union of periodic orbits of $\sigma$, which we will denote by $\Omega(\sigma)$ and called the \emph{maximal invariant set} of $\sigma$, is exactly the union of the $\omega$-limits sets of points of $E$. To study the dynamics of $\sigma$ one can focus on quantities such as the cardinality of $\Omega(\sigma)$, the stabilization time of $\sigma$ (i.e. the smallest $t\in\N$ such that $\sigma^t(E)=\Omega(\sigma)$), the number of orbits of $\sigma_{|\Omega(\sigma)}$, their lengths, the period of $\sigma_{|\Omega(\sigma)}$\dots

Another interesting dynamical quantity for a finite map $\sigma :  E\to E$ is the \emph{recurrence rate}: it is the ratio between the cardinalities of $\Omega(\sigma)$ and that of $E$. The starting point of our article is a question from \'E. Ghys (see chapter 6 of \cite{Ghys-vari}): for a generic conservative homeomorphism of the torus, what is the asymptotical behaviour of the sequence of recurrence rates of $f_N$? A partial answer to this question was obtained by T. Miernowski in the second chapter of his thesis.

\begin{theorem}[Miernowski]
For a generic conservative homeomorphism\footnote{I.e. there is a $G_\delta $ dense subset of the set of conservative homeomorphisms of the torus on which the conclusion of theorem holds.} $f$, there are infinitely many integers $N$ such that the discretization $f_N$ is a cyclic permutation.
\end{theorem}

To prove this theorem, T. Miernowski combines a genericity argument with a quite classical technique in generic dynamics of homeo(and auto)morphisms: approximation by permutations (see e.g. \cite{Kato-metr}, \cite{Alpe-appr}, \cite{MR0097489}), and more precisely Lax's theorem (see \cite{MR0272983}, see also \cite{MR1307740,MR1453713} for a generalisation and some simulations in dimension 1), which states that any conservative homeomorphism of the torus can be approximated by cyclic permutations of the discretization grids. In fact this proof can be generalized to obtain many results about discretizations. We will establish some variants of Lax theorem; each of them, combined with a genericity argument, leads to a result for discretizations of generic homeomorphisms. For instance:

\begin{theorem}
For a generic conservative homeomorphism $f$, tehre exists $C>0$ such that there are infinitely many integers $N$ such that the cardinality of $\Omega(f_N)$  is smaller than $C$.
\end{theorem}

Note that the combination of these two theorems answer the question of \'E. Ghys: for a generic homeomorphism $f$, the sequence of the recurrence rate of $f_N$ accumulates on both $0$ and $1$, one can even show that it accumulates on the whole segment $[0,1]$.

Let us move on to the behaviour on the maximal invariant set $\Omega(f_N)$. Another variations of Lax's theorem lead to a theorem that enlightens the behaviour of the discretizations on their maximal invariant set:

\begin{theorem}
For a generic conservative homeomorphism $f$, there are infinitely many integers $N$ such that the discretization $f_N$ is a cyclic permutation and for all $M\in\N$, there are infinitely many integers $N$ such that $f_N$ is a permutation of $E_N$ having at least $M$ periodic orbits.
\end{theorem}

To summarize, generically, infinitely many discretizations are cyclic permutations, but also infinitely many discretizations are highly non injective or else permutations with many cycles. In particular, it implies that for all $x\in X$, tehre exists infinitely many integers $N$ such the orbit of $x_N$ under $f_N$ does not shadow the orbit of $x$ under $f$: in this sense, generically, the dynamics of discretizations does not reflect that of the homeomorphism.

\begin{constat}
A dynamical property of a generic homeomorphism can not be deduced from the frequency it appears on discretizations either.
\end{constat}

The previous theorems express that the dynamics of a single discretization does not reflect the actual dynamics of the homeomorphism. However, one might reasonably expect that the properties of the homeomorphism are transmitted to many discretizations. More precisely, one may hope that given a property $(P)$ about discretizations, if there are many $N$ such that the discretization $f_N$ satisfies $(P)$, then the homeomorphism satisfies a similar property. It is not so, for instance:

\begin{theorem}
For a generic conservative homeomorphism $f$, when $M$ goes to infinity, the proportion of integers $N$ between $1$ and $M$ such that $f_N$ is a cyclic permutation accumulates on both $0$ and $1$.
\end{theorem}

In fact, for most of the properties considered in the previous paragraph, the frequency they appear on discretizations of orders smaller than $M$ accumulates on both $0$ and $1$ when $M$ goes to infinity.

\begin{constat}
Some important dynamical features of a generic homeomorphism can be detected by looking at some dynamical features of \emph{all} the discretizations.
\end{constat}

We have observed that one can not detect the dynamics of a generic homeomorphism when looking at the dynamics of its discretizations, or even at the frequency some dynamics appears on discretizations. Nevertheless, some dynamical features can be deduced from the analogous dynamical features of \emph{all} the discretizations. This idea of convergence of the dynamics when looking at arbitrary large precisions can be related to the work of P. Diamond \emph{et al} (see page \pageref{DiaDia}). For instance, the periods of periodic orbits of a homeomorphism can be read on the periods of periodic orbits of its discretizations:

\begin{propo}
A generic homeomorphism $f$ has a periodic orbit with period $p$ if and only if there exists infinitely many integers $N$ such that $f_N$ has a periodic orbit with period~$p$.
\end{propo}

We will also prove a theorem in the same vein for invariant measures. It expresses that the set of invariant measures of the homeomorphism can be deduced from the sets of invariant measures of its discretizations. More precisely: 

\begin{theorem}
Let $\mathcal M_N$ be the set of probability measures on $E_N$ that are invariant under $f_N$. For a generic conservative homeomorphism $f$, the upper limit over $N$ (for Hausdorff topology) of the sets $\mathcal M_N$ is exactly the set of probability measures that are invariant under~$f$.
\end{theorem}

\begin{constat}
Physical measures of a generic homeomorphism can not be detected on discretizations.
\end{constat}

Given $x\in \T^n$, the Birkhoff limit of $x$ is defined (when it exists) as the limit of $\frac 1m \sum_{i=0}^{m-1}{f}_*^i \delta_x$ when $m$ goes to infinity; the \emph{basin} of a measure $\mu$ is the set of points $x$ whose Birkhoff limit coincides with $\mu$. Heuristically, the basin of a measure is the set of points that can see the measure. A Borel measure is said \emph{physical} if its basin has positive Lebesgue measure (see e.g. \cite{Youn-wha}). The heuristic idea underlying this concept is that physical measures are the invariant measure which can be detected ``experimentally'' (since many initial conditions lead to these measures). Indeed, some experimental results on specific examples of dynamical systems show that they are actually the measures that are detected in practice (see e.g. \cite{MR0478237,MR534103} or \cite{Boya-comp, Gora-why}); moreover if the dynamical system is uniquely ergodic then the invariant measure appears naturally on discretizations (see \cite{MR2279269}).

According to this heuristic and these results, one could expect physical measures to be the only invariant measures that can be detected on discretizations of generic conservative homeomorphisms. This is not the case: for a generic conservative homeomorphism, there exists a unique physical measure, say Lebesgue measure. According to the previous theorem, invariant measures of the discretizations accumulate on all the invariant measures of the homeomorphism and not only on Lebesgue measure.

However, one can still hope to distinguish the physical measure from other invariant measures. For this purpose, we define the canonical physical measure $\mu^f_N$ associated to a discretization $f_N$: it is the limit in the sense of Cesàro of the images of the uniform measure on $E_N$ by the iterates of $f_N$: if $\lambda_N$ is the uniform measure on $E_N$, then
\[\mu^f_N = \lim_{M\to\infty}\frac{1}{M}\sum_{m=0}^{M-1}(f_N^m)_*\lambda_N.\]
This measure is supported by the maximal invariant set $\Omega(f_N)$; it is uniform on every periodic orbit and the total weight of a periodic orbit is proportional to the size of its basin of attraction. The following theorem expresses that these measures accumulate on the whole set of $f$-invariant measures: physical measures can not be distinguished from other invariant measures on discretizations, at least for generic homeomorphisms.

\begin{theorem}
For a generic conservative homeomorphism $f$, the set of accumulations points of the sequence $(\mu^f_N)_{N\in\N}$ is the set of all $f$-invariant measures.
\end{theorem}

\begin{constat}
The dynamics of discretizations of a generic dissipative homeomorphism tends to that of the initial homeomorphism.
\end{constat}

We then study properties of discretizations of generic \emph{dissipative} homeomorphisms\footnote{I.e. without assumption of preservation of a given measure.}. The basic tool is the \emph{shredding lemma} of F. Abdenur and M. Andersson \cite{MR3027586}, which implies that a generic homeomorphism has a ``attractor dynamics''. This easily transmits to discretizations, for example the basins of attraction of the homeomorphism can be seen on all the fine enough discretizations. Moreover there is \emph{convergence} of the dynamics of discretizations $f_N$ to that of $f$:

\begin{theorem}
For a generic dissipative homeomorphism $f$, for all $\varep>0$ and all $\delta>0$, there is a full measure dense open subset $O$ of $\T^n$ such that for all $x\in O$, all $\delta>0$ and all $N$ large enough, the orbit of $x_N$ under $f_N$ $\delta$-shadows the orbit\footnote{I.e. for all $k\in\N$, $d(f_N^k(x_N),f^k(x))<\delta$.} of $x$ under $f$.
\end{theorem}

Thus, one can detect on discretizations the dynamics of a generic dissipative homeomorphism, which is mainly characterized by position of the attractors and of the corresponding basins of attraction. Note that this behaviour is in the opposite of the conservative case, where the individual behaviour of discretizations does not indicate anything about the actual dynamics of the homeomorphism. 

\begin{constat}
In practice, one can not deduce the dynamics of a conservative homeomorphism from its discretizations, and the dynamics of a dissipative homeomorphism can be detected on discretizations only if the basins are large enough.
\end{constat}

Finally, we compare our theoretical results with the reality of numerical simulations. Indeed, it is not clear that the behaviour predicted by our results can be observed on computable discretizations of an homeomorphism defined by a simple formula. On the one hand, all our results are valid ``for \emph{generic} homeomorphisms'': nothing indicates that these results apply to practical examples of homeomorphisms defined by simple formulas. On the other hand, results such as ``there are infinitely many integers $N$ such that the discretization of order $N$\dots'' provide no control over the integers $N$ involved, they may be so large that the associated discretizations are not calculable in practice.

For conservative homeomorphisms, our numerical simulations produce mixed results. From a quantitative viewpoint, the behaviour predicted by our theoretical result cannot be observed on our numerical simulations. For example, we can not observe any discretization whose recurrence rate is equal to $1$ (i.e. which is a permutation). This is nothing but surprising: the events pointed out by the theorems are a priori very rare. For instance, there is a very little proportion of bijective maps among maps from a given finite set into itself. From a more qualitative viewpoint, the behaviour of the simulations is quite in accordance with the predictions of the theoretical results. For example, for a given conservative homeomorphism, the recurrence rate of a discretization depends a lot on the size of the grid used for the discretization. Similarly, the canonical invariant measure associated with a discretization of a homeomorphism $f$ does depend a lot on the size of the grid used for the discretization. 

We also carried out simulations of dissipative homeomorphisms. The results of discretizations of a small perturbation of identity (in $C^0$ topology) may seem disappointing at first sight: the attractors of the initial homeomorphisms can not be detected, and there is little difference with the conservative case. This behaviour is similar to that highlighted by J.-M. Gambaudo and C. Tresser in \cite{Gamb-dif} (see page \ref{refG}). That is why it seemed to us useful to test an homeomorphism which is $C^0$ close to the identity, but whose basins are large enough. In this case the simulations point out a behaviour that is very similar to that described by theoretical results, namely that the dynamics converges to the dynamics of the initial homeomorphism. In facts, one have actually observed behaviours as described by theorems only for examples of homeomorphisms with a very few number of attractors.
\bigskip

We carried out many other numerical simulations, they can be found on the web page: \url{http://www.math.u-psud.fr/~guiheneu/Simulations.html}
\bigskip

In section \ref{0..2} we will present the framework. Sections \ref{Grosse} to \ref{grobra} concern conservative homeomorphisms. More precisely, the results concerning the behaviour of a single discretization are set out in section \ref{partie 1.3}, and those concerning the average behaviour of discretizations in section \ref{bofbof}. We then come to the results about all the discretizations in section \ref{Sec8} and about physical measures in section \ref{grobra}. The behaviour of discretizations of generic dissipative homeomorphisms is established in section \ref{label}. Finally, the results of numerical simulations are presented in part \ref{partietrois}.

\section{Framework}\label{0..2}

\addtocontents{toc}{\SkipTocEntry}\subsection*{The manifold $X$ and the measure $\lambda$} The results stated in the introduction for the torus $\T^n$ and the Lebesgue measure $\Leb$ actually extend to any smooth connected manifold $X$ with dimension $n\ge 2$, compact and possibly with boundary, endowed with a Riemannian metric~$d$. \emph{We fix once and for all such a manifold $X$ endowed with the metric $d$.} In the general case, Lebesgue measure on $\T^n$ can be replaced by a \emph{good measure} $\lambda$ on $X$:

\begin{definition}\label{bonne mesure}
A Borel probability measure $\lambda$ on $X$ is called a \emph{good measure}, or an \emph{Oxtoby-Ulam} measure if it is nonatomic, it has total support (it is positive on nonempty open sets) and it is zero on the boundary of $X$.
\end{definition}

\emph{We fix once and for all a good measure $\lambda$ on $X$.}

\begin{notation}
We denote by $\Hom(X)$ the set of homeomorphisms of $X$, endowed by the metric $d$ defined by:
\[d(f,g) = \sup_{x\in X} d(f(x),g(x)).\]
We denote by $\Hom(X,\lambda)$ the subset of $\Hom(X)$ made of the homeomorphisms that preserve the measure $\lambda$, endowed with the same metric $d$. Elements of $\Hom(X)$ will be called \emph{dissipative} homeomorphisms and which of $\Hom(X,\lambda)$ \emph{conservative} homeomorphisms. 
\end{notation}

\addtocontents{toc}{\SkipTocEntry}\subsection*{Generic properties in $\Hom(X)$ and $\Hom(X,\lambda)$} The topological spaces $\Hom(X)$ and $\Hom(X,\lambda)$ are Baire spaces (see \cite{MR2931648}), i.e.: the intersection of every countable collection of dense open sets is dense. We call $G_\delta$ a countable intersection of open sets; a property satisfied on at least a $G_\delta$ dense set is called \emph{generic}. Note that in a Baire space, generic properties are stable under intersection.

Sometimes we will use the phrase ``for a generic homeomorphism $f\in \Hom(X)$ (resp. $\Hom(X,\lambda)$), we have the property $(P)$''. By that we will mean that ``the property $(P)$ is generic in $\Hom(X)$ (resp. $\Hom(X,\lambda)$)'', i.e. ``there exists $G_\delta$ dense subset $G$ of $\Hom(X)$ (resp. $\Hom(X,\lambda)$), such that every $f\in G$ satisfy the property $(P)$''.

\addtocontents{toc}{\SkipTocEntry}\subsection*{Discretization grids, discretizations of a homeomorphism} We now define a more general notion of discretization grid. Some of the assumptions about these grids will be useful later, it will be the subject of the next paragraph.

\begin{definition}[Discretization grids]\label{grillmiam}
A \emph{sequence of discretization grids} on $X$ is a sequence $(E_N)_{N\in\N}$ of discrete subsets of $X\setminus\partial X$, such that the grids are more and more precise: for all $\varep>0$, there exists $N_0\in\N$ such that for all $N\ge N_0$, the grid $E_N$ is $\varep$-dense. We denote by $q_N$ the cardinality of $E_N$.
\end{definition}

\emph{We fix once and for all a sequence $(E_N)_{N\in \N}$ of discretization grids on $X$.} We can now define discretizations associated to these grids:

\begin{notation}[Discretizations]\label{notA}
Let $P_N$ be a projection of $X$ on $E_N$ (the projection of $x_0\in X$ on $E_N$ is some $y_0\in E_N$ minimizing the distance $d(x_0,y)$ when $y$ runs through $E_N$). Such a projection is uniquely defined out of the set $E'_N$ made of the points $x\in X$ for which there exists at least two points minimizing the distance between $x$ and $E_N$. On $E'_N$ the map $P_N$ is chosen arbitrarily (nevertheless measurably). For $x\in X$ we denote by $x_N$ the \emph{discretization of order $N$} of $x$, defined by $x_N = P_N(x)$. For $f\in\Hom(X)$ we denote by $f_N : E_N\to E_N$ the \emph{discretization of order $N$} of $f$, defined by $f_N = P_N\circ f$.

Let $\mathcal{D}_N$ be the set of homeomorphisms $g$ such that $g(E_N)\cap E'_N=\emptyset$.

If $\sigma : E_N\to E_N$ and $f\in \Hom(X)$, we denote by $d_N(f,\sigma)$ the distance between $f_{|E_N}$ and $\sigma$, considered as maps from $E_N$ into $X$.
\end{notation}

\begin{rem}
One might wonder why the points of the discretization grids are supposed to be \emph{inside} $X$. The reason is simple: a homeomorphism $f$ of $X$ sends $\partial X$ on $\partial X$. Putting points of some grids on the edge could perturb the dynamics of discretizations $f_N$. In particular it would introduce at least one orbit\footnote{Recall that orbit means forward orbit.} with length smaller than $\card(E_N \cap\partial X) $.
\end{rem}

\begin{rem}\label{MerciReferee}
As the exponential map is a local diffeomorphism, the sets $E'_N$ are closed and have empty interior for every $N$ large enough. Subsequently, we will implicitely suppose that the union $\bigcup_{N\in\N} E'_N$ is an $F_\sigma$ with empty interior. It is not a limitating assumption as we will focus only on the behaviour of the discretizations for $N$ going to $+\infty$. It will allow us to restrict the study to the $G_\delta$ dense set $\bigcap_{N\in\N} \mathcal{D}_N$, which is the set of homeomorphisms whose $N$-th discretization is uniquely defined for all $N\in\N$.
\end{rem}

\addtocontents{toc}{\SkipTocEntry}\subsection*{Probability measures on $X$} From section \ref{Sec8}, we will be interested in \emph{ergodic} properties of discretizations of $f$. Denote $\Prb$ the set of Borel probability measures on $X$ endowed with the weak-star topology: a sequence $(\nu_m)_{m\in\N}$ of $\Prb$ tends to $\nu\in\Prb$ (denoted by $\nu_m\rightharpoonup\nu$) if for all continuous function $\varphi : X\to\R$,
\[\lim_{m\to\infty}\int_{X}\varphi\,\ud \nu_m = \int_{X}\varphi\,\ud \nu.\]
Under these conditions the space $\Prb$ is metrizable and compact, therefore separable (Prohorov, Banach-Alaoglu-Bourbaki theorem).

To study ergodic properties of homeomorphisms and their discretizations, we define natural invariant probability measures associated with these maps:

\begin{definition}\label{defmes}
For any nonempty open subset $U$ of $X$, we denote by $\lambda_{U}$ the normalized restriction of $\lambda$ on $U$, i.e. $\lambda_U=\frac{1}{\lambda(U)} \lambda_{|U}$. We also denote by $\lambda_{N,U}$ the uniform probability measure on $E_N\cap U$ and $\lambda_N = \lambda_{N,X}$. For $x\in X$ we denote by (when the limit exists) the \emph{Birkhoff limit} of $x$:
\[\mu^f_x = \lim_{m\to\infty}\frac 1m \sum_{i=0}^{m-1}{f}_*^i \delta_x,\]
and similarly for $f_N$. When it is well defined, we set
\[\mu^f_U = \lim_{m\to\infty}\frac 1m \sum_{i=0}^{m-1}{f}_*^i \lambda_U\]
and
\[\mu^f_{N,U} = \lim_{m\to\infty}\frac 1m \sum_{i=0}^{m-1}(f_N)_*^i\, \lambda_{N,U}.\]
Finally, we note $\mu^f_N = \mu^f_{N,X}$.
\end{definition}

We just define two types of invariant measures: on the one hand from a point $x$, the other from the uniform measure $\lambda$. The link between it is done by the following proposition:

\begin{prop}\label{convdom}
When $U$ is an open set whose almost every point admit a Birkhoff limit, the measure $\mu^f_U$ is well defined and satisfies, for every continuous map $\varphi : X\to\R$, 
\[\int_X\varphi\, \ud\mu^f_U = \int_U\left(\int_X \varphi\, \ud\mu^f_x\right)\ud\lambda_U.\]
Similarly,
\[\int_X\varphi\, \ud\mu^f_{N,U} = \int_U\left(\int_X \varphi\, \ud\mu^f_{N,x}\right)\ud\lambda_{N,U}.\]
\end{prop}

\begin{proof}[Proof of proposition \ref{convdom}]
It follows easily from the dominated convergence theorem.
\end{proof}

\addtocontents{toc}{\SkipTocEntry}\subsection*{Hypothesis on discretization grids} Previously, we have given a very general definition of the concept of sequence of discretization grids. In some cases, we will need additional technical assumptions about these sequences of grids. Of course all of them will be satisfied by the uniform discretization grids on the torus (as defined in the introduction).

The first assumption is useful to prove Lax's theorem (theorem \ref{Lax}), and therefore necessary only in the part concerning conservative homeomorphisms.

\begin{definition}[Well distributed and well ordered grids]\label{Ashe}
We say that a sequence of discretization grids $(E_N)_{N\in\N}$ is \emph{well distributed} if one can associate to each $x\in E_N$ a subset $C_{N,x}$ of $X$, which will be called a \emph{cube of order $N$}, such that:
\begin{itemize}
\item for all $N$ and all $x\in E_N$, $x\in C_{N,x}$,
\item for all $N$, $\{C_{N,x}\}_{x\in E_N}$ is a measurable partition of $X$: $\bigcup_{x\in E_N} C_{N,x}$ is full measure and for $x,y$ two distinct points of $E_N$, the intersection $C_{N,x}\cap C_{N,y}$ is null measure,
\item for a fixed $N$, all the cubes $C_{N,x}$ have the same measure,
\item the diameter of the cubes of order $N$ tends to $0$: $\max_{x\in E_N}\, \operatorname{diam}(C_{N,x})\underset{N\to+\infty}{\longrightarrow} 0$.
\end{itemize}

If $(E_N)_{N\in\N}$ is well distributed and if $\{C_{N,x}\}_{N\in\N,x\in E_N}$ is a family of cubes as below, we will say that $(E_N)_{N\in\N}$ is \emph{well ordered} if, for a fixed $N$, the cubes  $\{C_{N,x}\}_{x\in E_N}$ can be indexed by $\Z/q_N\Z$ such that two consecutive cubes (in $\Z/q_N\Z$) are close to each other: $\max_{i\in \Z/q_N\Z}\, \operatorname{diam}(C_N^i\cup C_n^{i+1})\underset{N\to+\infty}{\longrightarrow} 0$ (especially, it is true when the boundaries of two consecutive cubes overlap).
\end{definition}

At first glance, it can seem surprising that there is no link between the cubes and the projections. In fact, the existence of such cubes expresses that the grids ``fit'' the measure $\lambda$.

The two following definitions describe assumptions that will be useful especially to obtain properties in average.

\begin{definition}[Refining grids]\label{Ashe''}
We say that a sequence of discretization grids \emph{refines} if for all $N,N'\in\N$ such that $N\le N'$, one has $E_{N}\subset E_{N'}$.
\end{definition}

\begin{definition}[Self-similar grids]\label{Ashe'}
We say that a sequence of discretization grids $(E_N)_{N\in\N}$ is \emph{self similar} if for all $\varep>0$, there exists $N_0, N_1\in\N$ such that for all $N\ge N_1$, the set $E_{N}$ contains disjoint subsets $\widetilde E_N^1, \dots, \widetilde E_N^{\alpha_N}$ whose union fills a proportion greater than $1-\varep $ of $E_{N}$, and such that for all $j$, $\widetilde E_N^j$ is the image of the grid $E_{N_0}$ by a bijection $h_j$ which is $\varep$-close to identity.

We say that a sequence of discretization grids $(E_N)_{N\in\N}$ is \emph{strongly self similar} if it is self similar and for all $N\ge N_0$, one of the $h_j$ equals to identity.
\end{definition}

\begin{rem}
One easily verify that the hypothesis ``being strongly self-similar'' implies both ``being self-similar'' and ``refining''.
\end{rem}

\section{Some examples of discretization grids}\label{exgrilles}

In the previous section we set properties on discretizations --- namely being well distributed, being well ordered, refining, being strongly self-similar or self-similar --- that will be used subsequently. In this section we give some examples of grids that verify some of these assumptions.

\addtocontents{toc}{\SkipTocEntry}\subsection*{Uniform discretization grids on the torus}
The simplest example, which will be used for the simulations, is that of the torus $\T^n = \R^n/\Z^n$ of dimension $n\ge 2$ endowed with discretizations called \emph{uniform} discretizations, defined from the fundamental domain $I^n=[0,1]^n$ of $\T^n$: take an increasing sequence of integers $(k_N)_{N\in\N}$ and set
\[E_N = \left\{\left(\frac{i_1}{k_N},\dots,\frac{i_n}{k_N}\right)\in I^n \big|\ \forall j,\, 0\le i_j\le {k_N}-1\right\},\]
\[C_{N,(i_1/N,\dots,i_n/N)} = \prod_{j=1}^n \left[\frac{i_j}{k_N}-\frac{1}{2k_N}\, ,\, \frac{i_j}{k_N}+\frac{1}{2k_N}\right].\]

We easily verify that this sequence of grids is well distributed, well ordered and self-similar. If we further assume that for any $N\in\N$, $k_N$ divides $k_{N+1}$ (which is true when $k_N = p^N$ with $p\ge 2$), then the sequence is strongly self-similar (therefore refines). When $k_N = p^N$ with $p=2$ (resp. $p=10$) the discretization performs what one can expect from a numerical simulation: doing a binary (resp. decimal) discretization at order $N$ is the same as truncating each binary (resp. decimal) coordinate of the point $x\in I^n$ to the $N$-th digit, i.e. working with a fixed digital precision\label{pagefin}\footnote{In fact, in practice the computer works in floating point format, so that the number of decimal places is not the same when the number is close to $0$ or not.}.

\begin{figure}
\begin{center}
\makebox[\textwidth]{\parbox{\textwidth}{
\begin{minipage}[c]{.32\linewidth}
\begin{tikzpicture}[scale=.6]
\draw[thick] (.5,.5) rectangle (5.5,5.5);
\foreach \k in {0,...,5}
 {\foreach \l in {0,...,5}
  {\draw[color=red!70!black] node at (\k+.5,\l+.5){\footnotesize$\bullet$};}}
\clip (.5,.5) rectangle (5.5,5.5);
\draw[blue] (0,0) grid (6,6);
\end{tikzpicture}
\end{minipage}\hfill
\begin{minipage}[c]{.32\linewidth}
\begin{tikzpicture}[scale=.7714]
\draw[blue] (0,0) grid (4,4);
\draw[thick] (0,0) rectangle (4,4);
\foreach \k in {1,...,4}
 {\foreach \l in {1,...,4}
  {\draw[color=red!70!black] node at (\k*4/5,\l*4/5){\footnotesize$\bullet$};}}
\end{tikzpicture}
\end{minipage}\hfill
\begin{minipage}[c]{.32\linewidth}
\begin{tikzpicture}[scale=.7714]
\draw[blue] (0,0) grid (4,4);
\draw[thick] (0,0) rectangle (4,4);
\foreach \k in {1,...,4}
 {\foreach \l in {1,...,4}
  {\draw[color=red!70!black] node at (\k-.5,\l-.5){\footnotesize$\bullet$};}}
\end{tikzpicture}
\end{minipage}
}}
\end{center}
\caption{Uniform discretization grids of order $5$ on the torus $\T^2$ ($E^N$ left) and on the cube $I^2$ ($E_N^0$ middle and $E_N^1$ right) and their associated cubes}
\label{Grill}
\end{figure}
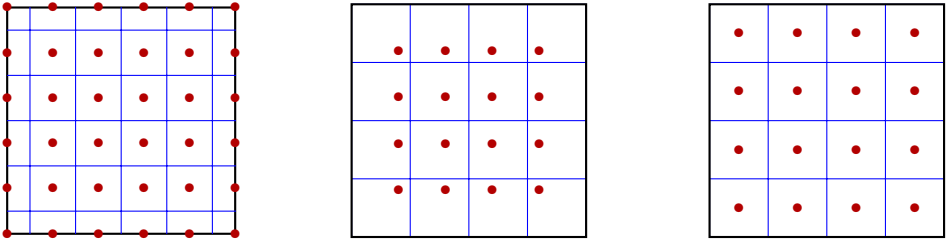

\addtocontents{toc}{\SkipTocEntry}\subsection*{Uniform discretization grids on the cube} Similarly, one can define uniform discretizations on the cube  $I^n = [0,1]^n$ by setting
\[E_N^0 = \left\{\left(\frac{i_1}{k_N},\dots,\frac{i_n}{k_N}\right)\in I^n \big|\ \forall j,\, 1\le i_j\le {k_N}-1\right\},\]
still with an increasing sequence $(k_N)_{N\in\N}$ of integers, to which are associated the cubes\footnote{Be careful, these cubes have their vertices on the grid of order $k_N-1$.} (see figure \ref{Grill})
\[C_{N,(i_1/N,\dots,i_n/N)} = \prod_{j=1}^n\left[\frac{i_j-1}{k_N-1}\, ,\, \frac{i_j}{k_N-1}\right].\]
As before, one easily verifies that this sequence of grids is well distributed, well ordered and self-similar (see figure \ref{soucubhihi}). If we further assume that for any $N\in\N$, $k_N$ divides $k_{N+1}$, then the sequence is strongly self-similar (therefore refines).

\begin{figure}
\begin{center}
\begin{tikzpicture}[scale=.7]
\draw (0,0) grid (7,7);
\draw[thick] (0,0) grid[step=2] (6,6);
\foreach \k in {1,...,7}
 {\draw[color=black] node at (\k*7/8,49/8){\footnotesize$\bullet$};
  \draw[color=black] node at (49/8,\k*7/8){\footnotesize$\bullet$};}
\foreach \k in {1,...,3}
 {\foreach \l in {1,...,3}
  {\draw[color=green!60!black] node at (\k*7/4-7/8,\l*7/4-7/8){\footnotesize$\bullet$};}}
\foreach \k in {1,...,3}
 {\foreach \l in {1,...,3}
  {\draw[color=blue!60!black] node at (\k*7/4-7/8,\l*7/4){\footnotesize$\bullet$};}}
\foreach \k in {1,...,3}
 {\foreach \l in {1,...,3}
  {\draw[color=cyan!60!black] node at (\k*7/4,\l*7/4-7/8){\footnotesize$\bullet$};}}
\foreach \k in {1,...,3}
 {\foreach \l in {1,...,3}
  {\draw[color=red!70!black] node at (\k*7/4,\l*7/4){\large$\bullet$};}}
\end{tikzpicture}
\caption{Self similarity of grids $E_N^0$}\label{soucubhihi}
\end{center}
\end{figure}
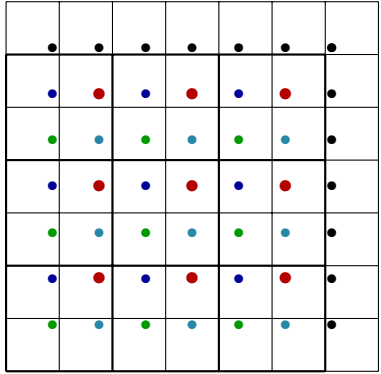

One can also take discretizations according to the centres of the cubes.
\[E_N^1 = \left\{\left(\frac{i_1+1/2}{k_N},\dots,\frac{i_n+1/2}{k_N}\right)\in I^n \big|\ \forall j,\, 0\le i_j\le {k_N}-1\right\}.\]
This time the cubes are much more natural (see figure \ref{Grill}):
\[C_{N,(i_1/N,\dots,i_n/N)} = \prod_{j=1}^n\left[\frac{i_j}{k_N}\, ,\, \frac{i_j+1}{k_N}\right].\]
Again, we easily verify that this sequence of grids is well distributed, well ordered and self-similar (but does not refine).

\addtocontents{toc}{\SkipTocEntry}\subsection*{Discretization grids on an arbitrary manifold }In fact, discretizations $E_N^0$ on $I^n$ can be generalized to an arbitrary manifold $X$ with the Oxtoby-Ulam theorem (see \cite{Oxto-meas}):

\begin{theoreme}[Oxtoby, Ulam]\label{Brown-mesure}
Under the assumptions that have been made on $X$ and $\lambda$, there exists a map $\phi : I^n\to X$ such that:
\begin{enumerate}
\item $\phi$ is surjective
\item $\phi_{|\mathring{I^n}}$ is an homeomorphism on its image
\item $\phi(\partial I^n)$ is a closed subset of $X$ with empty interior and disjoint from $\phi (\mathring{I^n})$
\item $\lambda(\phi(\partial I^n))=0$
\item $\phi^*(\lambda) = \Leb$, where $\Leb$ is Lebesgue measure.
\end{enumerate}
\end{theoreme}

This theorem allows us to define discretization grids on $X$ from uniform discretizations on the cube. We obtain the following informal proposition:

\begin{prop}
If there is a sequence $(E_N)_N$ of grids on $I^n$ whose elements are not on the edge of the cube, then its image by $\phi$ defines a sequence of grids on $X$ which satisfy the same properties as the initial grid on the cube.
\end{prop}

\begin{rem}
The example also includes the case where $X=I^n$, $\lambda=\Leb$ and where the grids are the images of the grids $E_N$ by a unique homeomorphism of $X$ preserving Lebesgue measure.
\end{rem}

\part{Discretizations of a generic conservative homeomorphism}

Recall that we have fixed once and for all a manifold $X$, a good measure $\lambda$ on $X$ and a sequence of discretization grids $(E_N)_{N\in\N}$ (see definition \ref{grillmiam}). It is further assumed that this sequence of grids is well distributed and well ordered (see definition \ref{Ashe}). In this section we will focus on discretizations of elements of $\Hom (X,\lambda)$, so homeomorphisms will always be supposed conservative.

\section{Dense types of approximations}\label{Grosse}

To begin with we define the notion of \emph{type of approximation}.

\begin{definition}
A \emph{type of approximation} $\tT = (\tT_N)_{N\in\N}$ is a sequence of subsets of the set $\mathcal F (E_N,E_N)$ made of applications from $E_N$ into itself.

Let $\mathcal{U}$ be an open subset of $\Hom (X,\lambda)$. A type of approximation $\tT$ is said to be \emph{dense} in $\mathcal{U}$ if for all $f\in \mathcal{U}$, $\varepsilon>0$ and all $N_0\in\N$, there exists $N\ge N_0$ and $\sigma_N\in \tT_N$ such that $d_N(f,\sigma_N)<\varepsilon$ (recall that $d_N$ is the distance between $f_{|E_N}$ and $\sigma_N$ considered as maps from $E_N$ into $X$).
\end{definition}

The goal of this paragraph is to obtain the following result: every dense type of approximation appears on infinitely many discretizations of a generic homeomorphism.

\begin{theoreme}\label{génécycl}
Let $\mathcal U$ be an open subset\footnote{We will often take $\mathcal U=\Hom (X,\lambda)$.} of $\Hom (X,\lambda)$ and $\tT$ be a type of approximation which is dense in $\mathcal U$. Then for a generic homeomorphism $f\in\mathcal U$ and for all $N_0\in\N$, there exists $N\ge N_0$ such that $f_{N}\in \tT_{N}$.
\end{theoreme}

When we will try to obtain properties of discretizations of generic homeomorphisms, this theorem will be the second step of the proof, the first being to show that the approximation type we are interested in is dense. In section \ref{partie 1.3} we will try to establish as many corollaries of this theorem as possible, showing that the approach ``prove that such approximation is dense then apply theorem \ref{génécycl}'' is a general method in the study of properties of discretizations\footnote{Although some properties can be proved in different ways: for instance corollary \ref{crush} can be proved in inserting horseshoes in the dynamics of a given homeomorphism, using the local modification theorem (theorem \ref{extension-sphères}, for a presentation of the technique in another context see section 3.3 of \cite{MR2931648}); also a variation of corollary \ref{corovar1} can be shown in perturbing any given homeomorphism such that it has a periodic orbit whose distance to the grid $E_N$ (grids are assumed to refine) is smaller than the modulus of continuity of $f$, so that the actual periodic orbit and the discrete orbit fit together.}.

To prove of theorem \ref{génécycl} we start from a dense type of approximations --- in other words discrete applications --- and we want to get properties of homeomorphisms. The tool that allows us to restore a homeomorphism from a finite map $\sigma_N : E_N\to E_N$, i.e. do the opposite of a discretization, is the finite map extension proposition:

\begin{prop}[Finite map extension]\label{extension} Let $(x_1,x_2,\dots,x_k)$ and $(y_1,y_2,\dots,y_k)$ be two $k$-tuples of pairwise distinct points in $X\setminus \partial X$ (i.e. for all $i\neq j$, $x_i\neq x_j$ and $y_i\neq y_j$). Then there exists $f\in\mathrm{Homeo}(X,\lambda)$ such that for all $i$, $f(x_i) = y_i$. Moreover $f$ can be chosen such that if $\max_i d(x_i,y_i)< \delta$ then $d(f,\mathrm{Id})<\delta$.
\end{prop}

The idea of the proof is quite simple: it suffice to compose homeomorphisms whose support's size is smaller than $\varep$ and which are central symmetries within this support. Then one constructs a sequence $(z_j)_{1 \le j \le k}$ such that $z_1 = x_1$, $y_1 = z_k$ and $d(z_j, z_{j+1}) <\varep/10$. Composing $k-1$ homeomorphisms as above, such that each one sends $z_j$ on $z_{j+1}$, one obtains a conservative homeomorphism which sends $x_0$ on $y_0$. Implementing these remarks is then essentially technical. A detailed proof can be found in section 2.2 of \cite{MR2931648}.

With this proposition, we can build homeomorphisms from a dense type of approximation: 

\begin{lemme}\label{lemmetrans}
Let $\mathcal U$ be an open subset of $\Hom(X,\lambda)$ and $\tT=(\tT_N)_{N\in\N}$ be a type of approximation which is dense in $\mathcal U$. Then for all $N_0\in\N$, the set of homeomorphisms $f$ such that there exists $N\ge N_0$ such that the discretization $f_N$ belongs to $\tT_N$ is dense in $\mathcal U$.
\end{lemme}

\begin{proof}[Proof of lemma \ref{lemmetrans}] Let $f\in\mathcal U$, $N_0\in\N$ and $\varepsilon>0$. Since $\tT$ is a type of approximation which is dense in $\mathcal U$, there exists $N\ge N_0$ and $\sigma_{N}\in \tT_{N}$ such that $d_{N}(f,\sigma_{N})< \varepsilon$. Let $x_1,\dots,x_{q_{N}}$ be the elements of $E_N$. Then for all $\ell$, $d(f(x_\ell),\sigma_{N}(x_\ell))<\varepsilon$. We modify $\sigma_N$ into a bijection $\sigma'_N : E_N\to X$ whose discretization on $E_N$ equals to $\sigma_N$: set $\sigma'_N(x_1) = \sigma(x_1)$, $\sigma'_N$ is defined by induction by choosing $\sigma'_N(x_\ell)$ such that $\sigma'_N(x_\ell)$ is different from $\sigma'_N(x_i)$ for $i\le\ell$, such that $P_N(\sigma'_N(x_\ell)) = P_N(\sigma_N(x_\ell))$ and that $d(f(x_\ell),\sigma'_N(x_\ell))<\varep$.

Since $\sigma'_N$ is a bijection proposition \ref{extension} can be applied to $f(x_\ell)$ and $\sigma'_{N}(x_\ell)$; this gives a measure-preserving homeomorphism $\varphi$ such that $\varphi(f(x_\ell)) = \sigma'_{N}(x_\ell)$ for all $\ell$ and such that $d(\varphi,\mathrm{Id})< \varepsilon$. Set $f' = \varphi\circ f$, we have $d(f,f')\leq \varepsilon$ and $f'(x_\ell) = \sigma'_{N}(x_\ell)$ for all $\ell$, and therefore $f_N = \sigma_N$.
\end{proof}

\begin{proof}[Proof of theorem \ref{génécycl}]
Let $(x_{{N},\ell})_{1\le \ell\le q_{N}}$ be the elements of $E_{N}$ and consider the set (where $\rho_{N}$ is the minimal distance between two distinct points of $E_{N}$)
\[\bigcap_{N_0\in\N}\,\bigcup_{\substack{N\ge N_0\\\sigma_{N}\in \tT_{N}}}\,\left\{f\in\mathcal U \mid\forall \ell,\, d_{N}\big(f(x_{{N},\ell}),\sigma_{N}(x_{{N},\ell})\big)<\frac{\rho_{N}}{2}\right\}.\]
This set is clearly a $G_\delta$ set and its density follows directly from lemma \ref{lemmetrans}. Moreover we can easily see that its elements satisfy the conclusions of the theorem.
\end{proof}

\section{Lax's theorem}

Now that we have shown theorem \ref{génécycl}, we have to obtain dense types of approximation. Again, a theorem will be systematically used: Lax's theorem, and more precisely its improvement stated by S. Alpern \cite{Alpe-newp} which allows us to approach any homeomorphism by a cyclic permutation of the elements of a discretization grid.

\begin{theoreme}[Lax, Alpern]\label{Lax}
Let $f\in\mathrm{Homeo}(X,\lambda)$ and $\varepsilon>0$. Then there exists $N_0\in\N$ such that for all $N\ge N_0$, there exists a cyclic permutation $\sigma_N$ of $E_N$ such that $d(f,\sigma_N)<\varepsilon$.
\end{theoreme}

As a compact metric set can be seen as a ``finite set up to $\varep$'', Lax's theorem allows us to see an homeomorphism as a ``cyclic permutation up to $\varep$''. In practice this theorem is used to ``break'' homeomorphisms into small pieces. Genericity of transitivity in $\Hom(X,\lambda)$ follows easily from Lax's theorem together with finite map extension proposition (again, see \cite{Alpe-typi} or part 2.4 of \cite{MR2931648}). In our case, applying theorem \ref{génécycl}, we deduce that infinitely many discretizations of a generic homeomorphism are cyclic permutations. The purpose of section \ref{partie 1.3} is to find variations of Lax's theorem (which are at the same time corollaries of it) concerning other properties of discretizations.

We give briefly the beautiful proof of Lax's theorem, which is essentially combinatorial and based on marriage lemma and on a lemma of approximation of permutations by cyclic permutations due to S. Alpern. Readers wishing to find proofs of these lemmas may consult section 2.1 of \cite{MR2931648}. This is here that we use the fact that discretizations are well distributed and well ordered.

\begin{lemme}[Marriage lemma]\label{mariage}
Let $E$ and $F$ be two finite sets and $\approx$ a relation between elements of $E$ and $F$. Suppose that the number of any subset $E'\subset E$ is smaller than the number of elements in $F$ that are associated with an element of $E'$, i.e.:
\[\forall E'\subset E,\quad \card (E')\leq \card\left\{f\in F\mid \exists e \in E' : e\approx f \right\},\]
then there exists an injection $\Phi : E\to F$ such that for all $e\in E$, $e\approx \Phi(e)$.
\end{lemme}

\begin{lemme}[Cyclic approximations in $\Sn_q$, \cite{Alpe-newp}]\label{Pioure}
Let $q\in\N^*$ and $\sigma \in \mathfrak{S}_q$ ($\Sn_q$ is seen as the permutation group of $\Z/q\Z$). Then there exists $\tau \in \mathfrak{S}_q$ such that $|\tau(k)-k|\leq 2$ for all $k$ (where $|.|$ is the distance in $\Z/q\Z$) and such that the permutation $\tau\sigma$ is cyclic.
\end{lemme}

\begin{proof}[Proof of theorem \ref{Lax}]
Let $f\in\mathrm{Homeo}(X,\lambda)$ and $\varep>0$. Consider $N_0\in\N$ such that for all $N\ge N_0$, the diameter of the cubes of order $N$ (given by the hypothesis ``being well distributed'') and their images by $f$ is smaller than $\varepsilon$. We define a relation $\approx$ between cubes of order $N-1$: $C\approx C'$ if and only if $f(C)\cap C'\neq\emptyset$. Since $f$ preserves $\lambda$, the image of the union of $\ell$ cubes intersects at least $\ell$ cubes, so marriage lemma (lemma \ref{mariage}) applies: there exists an injection $\Phi_N$  from the set of cubes of order $N$ into itself (then a bijection) such that for all cube $C$, $f(C)\cap\Phi_N(C)\neq\emptyset$. Let $\sigma_N$ be the application that maps the center of any cube $C$ to the center of the cube $\Phi_N(C)$, we obtain:
\[d_N(f,\sigma_N) \leq \sup_{C}\big(\operatorname{diam}(C)\big)+\sup_{C}\big(\operatorname{diam}(f(C))\big)\le 2\varepsilon.\]

It remains to show that $\sigma_N$ can be chosen as a cyclic permutation. Increasing $N$ if necessary, adn using the hypothesis that the grids are well ordered, we number the cubes such that the diameter of the union of two consecutive cubes is smaller than $\varep$. Then we use lemma \ref{Pioure} to obtain a cyclic permutation $\sigma_N'$ whose distance to $\sigma_N$ is smaller than $\varep$. Thus we have found $N_0\in\N$ and for all $N\ge N_0$ a cyclic permutation $\sigma_N'$ of $E_N$ whose distance to $f$ is smaller than $3\varepsilon$.
\end{proof}

\section{Individual behaviour of discretizations}\label{partie 1.3}

Once we have set theorems \ref{génécycl} and \ref{Lax}, we can establish results concerning the behaviour of discretizations of a generic conservative homeomorphism. Here we study \emph{individual} behaviour of discretizations, i.e. properties about only one order of discretization. As has already been said, applying theorem \ref{génécycl}, it suffices to find dense types of approximation to obtain properties about discretizations. In practice, these dense types of approximations are obtained from variations of Lax's theorem (theorem \ref{Lax}).

Recall that the sequence $(E_N)_{n\in\N}$ of discretization grids is well distributed and well ordered (see definition \ref{Ashe}), we denote by $f_N$ the discretization of an homeomorphism $f$ and $\Omega(f_N) $ the maximal invariant set of $f_N$ (i.e. the union of periodic orbits of $f_N$).

We will try to show that for simple dynamical properties $(P)$ about finite maps and for a generic conservative homeomorphism $f$, infinitely many discretizations $f_N$ verify $(P)$ as well as infinitely many discretizations verify its contrary. For instance, for a generic homeomorphism $f$, the maximal invariant set $\Omega(f_N)$ is sometimes as large as possible, i.e. $\Omega(f_N) = E_N$ (corollary \ref{typlax}), sometimes very small (corollary \ref{corovar1}) and even better sometimes the number of elements of the image of $E_N$ is small (corollary \ref{crush}). In the same way stabilization time\footnote{I.e. the smallest integer $k$ such that $f_N^k(E_N) = \Omega(f_N)$.} is sometimes zero (corollary \ref{typlax} for example), sometimes around $\card (E_N)$ (corollary \ref{corovar2}). Finally, concerning the dynamics of ${f_N}_{|\Omega(f_N)}$, sometimes it is a cyclic permutation (corollary \ref{typlax}) or a bicyclic permutation (remark \ref{VTT}, see also corollary \ref{méldiscr}), sometimes it has many orbits (corollary \ref{corovar3}).
\bigskip

Firstly, we deduce directly from Lax's theorem that cyclic permutations of sets $E_N$ form a dense type of approximation in $\Hom (X,\lambda)$. Combining this with theorem \ref{génécycl}, we obtain:

\begin{coro}[Miernowski, \cite{Mier-dyna}]\label{typlax}
For a generic homeomorphism $f\in\Hom(X,\lambda)$, for every $N_0\in\N$, there exists $N\ge N_0$ such that $f_{N}$ is a cyclic permutation\footnote{In fact, T. Miernowski proves ``permutation'' but his arguments, combined with the version of the Lax's theorem we gave, show ``cyclic permutation''.}.
\end{coro}

\begin{rem}\label{VTT}
The same result is obtained for with bicyclic permutations, which are permutations having exactly two orbits whose lengths are relatively prime (see \cite{MR2931648}).
\end{rem}

The first variation of Lax's theorem concerns the approximation of applications which maximal invariant set is small. It is the contrary to what happens in corollary \ref{typlax}.

\begin{prop}[First variation of Lax's theorem]\label{var1}
Let $f\in\Hom(X,\lambda)$. Then for all $\varepsilon, \varep'>0$, there exists $N_0\in\N$ such that for all $N\ge N_0$, there exists a map $\sigma_{N} : E_{N}\to E_{N}$ such that $d_{N}(f,\sigma_{N})<\varepsilon$ and
\[\frac{\card(\Omega(\sigma_N))}{\card(E_{N})} = \frac{\card(\Omega(\sigma_N))}{q_{N}}<\varepsilon',\]
and such that $E_N$ is made of a unique (pre-periodic) orbit of $\sigma_N$.
\end{prop}

\begin{proof}[Proof of proposition \ref{var1}]
Let $f\in\Hom(X,\lambda)$, $\varepsilon>0$ and $x$ a recurrent point of $f$. There exists $\tau\in\N^*$ such that $d(x,f^\tau(x))<\frac{\varepsilon}{8}$; this inequality remains true for fine enough discretizations: there exists $N_1\in\N$ such that if $N\ge N_1$, then
\[d(x,x_{N})<\frac{\varep}{8},\quad d\big(f^\tau(x),f^\tau(x_{N})\big)<\frac{\varepsilon}{8}\quad \text{and}\quad \frac{\tau}{q_{N}}<\varepsilon'.\]
Using the modulus of continuity of $f^\tau$ and Lax's theorem (theorem \ref{Lax}), we obtain an integer $N_0\ge N_1$ such that for all $N\ge N_0$, there exists a cyclic permutation $\sigma_{N}$ of $E_{N}$ such that $d_{N}(f,\sigma_{N})<\frac{\varepsilon}{2}$ and $d_{N}(f^\tau,\sigma_{N}^\tau)<\frac{\varepsilon}{8}$. Then
\begin{eqnarray*}
d(x_{N},\sigma_N^\tau(x_{N})) & \le & d(x_{N},x) + d\big(x,f^\tau(x)\big) + d\big(f^\tau(x),f^\tau(x_{N})\big)\\
                          &     & + d\big(f^\tau(x_{N}),\sigma_N^\tau(x_{N})\big)\\
                          & <   & \frac{\varepsilon}{2}.
\end{eqnarray*}
We compose $\sigma_{N}$ by the (non bijective) application mapping $\sigma_{N}^{\tau-1}(x_{N})$ on $x_{N}$ and being identity anywhere else (see figure \ref{trajectoire}), in other words we consider the application
\[\sigma'_{N}(x) = \left\{\begin{array}{ll}
x_{N} \quad & \text{if}\ x=\sigma_{N}^{\tau-1}(x_{N})\\
\sigma_{N}(x) \quad        & \text{otherwise.}
\end{array}\right.\]
The map $\sigma'_{N}$ has a unique injective orbit whose associated periodic orbit $\Omega(\sigma'_N)$ has length $\tau$ (it is $(x_{N},\sigma_N(x_{N}),\dots,\sigma_{N}^{\tau-1}(x_{N}))$). Since $d(f,\sigma'_N)<\varep$, the map $\sigma'_{N}$ verifies the conclusion of the proposition.
\end{proof}

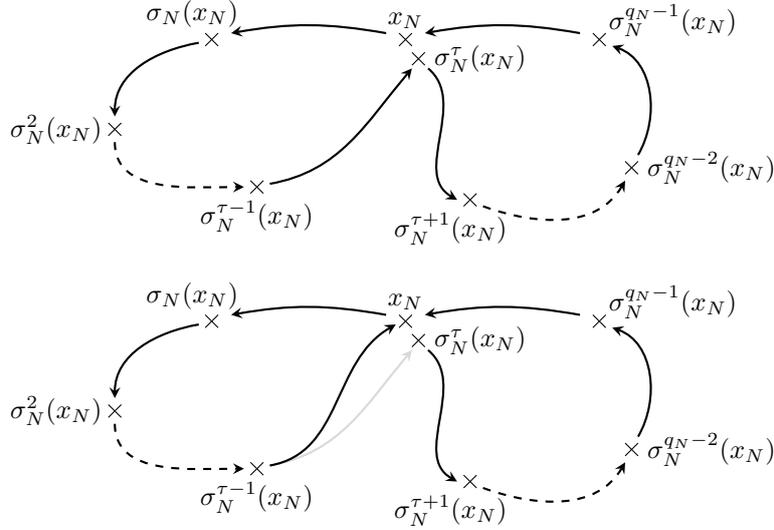
\begin{figure}
\begin{center}
\begin{tikzpicture}[scale=.85]
\draw (-.1,-.1) -- (.1,.1); \draw (-.1,.1) -- (.1,-.1); \draw (0,.3) node{$x_N$};
\draw (-3.1,-.1) -- (-2.9,.1);\draw (-3.1,.1) -- (-2.9,-.1);\draw (-3.3,.4) node{$\sigma_N(x_N)$};
\draw (-4.6,-1.5) -- (-4.4,-1.3);\draw (-4.6,-1.3) -- (-4.4,-1.5);\draw (-4.55,-1.4) node[left]{$\sigma_N^2(x_N)$};
\draw (-2.4,-2.4) -- (-2.2,-2.2);\draw (-2.4,-2.2) -- (-2.2,-2.4);\draw (-2.3,-2.75) node{$\sigma_N^{\tau-1}(x_N)$};
\draw (.1,-.4) -- (.3,-.2); \draw (.1,-.2) -- (.3,-.4);\draw (.3,-.3) node[right]{$\sigma_N^{\tau}(x_N)$};
\draw (.9,-2.4) -- (1.1,-2.6); \draw (.9,-2.6) -- (1.1,-2.4);\draw (.7,-2.95) node{$\sigma_N^{\tau+1}(x_N)$};
\draw (3.4,-1.9) -- (3.6,-2.1); \draw (3.4,-2.1) -- (3.6,-1.9);\draw (3.6,-2) node[right]{$\sigma_N^{q_N-2}(x_N)$};
\draw (3.1,-.1) -- (2.9,.1);\draw (3.1,.1) -- (2.9,-.1);\draw (3,.3) node[right]{$\sigma_N^{q_N-1}(x_N)$};
\draw[->,>=stealth,thick] (-.3,.1) to [out=170,in=10] (-2.7,.1);
\draw[->,>=stealth,thick] (-3.2,-.05) to [out=190,in=90] (-4.5,-1.2);
\draw[->,>=stealth,thick,dashed] (-4.5,-1.6) to [out=-90,in=180] (-2.5,-2.3);
\draw[->,>=stealth,thick] (-2.1,-2.25) to [out=10,in=-130] (.1,-.45);
\draw[->,>=stealth,thick] (.35,-.45) to [out=-40,in=150] (.8,-2.5);
\draw[->,>=stealth,thick,dashed] (1.2,-2.6) to [out=-20,in=-120] (3.4,-2.2);
\draw[->,>=stealth,thick] (3.6,-1.8) to [out=60,in=-20] (3.2,-.1);
\draw[->,>=stealth,thick] (2.7,.1) to [out=170,in=10] (.3,.1);
\draw (0,-3.6);
\end{tikzpicture}

\begin{tikzpicture}[scale=.85]
\draw (-.1,-.1) -- (.1,.1); \draw (-.1,.1) -- (.1,-.1); \draw (0,.3) node{$x_N$};
\draw (-3.1,-.1) -- (-2.9,.1);\draw (-3.1,.1) -- (-2.9,-.1);\draw (-3.3,.4) node{$\sigma_N(x_N)$};
\draw (-4.6,-1.5) -- (-4.4,-1.3);\draw (-4.6,-1.3) -- (-4.4,-1.5);\draw (-4.55,-1.4) node[left]{$\sigma_N^2(x_N)$};
\draw (-2.4,-2.4) -- (-2.2,-2.2);\draw (-2.4,-2.2) -- (-2.2,-2.4);\draw (-2.3,-2.75) node{$\sigma_N^{\tau-1}(x_N)$};
\draw (.1,-.4) -- (.3,-.2); \draw (.1,-.2) -- (.3,-.4);\draw (.3,-.3) node[right]{$\sigma_N^{\tau}(x_N)$};
\draw (.9,-2.4) -- (1.1,-2.6); \draw (.9,-2.6) -- (1.1,-2.4);\draw (.7,-2.95) node{$\sigma_N^{\tau+1}(x_N)$};
\draw (3.4,-1.9) -- (3.6,-2.1); \draw (3.4,-2.1) -- (3.6,-1.9);\draw (3.6,-2) node[right]{$\sigma_N^{q_N-2}(x_N)$};
\draw (3.1,-.1) -- (2.9,.1);\draw (3.1,.1) -- (2.9,-.1);\draw (3,.3) node[right]{$\sigma_N^{q_N-1}(x_N)$};
\draw[->,>=stealth,thick] (-.3,.1) to [out=170,in=10] (-2.7,.1);
\draw[->,>=stealth,thick] (-3.2,-.05) to [out=190,in=90] (-4.5,-1.2);
\draw[->,>=stealth,thick,dashed] (-4.5,-1.6) to [out=-90,in=180] (-2.5,-2.3);
\draw[->,>=stealth,thick,color=gray!30] (-2.1,-2.25) to [out=10,in=-130] (.1,-.45);
\draw[->,>=stealth,thick] (-2.1,-2.25) to [out=10,in=-150] (-.15,-.05);
\draw[->,>=stealth,thick] (.35,-.45) to [out=-40,in=150] (.8,-2.5);
\draw[->,>=stealth,thick,dashed] (1.2,-2.6) to [out=-20,in=-120] (3.4,-2.2);
\draw[->,>=stealth,thick] (3.6,-1.8) to [out=60,in=-20] (3.2,-.1);
\draw[->,>=stealth,thick] (2.7,.1) to [out=170,in=10] (.3,.1);
\end{tikzpicture}

\caption{Modification of a cyclic permutation}\label{trajectoire}
\end{center}
\end{figure}

A direct application of theorem \ref{génécycl} gives:

\begin{coro}\label{corovar1}
For a generic homeomorphism $f\in\Hom(X,\lambda)$,
\[\underline\lim_{N\to +\infty}\frac{\card(\Omega(f_N))}{\card(E_{N})} =0.\]
Specifically for all $\varepsilon>0$ and all $N_0\in\N$, there exists $N\ge N_0$ such that $\frac{\card(\Omega(f_N))}{\card(E_{N})} <\varepsilon$ and such that $E_N$ is made of a unique (pre-periodic) orbit of $f_N$.
\end{coro}

The same kind of idea leads to the following proposition:

\begin{prop}\label{propvar2}
Let $f\in\Hom(X,\lambda)$ having at least one periodic point of period $p$. Then for all $\varepsilon>0$, there exists $N_0\in\N$ such that for all $N\ge N_0$, there exists an application $\sigma_{N} : E_{N}\to E_{N}$ with $\card(\Omega(\sigma_N)) = p$ and $d_{N}(f,\sigma_{N})<\varepsilon$, such that $E_N$ is made of a unique (pre-periodic) orbit of $\sigma_N$.
\end{prop}

\begin{proof}[Proof of proposition \ref{propvar2}]
Simply replace the recurrent point by a periodic point of period $p$ in the proof of proposition \ref{var1}.
\end{proof}

We will show that owning a periodic point is a generic property in $\Hom(X,\lambda)$ (see also \cite{Daal-chao} or part 3.2 of \cite{MR2931648}), a proper application of Baire's theorem leads to a refinement of corollary \ref{corovar1}:

\begin{coro}\label{corovar2}
Let $\varep>0$. A generic homeomorphism $f\in\Hom(X,\lambda)$ has a periodic point whose orbit is $\varep$-dense. Moreover, $f_{N}$ has a unique periodic orbit for infinitely many integers $N$, whose period does not depend on $N\in\N$ and equals the smallest period of periodic points of $f$ which are $\varep$-dense; moreover $E_N$ is covered by a single (pre-periodic) orbit of $f_N$.
\end{coro}

To prove corollary \ref{corovar2} we have to introduce the concept of \emph{persistent point}.

\begin{definition}
Let $f\in\Hom(X)$. A periodic point $x$ of $f$ with period $p$ is said \emph{persistent} if for all neighborhood $U$ of $x$, there exists a neighborhood $\V$ of $f$ in $\Hom(X)$ such that every $\widetilde{f}\in \V$ has a periodic point $\widetilde{x}\in U$ with period $p$.
\end{definition}

\begin{ex}\label{linéaire}
The endomorphism $h = \operatorname{Diag}(\lambda_1,\dots,\lambda_n)$ of $\R^n$, with $\prod \lambda_i = 1$ and $\lambda_i\neq 1$ for all $i$, is measure-preserving and has a persistent fixed point at the origin (see e.g. \cite{Kato-intr}, p. 319). Let $s$ be a reflection of $\R^n$, the application $h\circ f$ is also measure-preserving and has a persistent fixed point at the origin.
\end{ex}

Finally, we use the theorem of local modification of conservative homeomorphisms, which allows us to replace locally a homeomorphism by another. Although it is geometrically ``obvious'' and it has an elementary proof in dimension two, it is easily deduced from the (difficult) \emph{annulus theorem}. For more details one may refer to \cite{Daal-chao} or part 3.1 of \cite{MR2931648}.

\begin{theoreme}[Local modification]\label{extension-sphères}
Let $\sigma_1$, $\sigma_2$, $\tau_1$ and $\tau_2$ be four bicollared embeddings\footnote{An embedding $i$ of a manifold $M$ in $\R^n$ is said \emph{bicollared} if there exists an embedding  $j :M\times [-1,1] \to \R^n$ such that $j_{M\times \{0\}} = i$.} of $\mathbf{S}^{n-1}$ in $\R^n$, such that $\sigma_1$ is in the bounded connected component\footnote{By the Jordan-Brouwer theorem the complement of a set which is homeomorphic to $\mathbf{S}^{n-1}$ has exactly two connected components: one bounded and one unbounded.} of $\sigma_2$ and $\tau_1$ is in the bounded connected component of $\tau_2$. Let $A_1$ be the bounded connected component of $\R^n\setminus\sigma_1$ and $B_1$ the bounded connected component of $\R^n\setminus\tau_1$; $\Sigma$ be the connected component of $\R^n\setminus(\sigma_1\cup\sigma_2)$ having $\sigma_1 \cup \sigma_2$ as boundary and $\Lambda$ the connected component of $\R^n\setminus(\tau_1\cup\tau_2)$ having $\tau_1 \cup \tau_2$ as boundary; $A_2$ be the unbounded connected component of $\R^n\setminus\sigma_2$ and $B_2$ the unbounded connected component of $\R^n\setminus\tau_2$ (see figure \ref{dessin-extension}).

Suppose that $\mathrm{Leb}(A_1) = \mathrm{Leb}(B_1)$ and $\mathrm{Leb}(\Sigma) = \mathrm{Leb}(\Lambda)$. Let $f_i : A_i\to B_i$ be two measure-preserving homeomorphisms such that either each one preserves the orientation, or each one reverses it. Then there exists a measure-preserving homeomorphism $f : \R^n \to \R^n$ whose restriction to $A_1$ equals $f_1$ and restriction to $A_2$ equals $f_2$.
\end{theoreme}

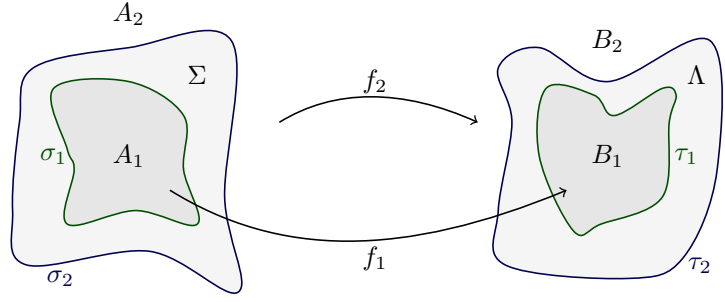
\begin{figure}
\begin{center}
\begin{tikzpicture}[scale=.9]
\draw[color=blue!30!black, fill=gray!8!white, semithick] plot[tension=0.6, smooth cycle] coordinates{(-1.4,1.3) (0,1.6) (1.5,1.7) (1.4,-.6) (1.6,-2) (0.3,-1.4) (-1.5,-1.6) (-1.7,-.8) (-1.6,.5)};
\draw[color=green!30!black, fill=gray!20!white, semithick] plot[tension=0.6, smooth cycle] coordinates{(-1.1,1) (0,1.1) (.8,.6) (.8,-.2) (1,-1) (0.1,-.8) (-.9,-1) (-.8,-.2) (-.9,.1)};
\draw[color=blue!30!black, fill=gray!8!white, semithick] plot[tension=0.6, smooth cycle] coordinates{(5.4,1.3) (6,1.6) (7,1.1) (8.5,1.7) (8.6,-.2) (8.1,-1.5) (7.3,-1.8) (5.5,-1.6) (5.4,-.8) (5.6,.5)};
\draw[color=green!30!black, fill=gray!20!white, semithick] plot[tension=0.6, smooth cycle] coordinates{((6.1,1) (6.8,.9) (7.1,.6) (7.8,1) (8,.9) (7.9,0.5) (7.8,-.6) (6.9,-1) (6.5,-1.1) (6,0)};
\draw[->, semithick] (.6,-.5) .. controls (2.2,-1.5) and (4,-1.5) .. (6.4,-.5);
\draw[->, semithick] (2.2,.5) .. controls (3,1) and (4,1) .. (5.1,.5);
\node at (0,0) {$A_1$};
\node[color=green!30!black] at (-1.1,0) {$\sigma_1$};
\node[color=blue!30!black] at (-1,-1.8) {$\sigma_2$};
\node at (1,1.2) {$\Sigma$};
\node at (0,2.1) {$A_2$};
\node at (3.6,-1.5) {$f_1$};
\node at (3.6,1.1) {$f_2$};
\node at (7,0) {$B_1$};
\node[color=green!30!black] at (8.15,0) {$\tau_1$};
\node[color=blue!30!black] at (8.35,-1.6) {$\tau_2$};
\node at (7,1.7) {$B_2$};
\node at (8.3,1.2) {$\Lambda$};
\end{tikzpicture}\caption{Local modification technique}\label{dessin-extension}
\end{center}
\end{figure}

Since this theorem is local, it can be applied to an open space $O$ instead of $\R^n$, or even better, together with Oxtoby-Ulam theorem (theorem \ref{Brown-mesure}), to any domain of chart of a manifold $X$ instead of $\R^n$ and measure $\lambda$ instead of Lebesgue measure

\begin{proof}[Proof of corollary \ref{corovar2}]
Let $\varep>0$ and $\mathcal{U}_p$ the set of homeomorphisms that have at least one periodic persistent point with period $p$ whose orbit is $\varep$-dense (strictly), but no periodic point with period less than $p$ whose orbit is $\varep$-dense (not strictly). The sets $\mathcal{U}_p$ are open and pairwise disjoints. Moreover their union is dense. Indeed, for $f\in\Hom(X,\lambda)$, if $f$ has at least one periodic point whose orbit is $\varep$-dense, then we choose a periodic point $x$ whose orbit is $\varep$-dense and whose period $p$ is minimal. Then we perturb $f$ such that this point is persistent, avoiding the creation of periodic points of smaller periods: we choose a neighborhood $D$ of $x$ such that the sets $D$, $f(D),\cdots, f^{p-1}(D)$ are pairwise disjoints and we replace locally $f$ by the map $h\circ f^{-(p-1)}$ (where $h$ is one of the two maps of the example \ref{linéaire}, depending of whether $f^{-(p-1)}$ preserves orientation or not) in the neighborhood of $f^{-1}(x)$, using the theorem of local modification (theorem \ref{extension-sphères}), such that $f$ do not change outside the union of the sets $f^i(D)$. On the contrary, if $f$ does not have any periodic point, we pick a recurrent point and perturb $f$ such that this point is periodic, and then we apply the previous case.

Let $p\in\N^*$. Proposition \ref{propvar2} indicates that the set of maps whose longest periodic orbit has length $p$ is a dense type of approximation in $\mathcal{U}_p$. So, by theorem \ref{génécycl}, there exists open sets $\mathcal{O}_{p,{N}}$ such that $\bigcap_{N\in\N} \mathcal{O}_{p,{N}}$ is a dense $G_\delta$ of $\mathcal{U}_p$, made of homeomorphisms of $\mathcal{U}_p$ whose an infinity of discretizations have only one injective orbit whose associated cycle has length $p$ and is $\varep$-dense.

Set $\mathcal{O}_{N} = \bigcup_{p\in\N^*} \mathcal{O}_{p,{N}}$. Since the sets $\mathcal{U}_p$ are pairwise disjoints,
\[\bigcap_{N\in\N}\,\mathcal{O}_{N} = \bigcap_{N\in\N}\,\bigcup_{p\in\N^*}\, \mathcal{O}_{p,{N}} = \bigcup_{p\in\N^*}\,\bigcap_{N\in\N}\,\mathcal{O}_{p,{N}}.\]
The right side of the equation forms a dense subset of $\Hom(X,\lambda)$, so the left side is a dense $G_\delta$ of $\Hom(X,\lambda)$. It is made of homeomorphisms $f$ such that for infinitely many $N\in\N$, the grid $E_N$ is covered by a single forward orbit of $f_N$ whose associated periodic orbit is $\varep$-dense and whose length is the shortest period of periodic points of $f$ whose orbit are $\varep$-dense.
\end{proof}

\begin{rem}
The shortest period of periodic points of generic homeomorphisms has not any global upper bound in $\Hom(X,\lambda)$: for example, for all $p\in\N$, there is a open set of homeomorphisms of the torus without periodic point of period less than $p$ (e.g. the neighborhood of a nontrivial rotation) and this property remains true for discretizations.
\end{rem}

We now want to get a discrete analogue to topological weak mixing (which is generic in $\Hom(X,\lambda)$, see e.g. \cite{Alpe-typi} or part 2.4 of \cite{MR2931648}).

\begin{definition}
A homeomorphism $f$ is said \emph{topologically weakly mixing} if for all nonempty open sets $(U_i)_{i\le M}$ and $(U'_i)_{i\le M}$, there exists $m\in\N$ such that $f^m(U_i)\cap U'_i$ is nonempty for all $i\le M$.
\end{definition}

The proof of the genericity of topological weak mixing starts by an approximation of every conservative homeomorphism by another having $\varep$-dense periodic orbits whose lengths are relatively prime. The end of the proof lies primarily in the use of Baire's theorem and Bezout's identity. In the discrete case, the notion of weak mixing is replaced by the following:

\begin{definition}\label{epmélfaibl}
Let $\varep>0$. A finite map $\sigma_N$ is said \emph{$\varep$-topologically weakly mixing} if for all $M\in\N$ and all balls $(B_i)_{i\le M}$ and $(B'_i)_{i\le M}$ with diameter $\varep$, there exists $m\in\N$ such that for all $i$
\[\sigma_N^m(B_i\cap E_{N})\cap (B'_i\cap E_{N})\neq\emptyset.\]
\end{definition}

The first step of the proof is replaced by the following variation of Lax's theorem:

\begin{prop}[Second variation of Lax's theorem]\label{melfaibl}
Let $f\in\Hom(X,\lambda)$ be a homeomorphism whose all iterates are topologically transitive. Then for all $\varep>0$ and all $M\in\N^*$, there exists $N_0\in\N$ such that for all $N\ge N_0$, there exists $\sigma_{N} : E_{N}\to E_{N}$ which has $M$ $\varep$-dense periodic orbits whose lengths are pairwise relatively prime, and such that $d_{N}(f,\sigma_{N})<\varep$.
\end{prop}

\begin{proof}[Proof of proposition \ref{melfaibl}]
We prove the proposition in the case where $M=2$, the other cases being easily obtained by an induction. Let $\varep>0$ and $f$ be an homeomorphism whose all iterates are topologically transitive. Then there exists $x_0\in X$ and $p\in\N^*$ such that $\{x_0,\dots,f^{p-1}(x_0)\}$ is $\varep$-dense and $d(x_0,f^p(x_0))<\varep/2$. Since transitive points of $f^p$ form a dense $G_\delta$ subset of $X$, while the orbit of $x_0$ form a $F_\sigma$ set with empty interior, the set of points whose orbit under $f^p$ is dense and disjoint from that of $x_0$ is dense. So we can pick such a transitive point $y_0$. Set $y_1 = f(y_0)$. Then there exists a multiple $q_1$ of $p$ such that the orbit $\{y_1,\dots,f^{q_1-1}(y_1)\}$ is $\varep$-dense and $d(y_1,f^{q_1}(y_1))<\varep/2$. Again, by density, we can choose a transitive point $y_2$ whose orbit is disjoint from that of $x_0$ and $y_1$, with $d(y_1,y_2)<\varep/2$ and $d(y_0,y_1)-d(y_0,y_2)>\varep/4$. Then there exists a multiple $q_2$ of $p$ such that $d(y_2,f^{q_2}(y_2))<\varep/2$. And so on, we construct a sequence $(y_m)_{1\le m\le \ell}$ such that (see figure \ref{constrigrec}):
\begin{enumerate}[(i)]
\item for all $m$, there exists $q_m>0$ such that $p| q_m$ and $d(y_m, f^{q_m}(y_m))<\varep/2$,
\item the orbits $\{x_0,\dots,f^{p-1}(x_0)\}$ and $\{y_m,\dots,f^{q_m-1}(y_m)\}$ ($m$ going from $0$ to $\ell-1$) are pairwise disjoints,
\item for all $m$, $d(y_m,y_{m+1})<\varep/2$ and $d(y_0,y_m)-d(y_0,y_{m+1})>\varep/4$,
\item $y_\ell = y_0$.
\end{enumerate}

Let $\sigma_{N}$ be a finite map given by Lax's theorem. For all $N$ large enough, $\sigma_{N}$ satisfies the same properties (i) to (iii) than $f$. Changing $\sigma_{N}$ at the points $\sigma_{N}^{q_m-1}((y_m)_{N})$ and $\sigma_{N}^{p-1}((x_0)_{N})$, we obtain a finite map $\sigma'_{N}$ such that ${\sigma'_{N}}^{q_N}((y_m)_{N}) = (y_{m+1})_{N}$ and $\sigma_{N}^{p}((x_0)_{N}) = (x_0)_{N}$. Thus the orbit of $(x_0)_N$ under $\sigma'_{N}$ is $2\varep$-dense and has period $p$ and the orbit of $(y_0)_{N}$ under $\sigma'_{N}$ is $2\varep$-dense, disjoint from which of $(x_0)_{N}$ and has period $1+q_1+\dots+q_{\ell-1}$ relatively prime to $p$.
\end{proof}

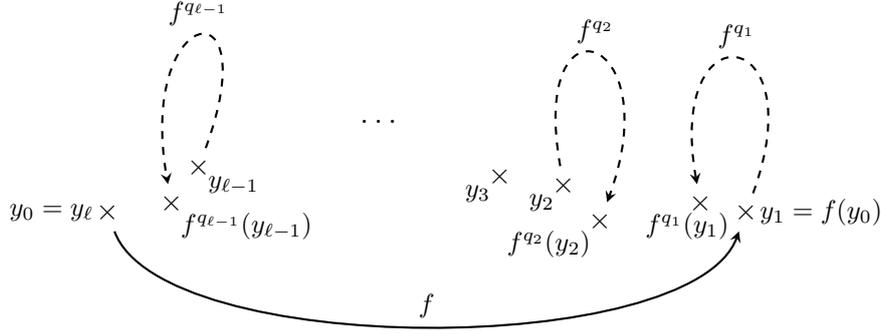
\begin{figure}
\begin{center}
\begin{tikzpicture}[scale=1.2]
\node (A) at (0,0) {\large$\times$};\node at (A) [left]{$y_0=y_\ell\,$};
\node (B) at (7,0) {\large$\times$};\node at (B) [right]{$\,y_1 = f(y_0)$};
\node (C) at (5,.3) {\large$\times$};\node at (C) [below left]{$y_2$};
\node (D) at (4.3,.4) {\large$\times$};\node at (D) [below left]{$y_3$};
\node (F) at (1,.5) {\large$\times$};\node at (F) [below right]{$y_{\ell-1}$};

\node (b) at (6.5,.1){\large$\times$};\node at (b)[below]{$f^{q_1}(y_1)\quad$};
\node (c) at (5.4,-.1){\large$\times$};\node at (c)[below left]{$f^{q_2}(y_2)$};
\node (f) at (.7,.1){\large$\times$};\node at (f)[below right]{$f^{q_{\ell-1}}(y_{\ell-1})$};
\node at (3,1){\Large$\dots$};

\draw[->,>=stealth,thick] (A) to [out=-70,in=-110,distance=1.5cm] node[above]{$f$}(B);
\draw[->,>=stealth,thick,dashed] (B) to [out=70,in=100,distance=2cm] node[above]{$f^{q_1}$}(b);
\draw[->,>=stealth,thick,dashed] (C) to [out=100,in=70,distance=2cm] node[above]{$f^{q_2}$}(c);
\draw[->,>=stealth,thick,dashed] (F) to [out=70,in=100,distance=2cm] node[above]{$f^{q_{\ell-1}}$}(f);

\end{tikzpicture}
\caption{Construction of the sequence $(y_m)_{1\le m\le \ell}$}\label{constrigrec}
\end{center}
\end{figure}

\begin{coro}\label{méldiscr}
For a generic homeomorphism $f\in \Hom(X,\lambda)$, for all $\varep>0$ and all $N_0\in\N$, there exists $N\ge N_0$ such that $f_{N}$ is $\varep$-topologically weakly mixing.
\end{coro}

\begin{proof}[Proof of corollary \ref{méldiscr}]
Again, we prove the corollary in the case where $M=2$, other cases being easily obtained by induction. Let $\varep>0$ and $N_0\in\N$. All iterates of a generic homeomorphism $f$ are topologically transitive: it is an easy consequence of the genericity of transitivity (see e.g. corollary \ref{typlax} or theorem 2.11 of \cite{MR2931648}); we pick such a homeomorphism. Combining theorem \ref{génécycl} and proposition \ref{melfaibl}, we obtain $N\ge N_0$ such that $f_{N}$ has two $\varep/3$-dense periodic orbits whose lengths $p$ et $q$ are coprime. We now have to prove that $f_N$ is $\varep$-topologically weakly mixing. Let $B_1$, $B_2$, $B'_1$ and $B'_2$ be balls with diameter $\varep$. Since each one of these orbits is $\varep/3$-dense, there exists $x_N\in X$ which is in the intersection of the orbit whose length is $p$ and $B_1$, and $y_{N}\in X$ which is in the intersection of the orbit whose length is $q$ and $B_2$. Similarly, there exists two integers $a$ and $b$ such that $f_{N}^a(x_{N})\in B'_1$ and $f_{N}^b(y_{N})\in B'_2$.

Recall that we want to find a power of $f_{N}$ which sends both $x_{N}$ in $B'_1$ and $y_{N}$ in $B'_2$. It suffices to pick $m\in\N$ such that $m = a+\alpha p = b+\beta q$. Bezout's identity states that there exists two integers $\alpha$ and $\beta$ such $\alpha p-\beta q = b-a$. Set $m = a+\alpha p$, adding a multiple of $pq$ if necessary, we can suppose that $m$ is positive. Thus $f_{N}^m(x_{N})\in B'_1$ and $f_{N}^m(y_{N})\in B'_2$.
\end{proof}

We now establish a new variation of Lax's theorem. It asserts that every homeomorphism can be approximated by a finite map whose image has a small cardinality, unlike what happens in Lax's theorem.

\begin{prop}[Third variation of Lax's theorem]\label{crunch}
Let $f\in\Hom(X,\lambda)$ and $\vartheta : \N\to\R_+^*$ a map which tends to $+\infty$ at $+\infty$. Then for all $\varep>0$, there exists $N_0\in\N$ such that for all $N\ge N_0$, there exists a map $\sigma_{N} : E_{N}\to E_{N}$ such that $\card(\sigma_{N}(E_N))<\vartheta({N})$ and $d_{N}(f,\sigma_{N})<\varep$.
\end{prop}

\begin{proof}[Proof of proposition \ref{crunch}]
Let $f\in\Hom(X,\lambda)$, $\vartheta : \N\to\R_+^*$ a map which tends to $+\infty$ at $+\infty$ and $\varep>0$. By Lax's theorem (theorem \ref{Lax}) there exists $N_1\in\N$ such that for all $N\ge N_1$, there exists a cyclic permutation $\sigma_{N} : E_{N}\to E_{N}$ whose distance to $f$ is smaller than $\varep/2$. For $N\ge N_1$, set $\sigma'_{N} = P_{N_1}\circ \sigma_{N}$. Increasing $N_1$ if necessary we have $d(f,\sigma'_{N})<\varep$, regardless of $N$. Moreover $\card(\sigma'_{N}(E_{N}))\le q_{N_1}$, if we choose $N_0$ large enough such that for all $N\ge N_0$ we have $q_{N_1}<\vartheta({N})$, then $\card(\sigma'_{N}(E_{N}))\le \vartheta({N})$. We have shown that the map $\sigma'_{N}$ satisfies the conclusions of proposition for all $N\ge N_0$.
\end{proof}

\begin{coro}\label{crush}
Let $\vartheta : \N\to\R_+^*$ a map which tends to $+\infty$ at $+\infty$. Then for a generic homeomorphism $f\in\Hom(X,\lambda)$, 
\[\underset{N\to+\infty}{\underline\lim}\ \frac{\card(f_{N}(E_{N}))}{\vartheta({N})}=0.\]
In particular, generically, $\underline{\lim}\frac{\card(f_{N}(E_{N}))}{\card(E_{N})}=0$.
\end{coro}

\begin{proof}[Proof of corollary \ref{crush}]
Remark that if we replace $\vartheta({N})$ by $\sqrt{\vartheta({N})}$, it suffices to prove that for a generic homeomorphism, $\underline\lim\frac{\card(f_{N}(E_{N}))}{\vartheta({N})}\le 1$. It is easily obtained in combining theorem \ref{génécycl} and proposition \ref{crunch}.
\end{proof}

So far all variations of Lax's theorem built finite maps with a small number of orbits. With the additional assumption that the sequence of grids is self-similar, we show a last variation of Lax's theorem approaching every homeomorphism with a finite map with a large number of orbits.

\begin{prop}[Fourth variation of Lax's theorem]\label{killing}
Assume that the sequence of grids $(E_N)_{N\in\N}$ is self-similar. Let $f\in\Hom(X,\lambda)$ and $\vartheta : \N\to\R$ such that $\vartheta(N) = o(q_N)$. Then for all $\varepsilon>0$ there exists $N_1\in\N$ such that for all $N\ge N_1$, there exists a permutation $\sigma_{N}$ of $E_{N}$ such that $d_{N}(f,\sigma_{N})<\varepsilon$ and that the number of cycles of $\sigma_{N}$ is greater than $\vartheta({N})$. Moreover these $\vartheta({N})$ cycles of $\sigma_{N}$ are conjugated to a cyclic permutation of $E_{N_0}$ by bijections whose distance to identity is smaller than $\varep$.
\end{prop}

\begin{proof}[Proof of proposition \ref{killing}]
Let $\varep>0$, for all $N_0\in\N$ large enough, Lax's theorem gives us a cyclic permutation $\sigma'_{N_0}$ of $E_{N_0}$ whose distance to $f$ is smaller than $\varep$. Since the grids are self-similar, there exists $N_1\in\N$ such that for all $N\ge N_1$, the set $E_N$ contains $q_N/q_{N_0}\ge\vartheta(N)$ disjoint subsets $\widetilde E_N^j$, each one conjugated to a grid $E_{N_0}$ by a bijection $h_j$ whose distance to identity is smaller than $\varep$. On each $\widetilde E_N^j$, we define $\sigma_N$ as the conjugation of $\sigma'_{N_0}$ by $h_j$; outside these sets we just pick $\sigma_N$ such that $d_N(f,\sigma_N)<\varep$. Since the distance between $h_j$ and identity is smaller than $\varep$, we have $d_N(f,\sigma_N)<2\varep$. Moreover, $\sigma_N$ has at least $\vartheta(N)$ cycles; this completes the proof.
\end{proof}

The application of theorem \ref{génécycl} gives us:

\begin{coro}\label{corovar3}
We still assume that the sequence of grids $(E_N)_{N\in\N}$ is self-similar. Let $\vartheta : \N\to\R$ such that $\vartheta(N) = o(q_N)$. Then for a generic homeomorphism $f\in\Hom(X,\lambda)$ and for infinitely many integers $N$, the discretization $f_{N}$ of $f$ has at least $\vartheta(N)$ cycles which are pairwise conjugated.
\end{coro}

\section{Average behaviour of discretizations}\label{bofbof}

We now want to study the average behaviour of discretizations of a generic homeomorphism. For example one could imagine that even if for a generic homeomorphism $f$, the event ``$f_N$ is a cyclic permutation'' appears for infinitely many orders, it is still quite rare. More precisely, we study the frequency of occurrence of properties related to the discretizations of generic homeomorphisms in the Cesàro sense: given a property $(P)$ concerning discretizations, what is the behaviour of the proportion between $1$ and $M$ of discretizations satisfying the property $(P)$, when $ M $ goes to infinity? For this study, we assume that the sequence of discretization grids refines (which is true for example for discretizations upon uniform grids of orders powers of an integer, see section \ref{exgrilles}). This prevents us from tricky arithmetic problems about overlay of grids.

\begin{definition}
Let $f\in\Hom(X,\lambda)$. We say that a property $(P)$ about discretizations is \emph{satisfied in average} if for all $N_0\in\N$ and all $\varep>0$, there exists $N\ge N_0$ such that the proportion of integers $M\in \{ 0,\dots,N\}$ such that $f_M$ satisfies $(P)$ is greater than $1-\varep$, i.e. 
\[\underset{N\to+\infty}{\overline \lim}\,\frac{1}{N+1}\card\big\{M\in\{ 0,\dots,N\}\mid f_M\text{ satisfies }(P)\big\}=1.\]
\end{definition}

We will show that most of the dynamical properties studied in the previous section are actually satisfied on average for generic homeomorphisms. To start with we set out a technical lemma:

\begin{lemme}\label{recopie}
Let $\tT$ be a dense type of approximation in $\Hom(X,\lambda)$. Then for a generic homeomorphism $f\in\Hom(X,\lambda)$, for all $\varep>0$ and all $\alpha>0$, the property $(P)$ : ``$E_N$ contains at least $\alpha$ disjoints subsets which fulfils a proportion greater than $1-\varep$ of $E_N$, each one stabilized by $f_N$ and such that the restriction of $f_N$ to each one is conjugated to a map of $\tT$ by a bijection whose distance to identity is smaller than $\varep$'' is satisfied in average.
\end{lemme}

In practice, this lemma provides many properties satisfied on average, for instance:
\begin{itemize}
\item quantitative properties on discretizations, such as owning at least $M$ periodic orbits,
\item properties of existence of sub-dynamics on discretizations, such as owning at least one dense periodic orbit.
\end{itemize}

\begin{proof}[Proof of lemma \ref{recopie}]
Let us consider the set
\[\mathcal{C} = \bigcap_{\substack{\varepsilon>0\\ N_0\in\N}}\,\bigcup_{N\ge N_0}\,\left\{
\begin{array}{c}
f\in\Hom(X,\lambda)\ \big|\\ \frac{1}{N+1}\card\big\{M\in\{ 0,\dots,N\}\mid f_M\text{ satisfies }(P)\big\}>1-\varep
\end{array}\right\}.\]
We want to show that $\mathcal{C}$ contains a dense $G_\delta$ of $\Hom(X,\lambda)$. The set $\mathcal{C}$ is a $G_\delta$ of the generic set  $\bigcap_{N\in\N} \mathcal{D}_{N}$, it suffices to prove that it is dense in $\bigcap_N \mathcal{D}_{N}$. Let $f\in \Hom(X,\lambda)$, $N_0\in\N$, $\delta>0$ and $\varepsilon>0$. To prove it we want to find a homeomorphism $g$ whose distance to $f$ is smaller than $\delta$ and an integer $N\ge N_0$ such that
\[\frac{1}{N+1}\card\big\{M\in\{ 0,\dots,N\}\mid g_M\text{ satisfies }(P)\big\}>1-\varep.\]
It is simply obtained in combining the density of the type of approximation $\tT$ and the fact that the grids refines.
\end{proof}

This corollary allows us to obtain some properties about the average behaviour of discretizations. For instance there is an improvement of corollary \ref{corovar2}:

\begin{coro}
For a generic homeomorphism $f\in\Hom(X,\lambda)$, the property ``$f_{N}$ has a $\varep$-dense periodic orbit and the cardinality of $\Omega(f_N)$ is smaller than $\vartheta(N) = o(q_N)$'' is satisfied in average.
\end{coro}

Or an improvement of corollary \ref{méldiscr}:

\begin{coro}
For a generic homeomorphism $f\in\Hom(X,\lambda)$ and for all $\varep>0$, the property ``$f_{N}$ is $\varep$-topologically weakly mixing (see definition \ref{epmélfaibl})'' is satisfied in average.
\end{coro}

Or even an improvement of corollary \ref{crush}:

\begin{coro}\label{petitepermieux}
For a generic homeomorphism $f\in\Hom(X,\lambda)$ and for all $\varep>0$, the property ``$\frac{\card(f_{N}(E_{N}))}{\card(E_{N})}<\varep$'' is satisfied in average.
\end{coro}

And an improvement of corollary \ref{corovar3}:

\begin{coro}
For a generic homeomorphism $f\in\Hom(X,\lambda)$ and for all $M\in\N$, property ``$f_N$ has at least $M$ periodic orbits'' is satisfied in average.
\end{coro}

However, note that the most simple property about discretizations, i.e. being a cyclic permutation, can not be proved by using lemma \ref{recopie}. To do this, we need a slightly more precise result, that we will not prove here. %See \cite{???}?? for a complete proof.

\begin{prop}\label{propdemin}
For a generic homeomorphism $f\in\Hom (X,\lambda)$, the property ``$f_{N}$ is a cyclic permutation'' is satisfied in average.
\end{prop}

\begin{rem}
However, the property of approximation by bicyclic permutations in average can not be proven with this technique.
\end{rem}

\begin{rem}
A simple calculation shows that everything that has been done in this section also applies to the behaviour of discretizations in average of Cesàro average, in average of average of Cesàro average etc., i.e. when studying quantities
\[\frac{1}{N_2+1}\sum_{N_1=0}^{N_2}\frac{1}{N_1+1}\card\big\{M\in\{ 0,\dots,N_1\}\mid f_M\text{ satisfies }(P)\big\},\]
\[\frac{1}{N_3+1}\sum_{N_2=0}^{N_3}\frac{1}{N_2+1}\sum_{N_1=0}^{N_2}\frac{1}{N_1+1}\card\big\{M\in\{ 0,\dots,N_1\}\mid f_M\text{ satisfies }(P)\big\}\dots\]
\end{rem}

\section{Behaviour of all the discretizations}\label{Sec8}

In the previous sections we showed that the dynamical behaviour of discretizations seems very chaotic depending of the order of discretization. In contrast, the dynamics of a generic homeomorphism is well known (see e.g. \cite{MR2931648}) and independent from the homeomorphism. One even has a 0-1 law on $\Hom(X,\lambda)$ (see \cite{Glas-zero} or the last chapter of \cite{MR2931648}) which states that either a given ergodic property on conservative homeomorphisms is generic, or its contrary is generic. Thus, the dynamics of a generic homeomorphism and that of its discretizations seem quite of independent. In fact one can deduce some dynamical features of a generic homeomorphism from the corresponding dynamical features of \emph{all} its discretizations. For instance, the following property can be easily deduced from corollary \ref{corovar2}.

\begin{prop}\label{détecpér}
Let $f$ be a generic homeomorphism of $\Hom(X,\lambda)$ and $p$ be an integer. Then $f$ has a periodic orbit with period $p$ if and only if there exists infinitely many integers $N$ such that $f_N$ has a periodic orbit with period $p$.
\end{prop}

\begin{proof}[Proof of proposition \ref{détecpér}]
An easy variation of corollary \ref{corovar2} shows that if $p$ is a period of a periodic orbit of a generic homeomorphism $f$, then there exists infinitely many discretizations $f_N$ such that $f_N$ has a periodic orbit with period $p$. The other implication of the proposition follows easily from a compactness argument.
\end{proof}

We now try to obtain information about invariant measures of a generic homeomorphism from invariant measures of its discretizations. More precisely, given all the invariant measures of discretizations of a generic homeomorphism, what can be deduced about invariant measures of the initial homeomorphism? A first step in this study was performed by T. Miernowski in 2006 in part 8 of his article \cite{MR2279269}:

\begin{prop}[Miernowski]\label{Miern}
Let $f : X\to X$ be a uniquely ergodic homeomorphism and $\mu^f$ his unique invariant probability measure. For all $N\in\N$ let $\gamma_N\subset E_N$ be a periodic cycle of $f_N$ and $\nu_N$ the uniform probability measure on $\gamma_N$. Then the weak convergence $\nu_N\rightharpoonup\mu^f$ occurs independently of the choice of the cycles $\gamma_N$.
\end{prop}

The proof of this proposition essentially consists in an appropriate application of Prokho\-rov's theorem which express the compactness of the set of probability measures on $X$. We now set a theorem of this kind for \emph{generic} homeomorphisms. Recall that $\mu^f_N$ is the limit in the sense of Cesàro of the pushforwards by iterates of $f_N$ of uniform measure on $E_N$ (see definition \ref{defmes}):
\[\mu^f_{N} = \lim_{m\to\infty}\frac 1m \sum_{i=0}^{m-1}(f_N)_*^i \lambda_N.\]
The measure $\mu^f_N$ is supported by the maximal invariant set of $f_N$; it is uniform on every periodic orbit and the total weight of a periodic orbit is proportional to the size of its basin of attraction. The following theorem expresses that we can obtain all the invariant measures of a generic homeomorphism from all the invariant measures of its discretizations:

\begin{theoreme}\label{EnsMesInv}
Let $f\in\Hom(X,\lambda)$ be a generic homeomorphism and suppose that the sequence of grids $(E_N)_{N\in\N}$ is self-similar. Let $\mathcal M_N$ be the set of probability measures on $E_N$ that are invariant under $f_N$. Then the upper limit over $N$ (for Hausdorff metric) of the sets $\mathcal M_N$ is exactly the set $\mathcal M$ of probability measures that are invariant under $f$.
\end{theoreme}

Before giving a detailed proof of theorem \ref{EnsMesInv} let us give its main arguments. To show that the upper limit of sets $\mathcal M_N$ is a subset of $\mathcal M$, by a compactness argument, we only have to prove that limit points of a sequence $(\nu^f_N)_{N\in\N}$, such that every measure $\nu^f_N$ is invariant under $f_N$, are invariant under $f$. This can be easily verified in using uniform convergence of $f_N$ to $f$ and equicontinuity of the measures $\nu^f_N$. Then (again by a compactness argument) we have to show that generically every invariant measure is a limit point of a sequence of discretizations. An ad hoc application of Baire's theorem reduces the proof to which of the following variation of Lax's theorem\footnote{The first point of this lemma is somehow a weak version of the latter. It will be useful in the next section.}:

\begin{lemme}[Ergodic variation of Lax's theorem]\label{Laxergod}
Suppose that the sequence of grids $(E_N)_{N\in\N}$ is self-similar. For all $f\in\Hom(X,\lambda)$, for all collection of $f$-invariant measures $\nu_1,\dots,\nu_\ell$, for all $\varep>0$ and $k_0,N_0\in\N$:
\begin{enumerate}
\item It exists $g\in\Hom(X,\lambda)$ and $N\ge N_0$ such that $d(f,g)<\varep$, $d(\nu_1,\mu_N^g)<\frac{1}{k_0}$ and $\mu_N^g$ is $g$-invariant.
\item It exists $g\in\Hom(X,\lambda)$ and $N\ge N_0$ such that $d(f,g)<\varep$, and for all $i\le\ell$ there exists a measure $\nu_{i,N}^g$ such that $d(\nu_i,\nu_{i,N}^g)<\frac{1}{k_0}$ and that $\nu_{i,N}^g$ is invariant under $g$ and $g_N$.
\end{enumerate}
\end{lemme}

Suppose first that $\ell=1$ and that $\nu_1$ is ergodic. For this purpose we apply Birkhoff's theorem to $f$, $\nu_1$, $\varphi$ and a recurrent point $x$: for all $M$ large enough,
\[\frac{1}{M}\sum_{k=0}^{M-1}\varphi\circ f^k(x)\simeq\int\varphi\,\ud \nu_1.\]
Since $x$ is recurrent we can choose an integer $M$ large enough such that $x\simeq f^M(x)$. First we approximate $f$ by a cyclic permutation $\sigma_N$ given by Lax's theorem, then we slightly modify $\sigma_N$ into a map $\sigma'_N$ by choosing $\sigma_N'(\sigma_N^M(x_N)) = x_N$, as in proposition \ref{var1}. The measure  $\nu^{\sigma_N'}_N$ is the uniform measure on the orbit $x_N,\dots, \sigma_N^{M-1}(x_N)$, so it is close to $\nu_1$ with respect to the test function $\varphi$. This property can be obtained for all test functions by a diagonal-like method. The proof of the lemma when $\nu_1$ is not ergodic but only invariant is obtained by approximating the invariant measure by a finite convex combination of ergodic measures; this proves the first point of the lemma. To obtain the latter, it suffices to use the self-similarity hypothesis: on at least $\ell$ sub-grids we apply the first point to each measure $\nu_i$.

\begin{proof}[Proof of theorem \ref{EnsMesInv}]
\emph{For all $\delta>0$, for all $N$ large enough, $\mathcal M_N\subset B(\mathcal M,\delta)$.} For a set $A$ and $\delta>0$, we denote by $B(A,\delta)$ the set of elements whose distance to $A$ is smaller than $\delta$. This inclusion follows easily from the upper semi-continuity of the application $g\mapsto \mathcal M^g$ (where $\mathcal M^g$ denotes the set of Borel probability measures which are invariant under $g$) and the compactness of $\Prb$.
\bigskip

\emph{Generically, for all $\delta>0$ and all $N_0\in\N$, there exists $N\ge N_0$ such that $\mathcal M\subset B(\mathcal M_N,\delta)$.} Recall that we denote by $\Prb$ the set of all Borel probability measures on $X$. Set $d$ a distance on $\Prb$ defining the weak-* topology. Thereafter homeomorphisms will be taken in the generic set
\[\bigcap_{N\in\N} \mathcal{D}_{N},\]
made of homeomorphisms whose $N$th discretization is uniquely defined for all $N$. Consider
\[\mathcal{A} = \bigcap_{(N_0,k_0)\in\N^2} \mathcal{O}_{N_0,k_0},\]
where
\[\mathcal{O}_{N_0,k_0} = \left\{f\in \bigcap_{N\in\N} \mathcal{D}_{N}\  \middle\vert\
\exists N\ge N_0 \text{ s.t. } \mathcal M\subset B\big(\mathcal M_N,\frac{3}{k_0}\big)
\right\}.\]
Trivially, if $f\in \mathcal{A}$, then the upper limit of the sets $\mathcal M_N$ contains $\mathcal M$.
\bigskip

To show that the sets $\mathcal{O}_{N_0,k_0}$ are open, it suffices to remark that if $M\subset B\big(\mathcal M_N,\frac{3}{k_0}\big)$ for a given $N$, then it is also true on a neighborhood of $f$ since $g\mapsto \mathcal M_N^g$ is constant on a neighborhood of $f$ and $g\mapsto \mathcal M^g$ is upper semi-continuous.
\bigskip

It remains to show that the sets $\mathcal{O}_{N_0,k_0}$ are dense; it follows from the second point of the variation of Lax's theorem (lemma \ref{Laxergod}): 
by upper semi-continuity of $g\mapsto \mathcal M^g$, one has $\mathcal M^g\subset B\big(\mathcal M^f,\frac{1}{k_0}\big)$ for all $g$ close enough to $f$. Moreover, by compactness, there exists $f$-invariant measures $\nu_1,\dots,\nu_{\ell}$ such that $\mathcal M^f\subset \bigcup_i B\big(\nu_i,\frac{1}{k_0}\big)$. So it suffices to find $g$ close to $f$ such that for a $N\ge N_0$, there exists $g_N$-invariant measures $\nu_{1,N}^g,\dots,\nu_{\ell,N}^g$ such that $d(\nu_i,\nu_{i,N}^g)<\frac{1}{k_0}$ for all $i$, but this is exactly the conclusion of the lemma.
\end{proof}

\begin{proof}[Proof of lemma \ref{Laxergod}]
To begin with we prove the first point of the lemma. Suppose first that the measure $\nu=\nu_1$ is ergodic. We want to show that there exists a homeomorphism $g$ whose distance to $f$ is smaller than $\varep$, and an integer $N\ge N_0$ such that $d(\nu,\mu_{N}^g)<\frac{2}{k_0}$. Let $(\varphi_j)_{j\in\N}$ be a sequence which is dense in the set of continuous maps from $X$ into $\R$. By Prokhorov's theorem, there exists $i\in\N$ such that if we have
\[\left|\int\varphi_j\,\ud \nu - \int\varphi_j\,\ud \mu\right|\le \frac{1}{i}\]
for all $j\le i$, then $d(\nu,\mu)<\frac{1}{k_0}$.

Since $\nu$ is ergodic, for all continuous map $\varphi$, by Birkhoff's theorem,
\begin{equation}\label{eqergo}
\frac{1}{M}\sum_{m=0}^{M-1}\varphi\circ f^m(x)\underset{M\to+\infty}{\longrightarrow}\int\varphi\,\ud \nu
\end{equation}
for $\nu$-a.e. $x$. Let $x\in X$ be a recurrent point for $f$ satisfying equation (\ref{eqergo}) for all $\varphi_j$ (such points form $\nu$-full measure set). For all $j$ there exists $M_j\in\N$ such that for all $M\ge M_j$ one has
\begin{equation}\label{eqergo2}
\left|\frac{1}{M}\sum_{m=0}^{M-1}\varphi_j\circ f^m(x)-\int\varphi_j\,\ud \nu \right|\le \frac{1}{2i}.
\end{equation}

Let $M'_i = \max_{j\le i} M_j$. Since $x$ is recurrent, there exists $\tau\ge M'_i$ such that $d(x,f^\tau(x))\le\varep/4$. Let $\sigma_N$ be a map from $E_N$ into itself given by Lax's theorem: it is a cyclic permutation and its distance to $f$ is smaller than $\varep/2$. For all $N$ large enough the orbit $(\sigma_{N}^m(x_{N}))_{0\le m\le \tau}$ shadows the orbit $(f^m(x))_{0\le m\le \tau}$, thus $d(x_{N},\sigma_N^\tau(x_{N}))<\varep/2$. Then we ``close'' the orbit of  $x_{N}$ between the points $x_{N}$ and $\sigma_{N}^\tau(x_{N})$, i.e. we set (as in proposition \ref{var1}, see also figure \ref{trajectoire})
\[\sigma'_{N}(y_{N}) = \left\{\begin{array}{ll}
x_{N}\quad & \text{if}\ \,y_{N} = \sigma_{N}^{\tau-1}(x_N)\\
\sigma_{N}(y_N)\quad & \text{otherwise.} 
\end{array}\right.\]
Then $d_{N}(\sigma'_{N},f)<\varep$ and $(x_N,...,\sigma_{N}'^{\tau-1}(x_{N}))$ is a periodic orbit for $\sigma'_N$ whose basin of attraction is the whole set $E_N$.

Since the periodic orbit $(\sigma_{N}'^m(x))_{0\le m\le \tau}$ attracts $E_{N}$, for all $M'$ large enough and $y_{N}\in E_{N}$ we have
\begin{equation}\label{eqergo4}
\frac{1}{M'}\sum_{m=0}^{M'-1}\varphi_j\circ \sigma_{N}'^m(y_{N})\underset{M'\to\infty}{\longrightarrow}  \frac{1}{\tau}\sum_{m=0}^{\tau-1}\varphi_j\circ \sigma_{N}'^m(x_{N}).
\end{equation}
With proposition \ref{extension}, the same way as in lemma \ref{lemmetrans}, we construct a homeomorphism $g^N$ from the map $\sigma'_N$ such that the discretization $(g^N)_N$ and $\sigma'_N$ fit together, and such that $g$ and the restriction of $(g^N)_N$ to the orbit of $x_N$ fit together. Since $d_{N}(f,\sigma_{N}')<\varep$, we can furthermore assume that $d(f,g^N)<\varep$. Thus, equation (\ref{eqergo4}) implies that
\begin{equation}\label{defmug}
\mu^{g^N}_N = \frac{1}{\tau}\sum_{m=0}^{\tau-1}\delta_{\sigma'^m_N(x_N)}.
\end{equation}

Since $\varphi_j$ is continuous, increasing $N$ (if necessary) such that the orbit $(\sigma_{N}^m(x_{N}))_{0\le m\le \tau}$ shadows $(f^m(x))_{0\le m\le \tau}$ better, we have for all $j\le i$:
\begin{equation}\label{eqergo3}
\left|\frac{1}{\tau}\sum_{m=0}^{\tau-1}\varphi_j\circ f^m(x)-\frac{1}{\tau}\sum_{m=0}^{\tau-1}\varphi_j\circ \sigma_{N}'^m(x_{N}) \right|\le \frac{1}{2i}
\end{equation}

We now have all the necessary estimations to compute the distance between $\mu^{g^N}_N$ and $\nu$: by equation (\ref{defmug}) and triangle inequality,
\begin{eqnarray*}
\left| \int\varphi_j\,\ud\mu^{g^N}_N-\int\varphi_j\,\ud\nu\right| & = & \left|\frac{1}{\tau}\sum_{m=0}^{\tau-1}\varphi_j\circ \sigma_{N}'^m(x_{N})-\int\varphi_j\,\ud\nu\right|\\
   & \le & \left|\frac{1}{\tau}\sum_{m=0}^{\tau-1}\varphi_j\circ \sigma_{N}'^m(x_{N})-\frac{1}{\tau}\sum_{m=0}^{\tau-1}\varphi_j\circ f^m(x_{N})\right|\\
   &     & + \left|\frac{1}{\tau}\sum_{m=0}^{\tau-1}\varphi_j\circ f^m(x_{N}) - \int\varphi_j\,\ud\nu\right|
\end{eqnarray*}
Hence, with equation (\ref{eqergo2}) applied to $M=\tau$ and equation (\ref{eqergo3}),
\[\left| \int\varphi_j\,\ud\mu^{g^N}_N-\int\varphi_j\,\ud\nu\right| \le \frac{1}{i}.\]
We obtain $d(\mu^{g^N}_{N},\nu)<\frac{1}{k_0}$. Since $\mu^{g^N}_{N}$ is invariant under $g^N$ ($g^N$ is a cyclic permutation of the orbit of $x_N$), $g^N$ is the desired homeomorphism.

\bigskip
In the general case the measure $\nu$ is only invariant (and not ergodic). It reduces to the ergodic case by the fact that the set of invariant measures is a compact convex set whose extremal points are exactly the ergodic measures: by the Krein-Milman theorem, for all $M\ge 1$ there exists an $f$-invariant measure $\nu'$ which is a finite convex combination of ergodic measures: 
\[\nu' = \sum_{j=0}^{r} \lambda_j\nu_j^e,\]
and whose distance to $\nu$ is smaller than $\frac{1}{k_0}$.

Then we use the hypothesis of self-similarity of the grids: for all $\varep>0$ and all $P\in\N$ there exists a map $\sigma_N$ whose distance to $f$ is smaller than $\varep$ and whose at least $p\ge P$ cycles are conjugated to a permutation of a coarser subdivision, via the composition of a fixed homothety by a translation. Moreover the union on these $p$ cycles fulfils the set $E_N$ with a proportion of at least $1-\frac{1}{2k_0}$. Take $P\in\N$ such that $\frac{2r}{P}\le \frac{1}{2k_0}$. For all $j$ between $1$ and $r$ we set
\[\lambda'_j = \frac{\lfloor\lambda_j p\rfloor}{p}\]
and
\[\lambda'_0 = 1-\sum_{j=1}^{r}\lambda'_j.\]
Then for all $j>0$ we have $|\lambda_j-\lambda'_j|\le \frac{1}{P}$ and $|\lambda_0-\lambda'_0|\le \frac{r}{P}$. Since $\sum \lambda'_j=1$ we can associate bijectively a set $\Lambda_j$ made of $p\lambda'_j$ cycles to each integer $j$; the union of the sets $\Lambda_j$ fills the set $E_N$ with a proportion of at least $1-\frac{1}{2k_0}$. We do the previous construction concerning ergodic measures for each $j$: increasing  $N$ if necessary, the cycles of each $\Lambda_j$ can be changed such that each of them carries a (unique) invariant probability measure ${\nu_j^e}'$ satisfying $d({\nu_j^e} , {\nu_j^e}')\le\frac{1}{2k_0}$. These modifications transform the map $\sigma_{N}$ into a map $\sigma'_N$ satisfying $d(\sigma'_N, \sigma_N)<\varep$. Then the measure $\mu^{\sigma'_N}_{N}$ associated with the map $\sigma'_{N}$ satisfies:
\[d\left(\mu^{\sigma'_N}_{N}\, ,\, \sum_{j=0}^{r} \lambda'_j{\nu_j^e}'\right)\le \frac{1}{2k_0}\]
and we have:
\begin{eqnarray*}
d( \nu' ,\, \mu^{\sigma'_N}_{N} ) & \le & \sum_{j=0}^{r} d\left(\lambda_j\nu_j^e\, ,\, \lambda'_j{\nu_j^e}'\right) + \frac{1}{2k_0}\\
   & \le & \sum_{j=0}^{r} d\left(\lambda_j\nu_j^e\, ,\, \lambda'_j{\nu_j^e}\right) + \sum_{j=0}^{r} d\left(\lambda'_j\nu_j^e\, ,\, \lambda'_j{\nu_j^e}'\right) + \frac{1}{2k_0}\\
   & \le & \sum_{j=0}^{r} \left|\lambda_j - \lambda'_j\right| + \sum_{j=0}^{r} \lambda'_j d\left(\nu_j^e\, ,\, {\nu_j^e}'\right) + \frac{1}{2k_0}\\
   & \le & \frac{2r}{P} + \frac{1}{2k_0}  + \frac{1}{2k_0} \le  \frac{3}{2k_0}. 
\end{eqnarray*}
Thus $d(\nu,\mu^{\sigma'_N}_{N})\le \frac{2}{k_0}$. As above the desired homeomorphism is obtained by extending $\sigma'_{N}$ to a homeomorphism of $X$. We have proven the first point of the lemma.
\bigskip

Finally we handle the second point of the lemma. To do that, it suffices to apply the first point on $\ell$ different sub-grids of a given grid $E_{N_0}$. On the $i$-th sub-grid, define the discretization $g_N$ such that it has a single invariant measure which is close to $\nu_i$. This proves the lemma.
\end{proof}

There is a kind of duality between invariant measures and invariant compact sets, thus the previous theorem is also true for invariant compact sets:

\begin{prop}\label{CompactInv}
Let $f\in\Hom(X,\lambda)$ be a generic homeomorphism and suppose that the sequence of grids $(E_N)_{N\in\N}$ is self-similar. Set $\mathcal I_N$ the set of subsets of $E_N$ that are invariant under $f_N$. Then the upper limit over $N$ (for the Hausdorff metric) of the sets $\mathcal I_N$ is exactly the set of compacts that are invariant under $f$.
\end{prop}

\begin{proof}[Sketch of proof of proposition \ref{CompactInv}]
To show that all limit point $K$ of a sequence $(K_N)$ of compact invariant sets is a compact invariant set, it suffices to calculate $d_H(K,f(K))$. For the other inclusion, as in the case of invariant measures, it follows from an approximation lemma whose proof is the same as for invariant measures:
\end{proof}

\begin{lemme}[Compact variation of Lax's theorem]\label{CompacLax}
Suppose that the sequence of grids $(E_N)_{N\in\N}$ is self-similar. For all $f\in\Hom(X,\lambda)$, for all collection of $f$-invariant compacts $K_1,\dots,K_\ell$, for all $\varep>0$ and $k_0,N_0\in\N$:
\begin{enumerate}
\item It exists $g\in\Hom(X,\lambda)$ and $N\ge N_0$ such that $d(f,g)<\varep$, $d_H(K_1,\Omega(f_N))<\frac{1}{k_0}$ and $\Omega(f_N)$ is $g$-invariant.
\item It exists $g\in\Hom(X,\lambda)$ and $N\ge N_0$ such that $d(f,g)<\varep$, and for all $i\le\ell$, there exists a compact $K_{i,N}^g$ such that $d_H(K_i,K_{i,N}^g)<\frac{1}{k_0}$ and that $K_{i,N}^g$ is invariant under $g$ and $g_N$.
\end{enumerate}
\end{lemme}

\section{Physical measures}\label{grobra}

In the previous part we proved that generically, the upper limit of the sets of invariant measures of discretizations is the set of invariant measures of the initial homeomorphism. It expresses that the sets of invariant measures of discretizations behave ``as chaotically as possible''. However, one might expect that physical measures (Borel measures $\mu$ such that $\mu=\mu^f_x$ is verified on a set of points $x$ with $\lambda$ positive measure, see e.g. \cite{Youn-wha}) play a specific part: their definition expresses that they are the measures that can be observed in practice, because they governs the ergodic behaviour of $\lambda$ a.e. point. So one can hope that the natural invariant measures $\mu^f_N$ of $f_N$, which can be seen as the physical measures of $f_N$, converge to the physical measures of $f$. This expectation is supported by the following variation of proposition \ref{Miern}, obtained by replacing $\nu_N$ by $\mu^f_N$: \emph{if $f$ is uniquely ergodic, then the measures $\mu^f_N$ converge weakly to the only measure $\mu^f$ that is invariant under $f$}.

In this section we show that this is not at all the case: the sequence of measures $(\mu^f_N)$ accumulates of the whole set of $f$-invariant measures\footnote{It is easy to see that if $(\nu_N)_{N\in\N}$ is a sequence of measures, with $\nu_N$ invariant under $f_N$ for all $N$, then its limit set is included in the set of $f$-invariant measures. Again, this reflects the fact that the behaviour of discretizations of a generic homeomorphism is ``as chaotic as possible''.}. More precisely one has the following theorem:

\begin{theoreme}\label{mesinv}
If the sequence of grids $(E_N)_{N\in\N}$ is self-similar, for a generic homeomorphism $f\in\Hom(X,\lambda)$, the set of limit points of the sequence $(\mu^f_{N})_{N\in\N}$ is exactly the set of $f$-invariant measures.
\end{theoreme}

This theorem can be seen as a discretized version of the following conjecture:

\begin{conj}[F. Abdenur, M. Andersson, \cite{MR3027586}]
A homeomorphism $f$ which is generic in the set of homeomorphisms of $X$ (without measure preserving hypothesis) that do not have any open trapping set is \emph{wicked}, i.e. it is not uniquely ergodic and the measures
\[\frac{1}{m}\sum_{k=0}^{m-1}f_*^k(\Leb)\]
accumulate on the whole set of invariant measures under~$f$.
\end{conj}

The behaviour described in this conjecture is the opposite of that consisting of possessing a physical measure.

\begin{proof}[Sketch of proof of theorem \ref{mesinv}]
The proof is similar to which of theorem \ref{EnsMesInv}: the set $\mathcal A$ is replaced by
\[\mathcal{A}' = \bigcap_{(\ell,N_0,k_0)\in\N^3} \mathcal{O}_{\ell,N_0,k_0},\]
where
\[\mathcal{O}_{\ell,N_0,k_0} = \left\{f\in \bigcap_{N\in\N} \mathcal{D}_{N}\  \middle\vert\
\begin{array}{l}
\Big(\exists\nu f\text{-inv.}: d(\nu,\tilde\nu_\ell)\le\frac{1}{k_0}\Big)\implies\\
\Big(\exists N\ge N_0 : d(\tilde\nu_\ell,\mu^f_{N})<\frac{2}{k_0}\Big)
\end{array}\right\},\]
and the second point of lemma \ref{Laxergod} is replaced by the first.
\end{proof}

\begin{rem}
Taking over the proof of the theorem, if we set $x\notin\bigcap_{N} E'_{N}$, then generically the measures $\mu^f_{{N},x}$ accumulate on the whole set of $f$-invariant measures. This seems to contradict the empirical observations made by A. Boyarsky in 1986 (see \cite{Boya-comp} or \cite{Gora-why}): when a homeomorphism $f$ has a unique ergodic measure $\mu^f$ which is absolutely continuous with respect to Lebesgue measure ``most of'' the measures $\mu^f_{{N},x}$ tend to measure $\mu^f$. However, the author does not specify in what sense he means ``most of the points'', or if his remark is based on a tacit assumption of regularity for $f$.
\end{rem}

Note that as in the previous section, we have a compact counterpart of theorem \ref{mesinv}:

\begin{prop}
If the sequence of grids $(E_N)_{N\in\N}$ is self-similar, for a generic homeomorphism $f\in\Hom(X,\lambda)$, the set of limit points of the sequence $(\Omega(f_{N}))_{N\in\N}$ is exactly the set of $f$-invariant compact sets.
\end{prop}

\part{Discretizations of a generic dissipative homeomorphism}

Throughout this section, we fix a manifold $X$ and a good measure $\lambda$ as defined in section \ref{0..2}. Now homeomorphisms are no longer supposed conservative. More formally, we recall that we denote by $\Hom(X)$ the set of all \emph{dissipative} homeomorphisms of $ X $ (without assumption of conservation of a given measure), we are interested in properties of discretizations of generic elements of $\Hom(X)$. The manifold $X$ is equipped with a sequence $(E_N)_{N\in\N}$ of discretization grids.

Generic topological properties of dissipative homeomorphisms have been discussed in a survey of E. Akin, M. Hurley and J. Kennedy \cite{MR1980335}. However, ergodic properties of generic homeomorphisms had not been studied until the work of F. Abdenur and M. Andersson in \cite{MR3027586}, in which the authors are interested in the typical behaviour of Birkhoff averages. For this purpose they establish a technical lemma they call \emph{shredding lemma} which allows them to show that a generic homeomorphism is \emph{weird}.

\begin{definition}\label{stange}
A homeomorphism $f$ is said \emph{weird} if almost every point $x\in X$ (for $\lambda$) has a Birkhoff's limit $\mu^f_x$, and if $f$ is \emph{totally singular} (\emph{i.e.} there exists a Borel set with total measure whose image by $f$ is null measure) and does not admit any physical measure.
\end{definition}

In opposition to the conservative case, the lack of conservation of a given measure allows us to perturb any diffeomorphism to create attractors whose basins of attraction cover almost the entire space $X$. This is an open property, so such basins of attraction can be seen on discretizations.

With what has just been said, we can already note that properties of discretizations of generic dissipative and conservative homeomorphisms are quite different: if before (almost) all possible dynamical properties are observed on the discretizations of a homeomorphism, the dissipative case reveals a much more regular behaviour: ``the dynamics of discretizations of a homeomorphism $f$ converges to the dynamics of $f$''. For example, the measures $\mu^f_N$ accumulate on a single measure, namely $\mu^f_X$, and not on all the invariant measures under $f$.

\section{The shredding lemma and its discrete counterpart}

In their article \cite{MR3027586}, F. Abdenur and M. Andersson try to identify the generic ergodic properties of continuous maps and homeomorphisms of compact manifolds. More precisely, they study the behaviour of Birkhoff limits $\mu^f_x$ for a generic homeomorphism $ f \in \Hom(X) $ and almost every point $x$ for Lebesgue measure. To do this they define some interesting behaviours of homeomorphisms related to Birkhoff limits, including one they call \emph{weird} (see def \ref{stange}). This definition is supported by their proof, based on the shredding lemma, that a generic homeomorphism is weird. We outline an improvement of this lemma, whose main consequence is that a generic homeomorphism has many open attractive sets, all of small measure and decomposable into a small number of small diameter open sets:

\begin{lemme}[Shredding lemma, F. Abdenur, M. Andersson, \cite{MR3027586}]\label{déchet}
For all homeomorphism $f\in\Hom(X)$, for all $\varep,\delta>0$, there exists a family of regular pairwise disjoints open sets\footnote{An open set is said \emph{regular} if it is equal to the interior of its closure.} $U_1,\dots,U_\ell$ such that for all $\varep'>0$, there exists $g\in\Hom(X)$ such that $d(f,g)<\delta$ and:
\begin{enumerate}[(i)]
\item $g(\overline{U_j})\subset U_j$,
\item $\lambda(U_j)<\varep$,
\item $\lambda\left(\bigcup_{j=1}^\ell U_j\right)>1-\varep$,
\item $\lambda(g(U_j))<\varep\, \lambda(U_j)$,
\item there exists open sets $W_{j,1},\dots,W_{j,\ell_j}$ such that:
  \begin{enumerate}
  \item $\operatorname{diam}(W_{j,i})<\varep'$ for all $i\in\{1,\dots,\ell_j\}$,
  \item $g(\overline{W_{j,i}})\subset W_{j,i+1}$, for every $i\in\{1,\dots,\ell_j-1\}$ and $g(\overline{W_{j,\ell_j}})\subset W_{j,1}$,
  \item \[\overline{U_j}\subset\bigcup_{m\ge 0} g^{-m}\big(\bigcup_{i=1}^{\ell_j}W_{j,i}\big).\]
  \end{enumerate}
\end{enumerate}
\end{lemme}

The lemma as stated by F. Abdenur and M.Andersson in their paper can be obtained making $\varep = \varep'$. To get this improvement, it suffices to consider the sets $W_{j,i}$ given by the weak version of the lemma and to crush them with a contraction mapping, as the authors do in the first part of the proof of the lemma.

\begin{rem}
It can be easily seen that the lemma remains true on a neighborhood of the homeomorphism $g$.
\end{rem}

\begin{rem}\label{remsh}
We can further assume that:
\begin{enumerate}
\item The open sets $W_{j,i}$ have disjoints attractive sets, \emph{i.e.} for all $j\in\N$ and all $i\neq i'$, we have\footnote{These two sets are decreasing intersections of compact sets, so they are compact and nonempty.} 
\[\left(\bigcap_{m\ge 0}\,\bigcup_{m'\ge m} f^{m'}(W_{j,i})\right) \cap \left(\bigcap_{m\ge 0}\,\bigcup_{m'\ge m} f^{m'}(W_{j,i'})\right) = \emptyset.\]
For this it suffices to replace $W_{j,i}$ by $W'_{j,i}$ defined by induction by:
\begin{eqnarray*}
W'_{j,1} & = & W_{j,1}\\
W'_{j,i+1} & = & W_{j,i+1}\setminus \left(\bigcup_{i'\le i}\,\bigcap_{m\ge 0}\,\bigcup_{m'\ge m} f^{m'}(W_{j,i'})\right).
\end{eqnarray*}
\item Each set $W_{j,i}$ contains a periodic Lyapunov stable point. Indeed, for all $\varep>0$, there exists $x\in W_{j,i}$ and $m>0$ such that $d(x,f^m(x))<\varep$. Then we perturb $f$ in order to make $x$ $m$-periodic and attractive, and the Lyapunov stability remains true in a neighborhood of $f$.
\item The set $g(\overline{U_j})$ is ``independent'' from the choice of $\varep'$: there exists compact sets $V_j$ such that $g(\overline{U_j})\subset V_j$ and $V_j\subset U_j$, $\lambda(V_j)<\varep\, \lambda(U_j)$.
\end{enumerate}
\end{rem}

This lemma tells us a lot about the dynamics of a generic homeomorphism, which becomes quite clear: there are many attractors whose basins of attraction are small and attract almost all the manifold $X$. Moreover there is convergence of the attractive sets of the shredding lemma to the closure of the Lyapunov stable periodic points of $f$:

\begin{notation}\label{echo}
For an homeomorphism $f\in\Hom(X)$, we denote by $A_0$ the set of Lyapunov stable periodic points of $f$, \emph{i.e.} the sets of periodic points $x$ such that for all $\delta>0$, there exists $\eta>0$ such that if $d(x,y)<\eta$, then $d(f^m(x),f^m(y))<\delta$ for all $m\in\N$.
\end{notation}

\begin{coro}\label{convattra}
Let $f\in\Hom(X)$ verifying the conclusions of the shredding lemma for all $\varep=\varep'>0$ and $W_{j,i,\varep}$ be the corresponding open sets. Such homeomorphisms form a $G_\delta$ dense subset of $\Hom(X)$. Then the sets
\[A_\varep = \overline{\bigcup_{j,i} W_{j,i,\varep}}\]
converge for Hausdorff distance when $\varep$ tends to $0$ to a closed set. This is a null set\footnote{And better, if we are given a countable family $(\lambda_m)_{m\in\N}$ of good measures, generically $\lambda_m(A_0)=0$.} which coincides generically with the closure of the set $A_0$.
\end{coro}

\begin{proof}[Proof of corollary \ref{convattra}]
Let $f$ verifying the hypothesis of the corollary. We want to show that the sets $A_\varep$ tend to $A_0$ when $\varep$ goes to $0$. This is equivalent to show that for all $\delta>0$, there exists $\varep_0>0$ such that for all $\varep<\varep_0$, $A_0\subset B(A_\varep,\delta)$ and $A_\varep\subset B(A_0,\delta)$ (where $B(A,\delta)$ denotes the set of points of $X$ whose distance to $A$ is smaller than $\delta$). Subsequently we will denote by $U_{j,\varep}$ and by $W_{i,j,\varep}$ the open sets given by the shredding lemma for the parameters $\varep=\varep'$.

Let $\delta> 0$. We start by taking $x\in X$ whose orbit is periodic and Lyapunov stable. Then there exists $\eta>0$ such that if $d(x,y)<\eta$, then $d(f^m(x),f^m(y))<\delta/2$ for all $m\in\N$; let $O=B(x\eta)$. Then there exists $\varep_0>0$ such that for all $\varep\in]0,\varep_0[$, there exists $j\in\N$ such that the intersection between $O$ and $U_{j,\varep}$ is nonempty. Let $y$ be an element of this intersection. By compactness, there exists a subsequence of $(f^m(y))_{m\in\N}$ which tends to $x_0$; moreover $f^m(y)\in \bigcup_{i}W_{j,i}$ eventually and $d(x,x_0)<\delta/2$. We deduce that $x\in  B\big(\overline{\bigcup_{j,i} W_{j,i,\varep_0}},\delta/2\big)$ for all $\varep_0$ small enough. Since $\overline{A_0} $ is compact, it is covered by a finite number of balls of radius $\delta/2$ centred in some points $x_i$ whose orbits attract nonempty open sets. Taking $\varep_0'$ the minimum of all the $\varep_0$ associated to the $x_i$, the inclusion $A_0\subset B(A_\varep,\delta)$ takes place for all $\varep <\varep_0'$.

In the other way, let $\delta>0$, $\varep<\delta$ and focus on the set $W_{j,i,\varep}$. By point 2. of the remark \ref{remsh}, we can suppose that there exists $x\in W_{j,i,\varep}$ whose orbit is periodic and Lyapunov stable. Thus $x\in A_0$ and since the diameter of $W_{j,i,\varep}$ is smaller than $\delta$, $W_{j,i,\varep}\subset B(A_0,\delta)$. The corollary is proved.
\end{proof}

Now we establish a discrete counterpart of the shredding lemma. This result shows that the general dynamics of a generic dissipative homeomorphism is quite simple; thus the dynamics of the discretizations these homeomorphisms is also quite simple. More precisely, since having basins of attraction is stable by perturbation, we have a similar statement for discretizations of a generic homeomorphism. The following develops these arguments.

To each point $x_{N} \in E_{N}$, we associate a closed set $\overline{P_N^{-1}(\{x_N\})}$, made of the points in $X$ one of whose the projections on $E_{N}$ is $x_{N}$. The closed sets $\overline{P_N^{-1}(\{x_N\})}$ form a basis of the topology of $X$ when $N$ runs through $\N$ and $x_{N}$ runs through $E_{N}$.\label{CNxN} Let $f\in\Hom(X)$, $N\in\N$ and $\vartheta : \N\to\R_+^*$ be a function that tends to $+\infty$ at $+\infty$. Let $\delta,\varep>0$ and $U_j$ be the sets obtained by the shredding lemma for $f$, $\delta$ and $\varep$.

For all $j\le \ell$ we denote by $\widetilde U_j^{N}$ the union over $x_{N}\in E_{N}$ of the closed sets $\overline{P_N^{-1}(\{x_N\})}$ whose intersection with $U_j$ are nontrivial. Then $\widetilde U_j^{N}$ tends to $U_j$ for the metric $d(A,B) = \lambda(A\Delta B)$, and for the Hausdorff metric. By point (3) of remark \ref{remsh}, these convergences are independent from the choice of $\varep'$. Thus, for all $k$ big enough, properties $(i)$ to $(iii)$ of the shredding lemma remain true for the discretizations $g_{N}$ (for arbitrary $g$ satisfying the properties of the lemma). Taking $\varep'$ small enough and modifying a little $g$ if necessary, there exists sets $W_{j,i}$ and $g\in\Hom(X)$ such that $d(f,g)<\delta$, $W_{j,i}\subset\overline{P_N^{-1}(\{x_N\})}$ and $\card(\bigcup_{j,i} W_{j,i}\cap E_N)\le\vartheta(N)$. The others estimations over the sizes of the sets involved in the lemma are obtained similarly. Finally we have:

\begin{lemme}[Discrete shredding lemma]\label{déchetdiscr}
For all $f\in \Hom(X)$, for all $\varep,\delta>0$ and all function $\vartheta : \N\to\R_+^*$ that tends to $+\infty$ at $+\infty$, there exists $N_0\in\N$ such that for all $N\ge N_0$, there exists a family of subsets $U_1^{N},\dots,U_\ell^{N}$ of $E_{N}$ and $g\in\Hom(X)$ such that $d(f,g)<\delta$ and:
\begin{enumerate}[(i)]
\item $g_{N}(  U_j^{N})\subset   U_j^{N}$,
\item $\card(U_j^{N})<\varep q_{N}$,
\item $\card\left(\bigcup_{j=1}^\ell   U_j^{N}\right)>(1-\varep)q_{N}$,
\item $\card(g_{N}(U_j^{N}))<\varep\card(U_j^{N})$,
\item for all $j$, there exists $w_{j,1}^{N},\dots,w_{j,\ell_j}^{N}\in E_N$ such that
  \begin{enumerate}
  \item $\card(\bigcup_{j,i} \{w_{j,i}\})\le\vartheta(N)$,
  \item $g_N(w_{j,i}^N) = w_{j,i+1}^N$, for every $i\in\{1,\dots,\ell_j-1\}$ and $g_N(w_{j,\ell_j}^N) = w_{j,1}^N$,
  \item \[U_j^{N}\subset\bigcup_{m\ge 0} g_{N}^{-m}\big(\bigcup_{i=1}^{k_j}  w_{j,i}^{N}\big),\]
  \end{enumerate}
\item for all $j$ and all $i$, there exists sets $U_j$ and $W_{j,i}$ satisfying properties $(i)$ to $(v)$ of the shredding lemma such that $U_j\subset P_N^{-1}(U_j^{N})$ and $W_{j,i}\subset P_N^{-1}(w_{j,i}^{N})$.
\item if $N\ge N_0$, then for all $j$ and all $i$, we have
\[\sum_j d_H(\overline{U_j},\overline{P_N^{-1}(U_j^{N})})<\varep\quad\text{and}\quad\sum_{i,j} d_H(\overline{W_{j,i}},\overline{P_N^{-1}(w_{j,i}^{N})})<\varep\]
(where $d_H$ is the Hausdorff metric), and
\[\sum_j \lambda(U_j\Delta P_N^{-1}(U_j^{N}))<\varep\quad\text{and}\quad\sum_{i,j} \lambda(W_{j,i}\Delta P_N^{-1}(w_{j,i}^{N}))<\varep.\]
\end{enumerate}
\end{lemme}

\begin{rem}
Properties \emph{(i)} to \emph{(v)} are discrete counterparts of properties \emph{(i)} to \emph{(v)} of the continuous shredding lemma, but the two last ones reflect the convergence of the dynamics of discretizations to that of the original homeomorphism.
\end{rem}

This lemma implies that we can theoretically deduce the behaviour of a generic homeomorphism from the dynamics of its discretizations. The next section details this remark.

\section{Dynamics of discretizations of a generic homeomorphism}\label{label}

To begin with, we deduce from the shredding lemma that the dynamics of discretizations $f_N$ tends to that of the homeomorphism $f$. More precisely almost all orbits of the homeomorphism are $\delta $-shadowed by the orbits of the corresponding discretizations.

\begin{definition}
Let $f$ and $g$ be two maps from a metric space $X$ into itself, $x,y\in X$ and $\delta>0$. We say that the orbit of $x$ by $f$ \emph{$\delta$-shadows} the orbit of $y$ by $g$ if for all $m\in\N$, $d(f^m(x),g^m(y))<\delta$.
\end{definition}

\begin{coro}\label{stabilis}
For a generic homeomorphism $f\in\Hom(X)$, for all $\varep>0$ and all $\delta>0$, there exists an open set $A$ such that $\lambda(A)>1-\varep$ and $N_0\in\N$ such that for all $N\ge N_0$ and all $x\in A$, the orbit of $x_N = P_N(x)$ by $f_N$ $\delta$-shadows that of $x$ by $f$.

Therefore, for a generic homeomorphism $f$, there exists a full measure dense open set $O$ such that for all $x\in O$, all $\delta>0$ and all $N$ large enough, the orbit of $x_N$ by $f_N$ $\delta$-shadows that of $x$ by $f$.
\end{coro}

\begin{proof}[Proof of corollary \ref{stabilis}]
This easily follows from the discrete shredding lemma, and especially from the fact that the points  $w_{j,i}^N$ tend to the sets $W_{j,i}$ for Hausdorff metric.
\end{proof}

This statement is a bit different from the genericity of shadowing (see \cite{MR1711347}): here the start point is not a pseudo-orbit but a point $x\in X$; the corollary \ref{stabilis} expresses that we can ``see'' the dynamics of $f$ on that of $f_N$, with arbitrarily high precision, provided that $N$ is large enough. Among other things, this allows us to observe the basins of attraction of the Lyapunov stable periodic points of $f$ on discretizations. Better yet, to each family of attractors $(W_{j,i})_i$ of the basin $U_j$ of the homeomorphism correspond a unique family of points $(w_{j,i}^N)_i$ that are permuted cyclically by $f_N$ and which attracts a neighborhood of $U_j$. Thus attractors are shadowed by cyclic orbits of $f_N$ and we can detect the ``period'' of the attractor (i.e. the integer $k_j$) on discretizations. This behaviour is the opposite of what happens in the conservative case, where discretized orbits and true orbits are very different for most points.

Again, in order to show that the dynamics of discretizations converge to that of the initial homeomorphism, we establish the convergence of attractive sets of $f_N$ to that of $f$. Recall that $A_0$ is the closure of the Lyapunov periodic points of $f$ (see notation \ref{echo}).

\begin{prop}\label{Hausmodif}
For a generic homeomorphism $f\in\Hom(X)$, the sets $\Omega(f_N)$ tend weakly to $A_0$ in the following sense: for all $\varep>0$, it existe $N_0\in \N$ such that for all $N\ge N_0$, there exists a subset $\widetilde E_N$ of $E_N$, stabilized by $f_N$ such that, noting $\widetilde \Omega(f_N)$ the corresponding maximal invariant set, we have $\frac{\card(\widetilde E_N)}{\card(E_N)}>1-\varep$ and $d_H(A_0,\widetilde \Omega(f_N))<\varep$.
\end{prop}

\begin{proof}[Proof of proposition \ref{Hausmodif}]
Let $\varep>0$. For all $N\in\N$, set $\widetilde E_N$ the union of the sets $U_j^{N}$ of lemma \ref{déchetdiscr} for the parameter $\varep$. This lemma ensures that $\widetilde E_N$ is stable by $f_N$ and fills a proportion greater than $1-\varep$ of $E_N$. We also denote by $\widetilde \Omega(f_N)$ the associated maximal invariant set:
\[\widetilde \Omega(f_N) = \bigcup_{x\in\bigcup_{j=1}^\ell \widetilde U_j^{N}} \omega_{f_N}(x).\]
Property \emph{(viii)} of lemma \ref{déchetdiscr} ensures that
\[\underset{N\to +\infty}{\overline\lim} d_H(A_\varep,\widetilde \Omega(f_N))<\varep.\]
To conclude, it suffices to observe that $A_\varep\to A_0$ for Hausdorff distance.
\end{proof}

Finally, we conclude this section with a last consequence of lemma \ref{déchetdiscr} which reflects that the ratio between the cardinality of the image of discretizations and which of the grid is smaller and smaller:

\begin{coro}\label{totsingdiscr}
For a generic homeomorphism $f\in\Hom(X)$,
\[\lim_{N\to+\infty}\frac{\card(f_{N}(E_{N}))}{\card(E_{N})} = 0\, ;\]
particularly for all $\delta>0$, there exists $N_0\in\N$ such that if $N\ge N_0$, then the cardinality of each orbit of $f_{N}$ is smaller than $\delta\card(E_{N})$.
\end{coro}

This corollary can be seen as a discrete analogue of the fact that a generic homeomorphism is totally singular, \emph{i.e.} that there exists a Borel set of full measure whose image under $f$ is zero measure. Again, it reflects the regularity of the behaviour of the discretizations of a dissipative homeomorphism: generically, we can describe the behaviour of \emph{all} (fine enough) discretizations. This is very different from the conservative case, where sometimes $f_N(E_N) =  E_N$ and sometimes $\card(f_N(E_N)) \le \vartheta(N)$ where $\vartheta:\R_+ \to \R_+^* $ is a given map that tends to $+\infty$ at $+\infty$. Moreover, we have the following estimation on the maximal invariant set (easily deduced from the discrete shredding lemma):

\begin{coro}\label{totsingdiscr2}
Let $\vartheta:\R_+ \to \R_+^* $ be a function that tends to $+\infty$ at $+\infty$. Then for a generic homeomorphism $f\in\Hom(X)$,
\[\lim_{N\to+\infty}\card(\Omega(f_{N}))\le\vartheta(N).\]
\end{coro}

It remains to study the behaviour of measures $\mu^{f_{N,U}}$ (see definition \ref{defmes}). Again, the results are very different from the conservative case: for any open set $U$, the measures $\mu^f_{{N},U}$ tend to a single measure, say $\mu^f_U$.

\begin{theoreme}\label{convmesdissip}
For a generic homeomorphism $f\in\Hom(X)$ and an open subset $U$ of $X$, the measure $\mu^f_U$ is well defined\footnote{In other words a generic homeomorphism is weird, see also \cite{MR3027586}.} and is supported by the closure of the set of attractors $A_0$. Moreover the measures $\mu^f_{{N},U}$ tend weakly to $\mu^f_U$.
\end{theoreme}

\begin{proof}[Sketch of proof of theorem \ref{convmesdissip}]
The proof of this theorem is based on the shredding lemma: the set of homeomorphisms which satisfy the conclusions of the lemma is a $G_\delta$ dense, so it suffices to prove that such homeomorphisms $f$ satisfy the conclusion of the proposition. Let $U$ be an open subset of $X$ and $\varphi : X\to\R$ be a continuous function.

We want to show that on one hand the integral $\int_X \varphi\, \ud\mu^f_U$ is well defined, i.e. that the Birkhoff limits for the function $\varphi$
\[\lim_{m\to+\infty} \frac{1}{m}\sum_{i=0}^{m-1} \varphi(f^i(x))\]
are well defined for a.e. $x\in U$; and on the other hand we have the convergence
\[\int_X  \varphi \,\ud\mu^f_{{N},U}\underset{N\to+\infty}{\longrightarrow}\int_X \varphi\, \ud\mu^f_U,\]

For the first step, the idea of the proof is that most of the points (for $\lambda$) go in a $W_{j,i}$. Since the iterates of the sets $W_{j,i}$ have small diameter, by uniform continuity, the function $\varphi$ is almost constant on the sets $f^m(W_{j,i})$. Thus the measure $\mu^f_x$ is well defined and almost constant on the set of points which goes in $W_{j,i}$. And by the same construction, since the dynamics of $f_N$ converge to that of $f$, and in particular that the sets $U_j^N$ and $\{w_{j,i}^N\}$ converge to the sets $U_j$ and $W_{j,i}$, the measures $\mu_{N,U}^f$ tend to the measures $\mu_U^f$.
\end{proof}

\part{Numerical simulations}\label{partietrois}

\section{Conservative homeomorphisms}

Now we present the results of numerical simulations of conservative homeomorphisms we conducted. We sought whether it is possible to observe behaviours such as those obtained in section \ref{partie 1.3} or that described by theorem \ref{mesinv} on actual simulations: it is not clear that the orders of discretization described by these results can be reached in practice, or if simple examples of homeomorphisms behave the same way as generic homeomorphisms.

Recall that we simulate homeomorphisms $f(x,y) = P\circ Q\circ P(x,y)$, where $P$ and $Q$ are both homeomorphisms of the torus that modify only one coordinate:
\[P(x,y) = (x,y+p(x,y))\quad\text{and}\quad Q(x,y) = (x+q(x,y),y),\]
and that we discretize according to the uniform grids on the torus:
\[E_N = \left\{\left(\frac{i_1}{N},\dots,\frac{i_n}{N}\right)\in \T^n \big|\ \forall j,\, 0\le i_j\le {N}-1\right\}.\]
To obtain conservative homeomorphisms, we choose $p(x,y)=p(x)$ and $q(x,y)=q(y)$. We perform simulations of two conservative homeomorphisms which are small perturbations of identity or of the standard Anosov automorphism $A : (x,y)\mapsto (x+y,x+2y)$:
\begin{itemize}
\item To begin with we study $f_1 = P\circ Q\circ P$, with\label{defex}
\[p(x) = \frac{1}{259}\cos(2\pi\times 227x)+\frac{1}{271}\sin(2\pi\times 253x),\]
\[q(y) = \frac{1}{287}\cos(2\pi\times 241y)+\frac{1}{263}\sin(2\pi\times 217y).\]
This conservative homeomorphism is a small $C^0$ perturbation of identity. Experience shows that even dynamical systems with fairly simple definitions behave chaotically (see e.g. \cite{Gamb-dif}). We can expect that a homeomorphism such as $f_1$ has a complex dynamical behaviour and even behaves essentially as a generic homeomorphism, at least for small orders of discretization. Note that we choose coefficients that have (virtually) no common divisors to avoid arithmetical phenomena such as periodicity or resonance.
\item We also simulate $f_2$ the composition of $f_1$ with the standard Anosov automorphism $A$, say $f_2 = P\circ Q\circ P\circ A$. This makes it a small $C^0$-perturbation of $A$. Thus $f_2$ is semi-conjugated to $A$ but not conjugated: to each periodic orbit of $A$ corresponds many periodic orbits of $f_2$. As for $f_1$, we define $f_2$ in the hope that the behaviour of its discretizations is fairly close to that of discretizations of a generic homeomorphism.
\end{itemize}

Note that the homeomorphisms $f_1$ and $f_2$ have at least one fixed point (for $f_1$, simply make simultaneously $p(x)$ and $q(y)$ equal to $0$; for $f_2$, note that $0$ is a persistent fixed point of $A$). Thus, the theoretical results indicate that for a generic homeomorphism $f$ which has a fixed point, infinitely many discretizations has a unique fixed point; moreover a subsequence of $(\mu_N^f)_N$ tends to an invariant measure under $f$ supported by a fixed point. We can test if this is true on simulations.

From a practical point of view, we restricted ourselves to grids of sizes smaller than $2^{15}\times 2^{15}$: the initial data become quickly very large, and the algorithm creates temporary variables that are of size of the order of five times the size of the initial data. For example, for a grid $2^{15}\times 2^{15}$, the algorithm takes between $25$ and $30$ Go of RAM on the machine.

\addtocontents{toc}{\SkipTocEntry}\subsection{Combinatorial behaviour}\label{simulgrafcons}

As a first step, we are interested in some quantities related to the combinatorial behaviour of discretizations of homeomorphisms. These quantities are:
\begin{itemize}
\item the cardinality of the maximal invariant set $\Omega(f_N)$,
\item the number of periodic orbits of the map $f_N$,
\item the maximal size of a periodic orbit of $f_N$.
\end{itemize}
Recall that according to corollaries \ref{typlax}, \ref{corovar2} and \ref{corovar3}, for a generic homeomorphism, the ratio $\frac{\card(\Omega(f_N))}{q_N}$ must fluctuate between $0$ and $1$ depending on $N$, the number of orbits must fluctuate between $1$ and (e.g.) $\sqrt{q_N}$ and the maximal size of a periodic orbit must fluctuate between $1$ and $q_N$.

We calculated these quantities for discretizations of orders $128 k$, for $k$ from $1$ to $100$ and have represented them graphically. For information, if $N =128\times 100$, then $q_N\simeq 1.6.10^8$.

\begin{figure}[h]
\begin{bigcenter}
\makebox[1.1\textwidth]{\parbox{1.1\textwidth}{%
\begin{minipage}[c]{.33\linewidth}
	\includegraphics[width=5cm,trim = .8cm .3cm .8cm .1cm,clip]{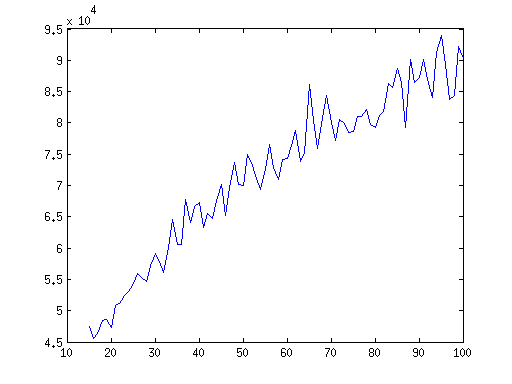}
\end{minipage}\hfill
\begin{minipage}[c]{.33\linewidth}
	\includegraphics[width=5cm,trim = .8cm .3cm .8cm .1cm,clip]{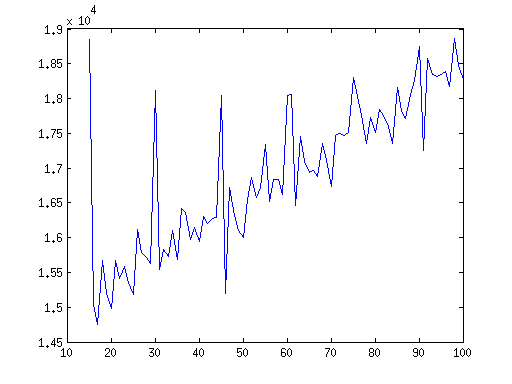}
\end{minipage}\hfill
\begin{minipage}[c]{.33\linewidth}
	\includegraphics[width=5cm,trim = .8cm .3cm .8cm .1cm,clip]{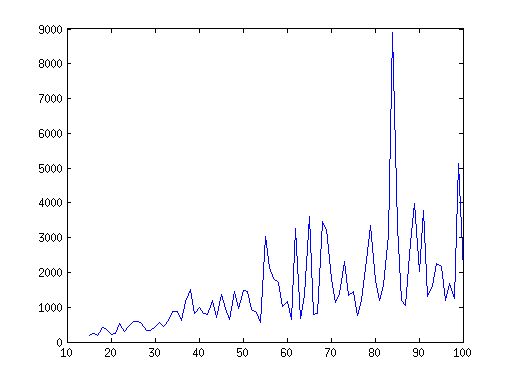}
\end{minipage}

\begin{minipage}[c]{.33\linewidth}
	\includegraphics[width=5cm,trim = .8cm .3cm .8cm .1cm,clip]{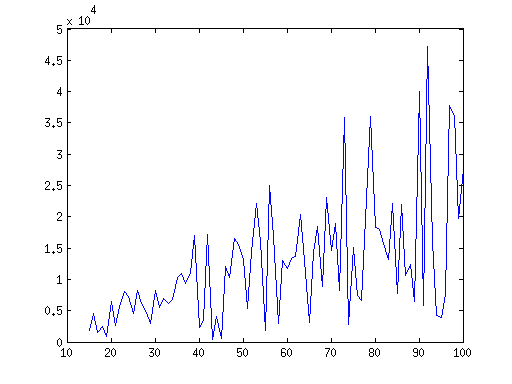}
\end{minipage}\hfill
\begin{minipage}[c]{.33\linewidth}
	\includegraphics[width=5cm,trim = .8cm .3cm .8cm .1cm,clip]{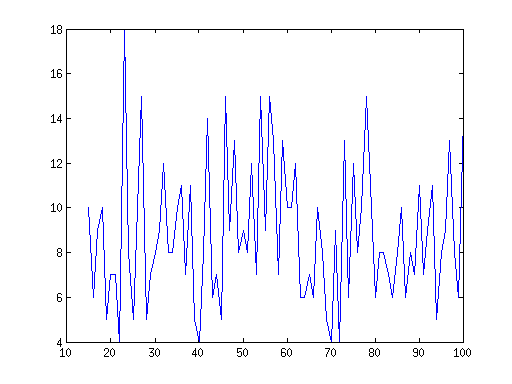}
\end{minipage}\hfill
\begin{minipage}[c]{.33\linewidth}
	\includegraphics[width=5cm,trim = .8cm .3cm .8cm .1cm,clip]{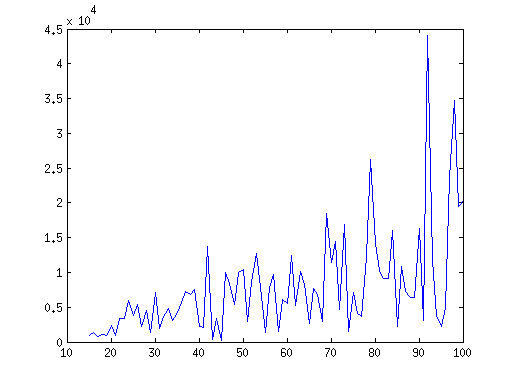}
\end{minipage}
}}
\end{bigcenter}
\caption{Size of the maximal invariant set $\Omega((f_i)_N)$ (left), number of periodic orbits of $(f_i)_N$ (middle) and length of the largest periodic orbit of $(f_i)_N$ (right) depending on $N$, for $f_1$ (top) and $f_2$ (bottom), on the grids $E_N$ with $N=128k$, $k=15,\dots,100$}\label{GrafCons}
\end{figure}

We begin with the cardinality of the maximal invariant set $\Omega(f_N)$ (figure \ref{GrafCons}). Contrary to what the theoretical results provide for a generic homeomorphism, for all simulations the ratio $\frac{\card(\Omega(f_N))}{q_N}$ tends to $0$ when $N$ increases. More specifically, the cardinality of $\Omega(f_N)$ evolves much more regularly for $f_1$ than for $f_2$:  for $f_1$ the value of this cardinality seems to be the sum of a smooth increasing function and a random noise, but for $f_2$ this value seems to be the product of a smooth increasing function with a random noise. We have no explanation to the parabolic shape of the curve for $f_1$: it reflects the fact that the cardinality of $\Omega((f_1)_N) $ evolves in the same way as $\sqrt[4]{q_N}$ (whereas for a random map of a finite set with $q$ elements into itself it evolves the same way as $\sqrt{q}$). Finally, it is interesting to note that for $f_2$, the size of the maximal invariant set is distributed more or less around the size of the maximal invariant set of a random map of a set with $q_N$ elements into itself, which depends (asymptotically) linearly of $N$ and is worth about $16,000$ for $N=128 \times 100$ (see \cite{Boll-rand}).

According to the results of the previous sections, for a generic homeomorphism $f$, the number of periodic orbits of $f_N$ must fluctuate between $1$ and (e.g.) $\sqrt{q_N}$. It is clearly not the case for the simulations. In addition, its behaviour is not the same for $f_1$ and for $f_2$ (figure \ref{GrafCons}): for $f_1$ the number of orbits reaches quickly a value around $2.10^4$ to increase slightly thereafter, while for $f_2$ it oscillates between $1$ and $18$, regardless the order of discretization. These graphics can be compared with those representing the size of the maximal invariant set $\Omega((f_i)_N)$: the number of periodic orbits and the size of the maximal invariant set are of the same order of magnitude for $f_1$ (up to a factor $5$), which means that the average period of a periodic orbit is small (which is not surprising, since $f_1$ is a small perturbation of identity). In contrast, they differ by a factor roughly equal to $10^3$ for $f_2$, which means that this time the average period of a periodic orbit is very large. This can be explained partly by the fact that the standard Anosov automorphism tends to mix what happens in the neighborhood of identity. The fact remains that these simulations (such as the size of the maximal invariant set $\Omega ((f_i)_N)$) suggest that the behaviour in the neighborhood of identity and of the standard Anosov automorphism are quite different, at least for small orders of discretization.

Regarding the maximum size of a periodic orbit of $f_N$ (figure \ref{GrafCons}), again its behaviour does not correspond to that of a generic homeomorphism: it should oscillate between $1$ and $q_N$, which is not the case. However, it varies widely depending on $N$, especially when $N$ is large. The qualitative behaviours are similar for both simulations, but there are some quantitative differences: the maximum of the maximal length of a periodic orbit is greater for $f_2$ than for $f_1$.

\addtocontents{toc}{\SkipTocEntry}\subsection{Behaviour of invariant measures}

We also simulated the measure $\mu_N^{f_i}$ for the two examples of conservative homeomorphisms $f_1$ and $f_2$ as defined in page \pageref{defex}. Our purpose is to test whether phenomena as described by theorem \ref{mesinv} can be observed in practice or not. We obviously can not expect to see the sequence $(\mu^{f_i}_N)_{N\in\N}$ accumulating on \emph{all} the invariant probability measures of $f$, since these measures generally form an infinite-dimensional convex set, but we can still test if it seems to converge or not.

We present images of sizes $128\times 128$ pixels representing in logarithmic scale the density of the measures $\mu_N^f$: each pixel is coloured according to the measure carried by the set of points of $E_N$ it covers. The blue corresponds to a pixel with very small measure and the red to a pixel with very high measure. Scales on the right of each image corresponds to the measure of one pixel in logarithmic scale to base $10$: if the green codes $-3$, then a green pixel will have measure $10^{-3}$ for $\mu_N^f$. For information, when Lebesgue measure is represented, all the pixels have a value about $-4.2$.\label{pagealgo}
\bigskip

Theoretically, the maximal value of the measure $\mu_N$ (figure \ref{GrafMaxmeasCons}) should oscillate between $1/q_N$ and $1$, but this is not the case for these examples. Again the behaviour is quite different for $f_1$ and for $f_2$: for $f_1$ it seems that the values of $\mu_N$ are well distributed between null function and a linear function of $N$ while for $f_2$ we observe peaks: there are a few values of $N$ for which $\mu_N$ is much higher than elsewhere. For $f_1$ the maximal value of $\mu_N$ seems globally grow with $N$, but with a quite irregular behaviour, apparently confirming theorem \ref{mesinv}. Finally, we note that the maximum value of $\mu_N$ is much lower for $f_2$ than for $f_1$.

\begin{figure}[H]
\begin{bigcenter}
\makebox[.75\textwidth]{\parbox{.75\textwidth}{%
\begin{minipage}[c]{.49\linewidth}
	\includegraphics[width=5cm,trim = .8cm .3cm .8cm .1cm,clip]{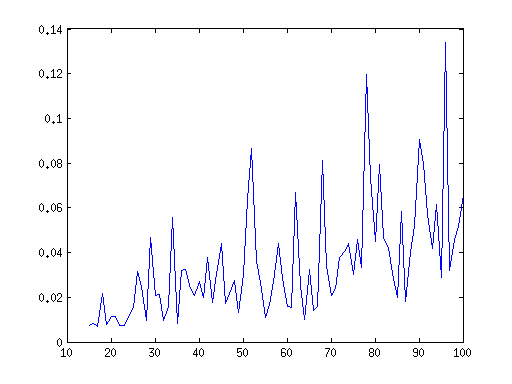}
\end{minipage}\hfill
\begin{minipage}[c]{.49\linewidth}
	\includegraphics[width=5cm,trim = .8cm .3cm .8cm .1cm,clip]{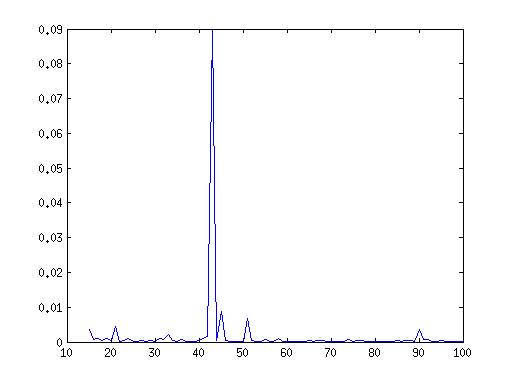}
\end{minipage}
}}
\end{bigcenter}
\caption{Maximal value of the measure $\mu_N$ on $E_N$ depending on $N$ for $f_1$ (left) and $f_2$ (right), on the grids $E_N$ with $N=128k$, $k=15,\dots,100$}\label{GrafMaxmeasCons}
\end{figure}

\newpage
\changepage{100pt}{}{}{}{}{-50pt}{}{}{}

\pagestyle{empty}

\begin{figure}[H]
\begin{bigcenter}
\makebox[1.25\textwidth]{\parbox{1.25\textwidth}{%

\begin{minipage}[c]{.33\linewidth}
	\includegraphics[width=6cm,trim = 1.2cm .3cm .6cm .8cm,clip]{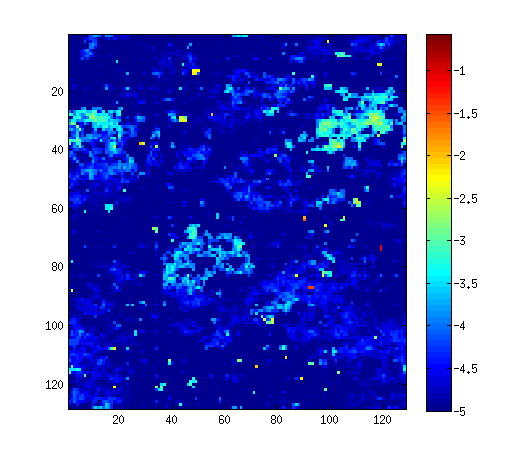}
\end{minipage}\hfill
\begin{minipage}[c]{.33\linewidth}
	\includegraphics[width=6cm,trim = 1.2cm .3cm .6cm .8cm,clip]{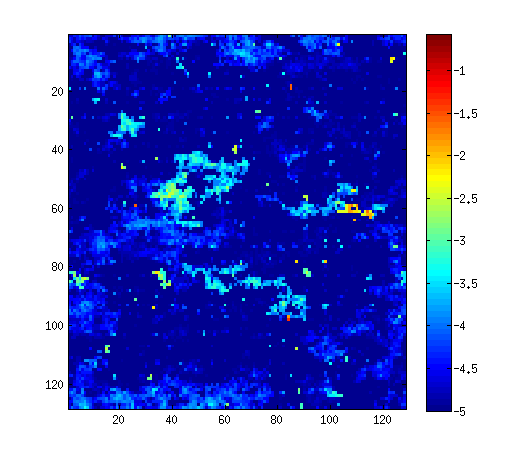}
\end{minipage}\hfill
\begin{minipage}[c]{.33\linewidth}
	\includegraphics[width=6cm,trim = 1.2cm .3cm .6cm .8cm,clip]{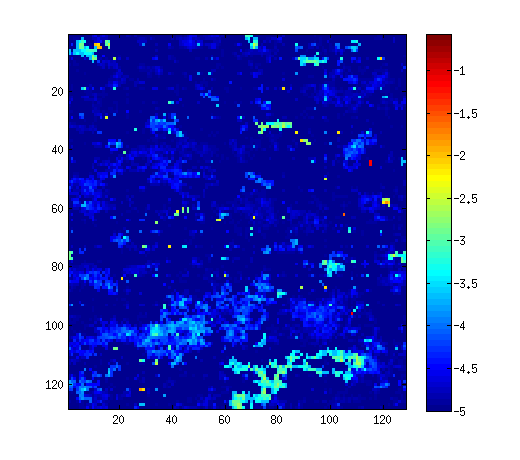}
\end{minipage}

\begin{minipage}[c]{.33\linewidth}
	\includegraphics[width=6cm,trim = 1.2cm .3cm .6cm .8cm,clip]{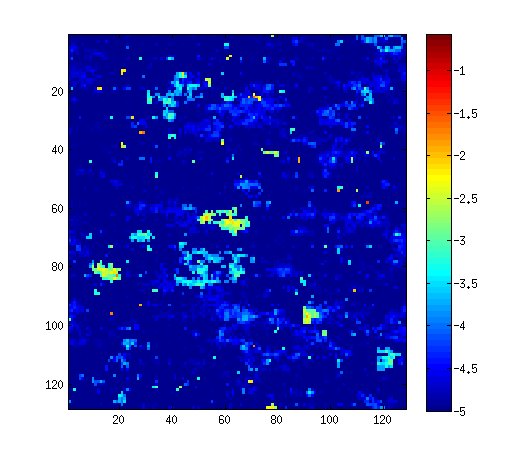}
\end{minipage}\hfill
\begin{minipage}[c]{.33\linewidth}
	\includegraphics[width=6cm,trim = 1.2cm .3cm .6cm .8cm,clip]{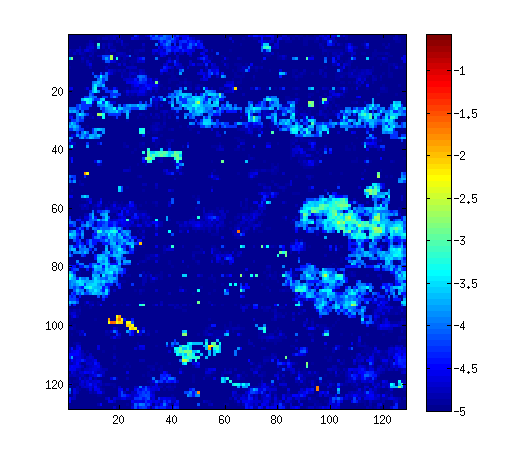}
\end{minipage}\hfill
\begin{minipage}[c]{.33\linewidth}
	\includegraphics[width=6cm,trim = 1.2cm .3cm .6cm .8cm,clip]{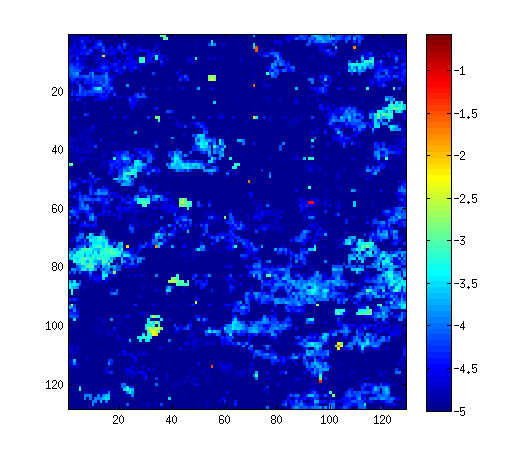}
\end{minipage}

\begin{minipage}[c]{.33\linewidth}
	\includegraphics[width=6cm,trim = 1.2cm .3cm .6cm .8cm,clip]{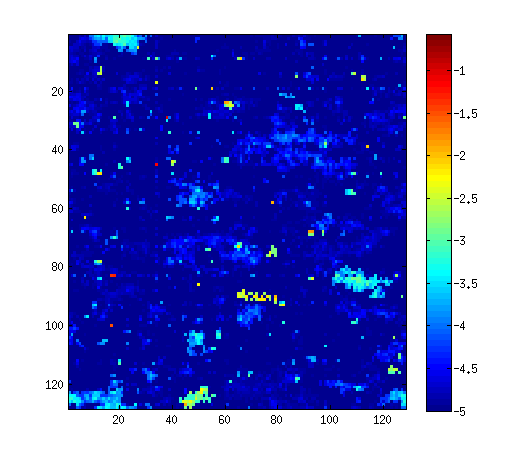}
\end{minipage}\hfill
\begin{minipage}[c]{.33\linewidth}
	\includegraphics[width=6cm,trim = 1.2cm .3cm .6cm .8cm,clip]{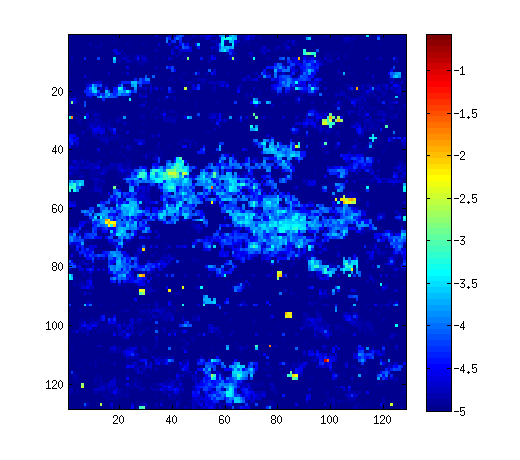}
\end{minipage}\hfill
\begin{minipage}[c]{.33\linewidth}
	\includegraphics[width=6cm,trim = 1.2cm .3cm .6cm .8cm,clip]{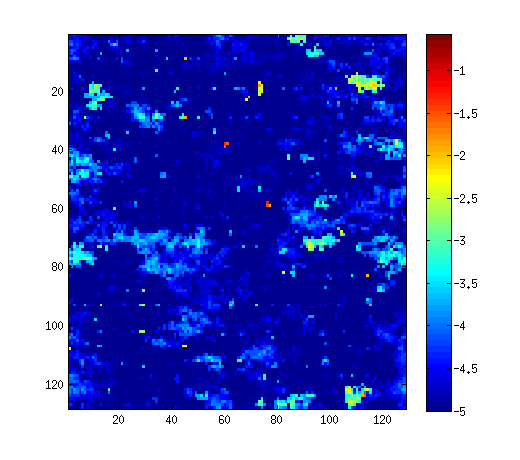}
\end{minipage}
}}
\caption{Simulations of invariant measures $\mu_N^{f_1}$ on grids of size $N\times N$, with $N=20000,\dots,20008$ (from left to right and top to bottom)}\label{MesC0IdConsPQPsinSer}
\end{bigcenter}
\end{figure}

The results of simulations of invariant measures of discretizations of $f_1$ (which is a $C^0$ conservative perturbation of identity) are quite positive: they agree with theoretical results about discretizations of generic conservative homeomorphisms, in particular with theorem \ref{mesinv}. When $f_1$ is discretized, we observe that the measure is first fairly well distributed. When the step of the discretization improves we can observe places where the measure accumulates; moreover these places changes a lot when the orders of discretizations  varies (see figure \ref{MesC0IdConsPQPsinSer}). This fairly agrees with what happens in the $C^0$ generic case, where the measure $\mu^f_N$ depends very much on the order of discretization rather than on the homeomorphism itself. There is also an other phenomenon: when the size of the grid is large enough (around $10^{12}\times 10^{12}$), it appears areas uniformly charged by the measure $\mu^f_N$; their sizes seems to be inversely proportional to the common mass of the pixels of the area.

\newpage

\begin{figure}[H]
\begin{bigcenter}
\makebox[1.25\textwidth]{\parbox{1.25\textwidth}{%

\begin{minipage}[c]{.33\linewidth}
	\includegraphics[width=6cm,trim = 1.2cm .3cm .6cm .8cm,clip]{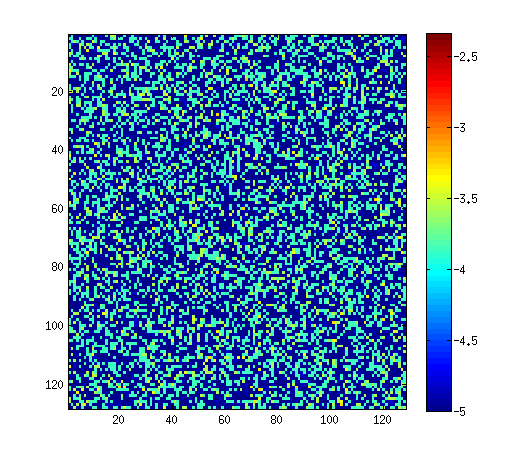}
\end{minipage}\hfill
\begin{minipage}[c]{.33\linewidth}
	\includegraphics[width=6cm,trim = 1.2cm .3cm .6cm .8cm,clip]{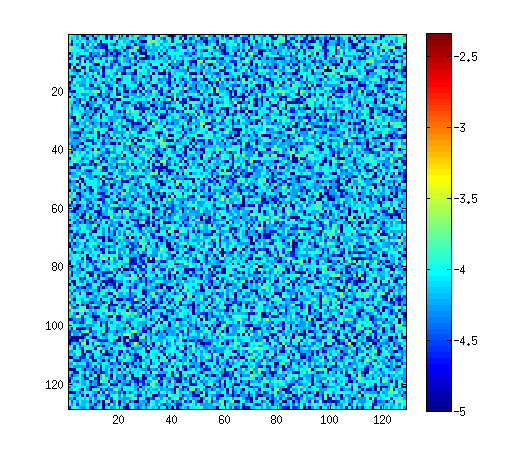}
\end{minipage}\hfill
\begin{minipage}[c]{.33\linewidth}
	\includegraphics[width=6cm,trim = 1.2cm .3cm .6cm .8cm,clip]{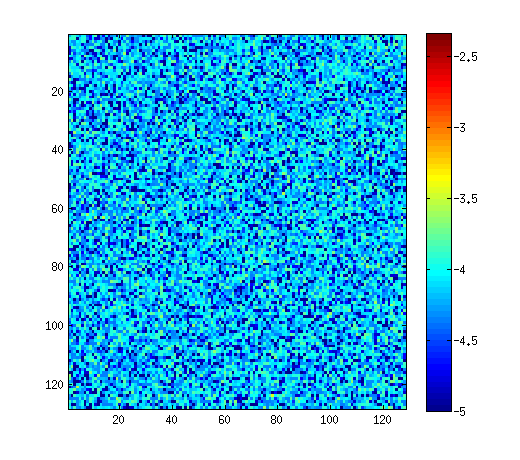}
\end{minipage}

\begin{minipage}[c]{.33\linewidth}
	\includegraphics[width=6cm,trim = 1.2cm .3cm .6cm .8cm,clip]{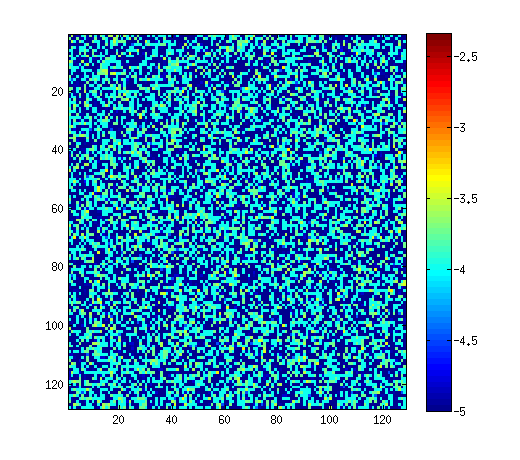}
\end{minipage}\hfill
\begin{minipage}[c]{.33\linewidth}
	\includegraphics[width=6cm,trim = 1.2cm .3cm .6cm .8cm,clip]{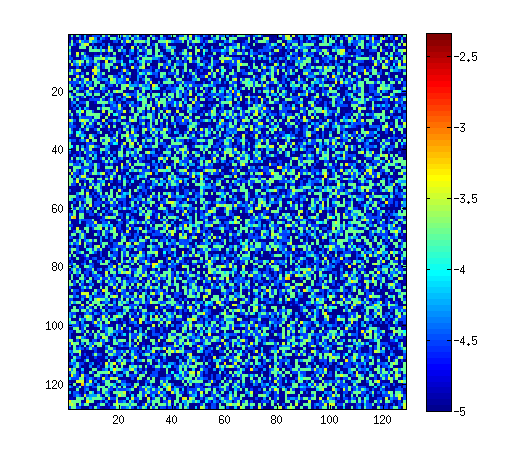}
\end{minipage}\hfill
\begin{minipage}[c]{.33\linewidth}
	\includegraphics[width=6cm,trim = 1.2cm .3cm .6cm .8cm,clip]{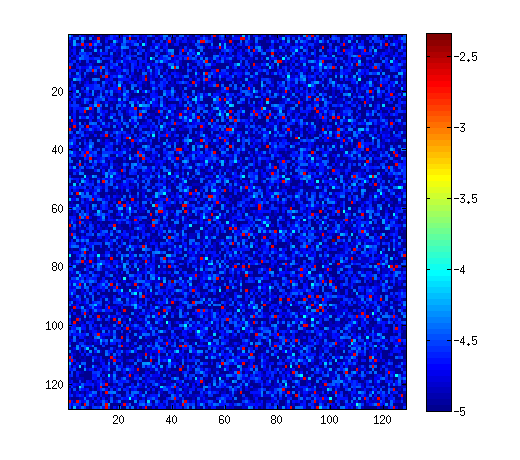}
\end{minipage}

\begin{minipage}[c]{.33\linewidth}
	\includegraphics[width=6cm,trim = 1.2cm .3cm .6cm .8cm,clip]{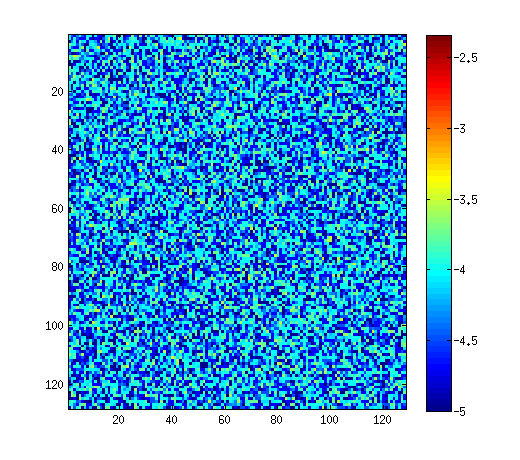}
\end{minipage}\hfill
\begin{minipage}[c]{.33\linewidth}
	\includegraphics[width=6cm,trim = 1.2cm .3cm .6cm .8cm,clip]{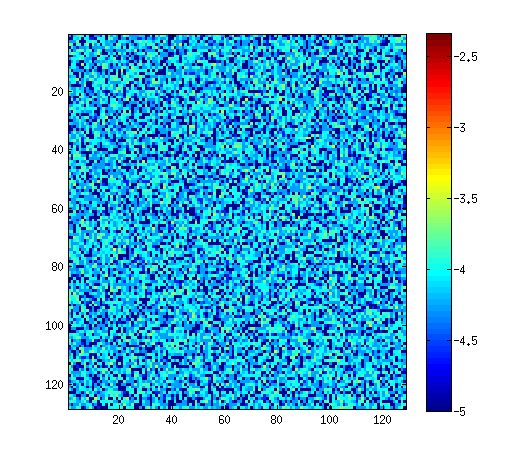}
\end{minipage}\hfill
\begin{minipage}[c]{.33\linewidth}
	\includegraphics[width=6cm,trim = 1.2cm .3cm .6cm .8cm,clip]{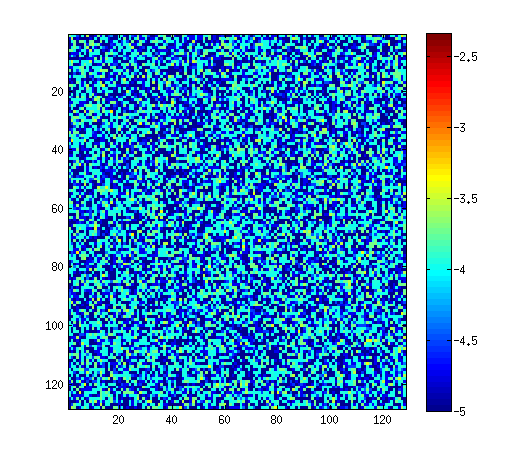}
\end{minipage}
}}

\caption{Simulations of invariant measures $\mu_N^{f_2}$ on grids of size $N\times N$, with $N=20005,\dots,20013$ (from left to right and top to bottom)}\label{MesC0AnoConsSer}
\end{bigcenter}
\end{figure}

\vspace{-5pt}
For the discretizations of $f_2$, which is a $C^0$ conservative perturbation of the standard Anosov automorphism, the simulations on grids of size $2^k\times 2^k$ might suggest that the measures $\mu^{f_2}_N$ tend to Lebesgue measure. In fact, making a large number of simulations, we realize that there is also strong variations of the behaviour of measures (figure \ref{MesC0AnoConsSer}): the measure is often well distributed in the torus and sometimes quite singular with respect to Lebesgue measure (as in figure \ref{GrafMaxmeasCons}). For example when we discretize on the grid of size $20010\times 20010$ (middle right of figure \ref{MesC0AnoConsSer}), one observe an orbit with length $369$ which mass $84\%$ of the total measure. In fact the behaviour of discretizations looks the same that in the neighborhood of identity, modulo the fact that the standard Anosov automorphism tends to spread the attractive periodic orbits of discretizations on the entire torus: for many values of $N$ composing by $A$ spreads the behaviour of the measure $\mu^{f_2}_N$, but sometimes (in fact seldom) a fixed point of $(f_1)_N$ which attracts a large part of $E_N$ is located around one of the few periodic points of small period for $A$. This then creates a periodic orbit for $(f_2)_N$ with a big measure for $\mu^{f_2}_N$.

\newpage
\changepage{-100pt}{}{}{}{}{50pt}{}{}{}
\pagestyle{headings}

\section{Dissipative homeomorphisms}

Let us finish by presenting the results of some numerical simulations of dissipative homeomorphisms. Again, our aim is to compare the theoretical results with the reality of numerical simulations: for simple homeomorphisms and reasonable orders of discretization does one have convergence of the dynamics of discretizations to that of the homeomorphism, as suggested by the above theorems?

Recall that we simulate homeomorphisms such that
\[f(x,y) = Q\circ P(x,y)\quad\text{or}\quad f(x,y) = P\circ Q\circ P(x,y),\]
where $P$ and $Q$ are two homeomorphisms of the torus that modify only one coordinate:
\[P(x,y) = (x,y+p(x,y))\quad\text{and}\quad Q(x,y) = (x+q(x,y),y).\]
Unlike the conservative case, $p$ and $q$ depends on both $x$ and $y$. Again, we tested two homeomorphisms:
\begin{itemize}\label{defbis}
\item To begin with we studied $f_3 = Q\circ P$, with
\[p(x,y) = \frac{1}{259}\cos(2\pi\times 227y)+\frac{1}{271}\sin(2\pi\times 233x),\]
\[q(x,y) = \frac{1}{287}\cos(2\pi\times 241y)+\frac{1}{263}\sin(2\pi\times 217y) + \frac{1}{263}\cos(2\pi\times 271x)\]
This dissipative homeomorphism is a small $C^0$ perturbation of identity, whose derivative has many oscillations whose amplitudes are close to $1$. That creates many fixed points which are attractors, sources or saddles.
\item It has also seemed useful to us to simulate a homeomorphism close to the identity in $C^0$ topology, but with a small number of attractors. Indeed, as explained heuristically by J.-M. Gambaudo and C. Tresser in \cite{Gamb-dif}, a homeomorphism such that $f_3$ can have a large number of attractors whose basins of attraction are small. It turns out that the dissipative behaviour of $f_3$ can not be detected in orders discretization that can be achieved in practice. We therefore defined another homeomorphism close to the identity in $C^0$ topology, but with much less attractors, say $f_4 =P \circ Q \circ P $, with
\[p(x,y) = \frac{1}{259}\Th\big(50\cos(2\pi\times y)\big)+\frac{1}{271}\Th\big(50\cos(2\pi\times 5x)\big),\]
\[q(x,y) = \frac{1}{287}\Th\big(50\cos(2\pi\times y)\big)+\frac{1}{263}\Th\big(50\cos(2\pi\times 7y)\big) +\frac{1}{263}\Th\big(50\cos(2\pi\times 3x)\big).\]
\end{itemize}

\addtocontents{toc}{\SkipTocEntry}\subsection{Combinatorial behaviour}\label{simulgrafdissip}

We simulated some quantities related to the combinatorial behaviour of discretizations of homeomorphisms, namely:
\begin{itemize}
\item the cardinality of the maximal invariant set $\Omega(f_N)$,
\item the number of periodic orbits of $f_N$,
\item the maximal size of a periodic orbit of $f_N$.
\end{itemize}
We calculated these quantities for discretizations of orders $128k$ for $k$ from $1$ to $100$ and represented it graphically. For information, if $N =128\times100$, then $q_N \simeq 1.6. 10^8$.

\begin{figure}[h]
\begin{bigcenter}
\makebox[1.1\textwidth]{\parbox{1.1\textwidth}{%
\begin{minipage}[c]{.33\linewidth}
	\includegraphics[width=5cm,trim = .8cm .3cm .8cm .1cm,clip]{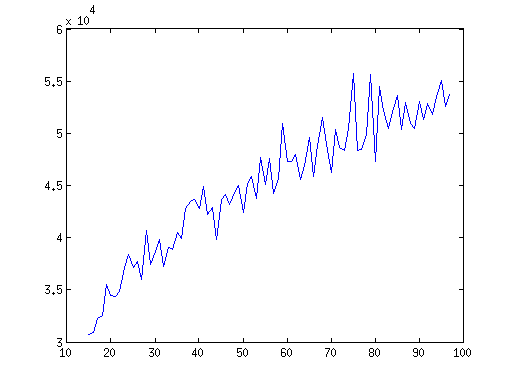}
\end{minipage}\hfill
\begin{minipage}[c]{.33\linewidth}
	\includegraphics[width=5cm,trim = .8cm .3cm .8cm .1cm,clip]{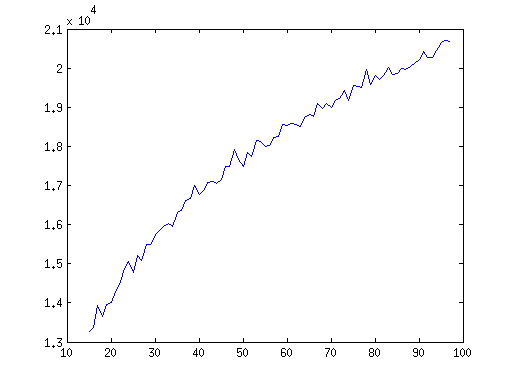}
\end{minipage}\hfill
\begin{minipage}[c]{.33\linewidth}
	\includegraphics[width=5cm,trim = .8cm .3cm .8cm .1cm,clip]{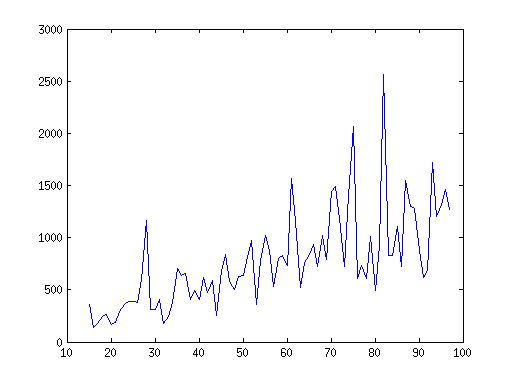}
\end{minipage}

\begin{minipage}[c]{.33\linewidth}
	\includegraphics[width=5cm,trim = .8cm .3cm .8cm .1cm,clip]{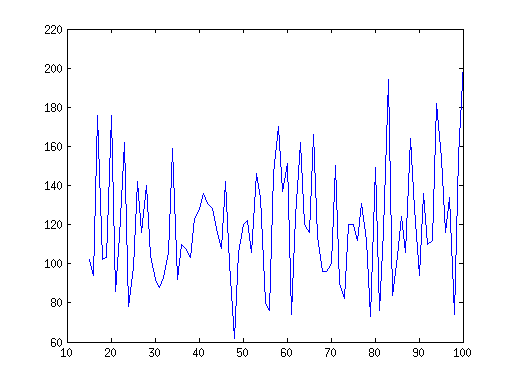}
\end{minipage}\hfill
\begin{minipage}[c]{.33\linewidth}
	\includegraphics[width=5cm,trim = .8cm .3cm .8cm .1cm,clip]{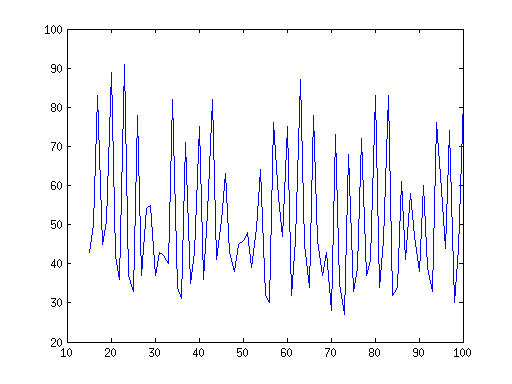}
\end{minipage}\hfill
\begin{minipage}[c]{.33\linewidth}
	\includegraphics[width=5cm,trim = .8cm .3cm .8cm .1cm,clip]{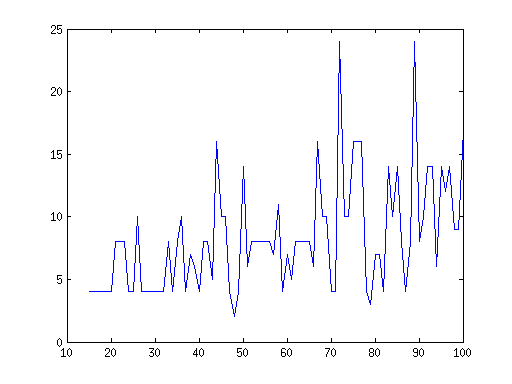}
\end{minipage}
}}
\end{bigcenter}
\caption{Size of the maximal invariant set $\Omega((f_i)_N)$ (left), number of periodic orbits of $(f_i)_N$ (middle) and length of the largest periodic orbit of $(f_i)_N$ (right) depending on $N$, for $f_3$ (top) and $f_4$ (bottom), on the grids $E_N$ with $N=128k$, $k=15,\dots,100$}\label{GrafCnDissip}
\end{figure}

Theoretically, the ratio between the cardinality of $\Omega(f_N)$ and $q_N$ must tend to $0$, this is what can be seen in simulations. This is not really surprising: we already observed that for discretizations of conservative homeomorphisms. In this context it is interesting to compare the behaviour of $\Omega(f_N)$ in the conservative and the dissipative case. The result is a bit disappointing: for $f_3$, the graphics in the case of conservative and dissipative homeomorphisms are as alike as two peas in a pod, while in theory they should be very different. Still, for $f_4$, there are very few attractors, and around each one there are just a few attractive points of the discretization (each attractor is ``sharp'') so that the cardinality of $\Omega((f_4)_N)$ is more or less constant. For this case, the behaviour of discretizations seems clearly typical of the dissipative case.
\bigskip

The behaviour of the number of periodic orbits of $f_N$ is not provided by the theoretical study. However, we can hope that it reflects the fact that the dynamics converges to that of the initial homeomorphism: among others, we can test if it is of the same order as the number of attractors of the homeomorphism. In practice, its behaviour is similar to that of the cardinality of the maximal invariant set $\Omega (f_N)$: its evolution is very regular for $f_3$ but oscillates (from $N=15$) between $30$ and $90$ for $f_4$. As in the conservative case, the comparison with figure \ref{GrafCnDissip} indicates that the average length of a periodic cycle is about $3$ for $f_3$ and $f_4$.
\bigskip

Since the dynamics of discretizations is assumed to converge to that of the initial homeomorphism, one could expect that the length of the longest periodic orbit of discretizations $(f_i)_N$ is almost always a multiple of that of an attractive periodic orbit of $f_i$. The graphic of this length for $f_3$ looks like the conservative case, so one can say that the dissipative behaviour of this homeomorphism is not detected for reasonable orders of discretizations. Despite this, the values of the lengths of the longest periodic orbit are much smaller (up to a factor $10$) than in the conservative case. It is not obvious if it is more a dissipative effect than a coincidence. For $f_4$, the length of the longest orbit is smaller than $25$, which is normal because we are supposed to detect the attractors of the homeomorphism. The changes of the length of the longest orbit are probably due to errors of discretization if the neighborhood of the attractors of $f_4$. Furthermore, we note that some values of the size of the longest orbit are more achieved: for example for $30$ different integers $N$ the longest periodic orbit has length $4$, for $19$ different integers $N$ this length is $8$ and  for $10$ different integers $N$ it is $10$. This is certainly due to the shadowing property of periodic orbits of $f_4$ by those of the discretizations: see the work of M. Blank for a discussion about the phenomenon of multiplication of the period \cite{MR1037009,MR875433,MR765293,MR1299502}.

\addtocontents{toc}{\SkipTocEntry}\subsection{Behaviour of invariant measures}

As in the conservative case we calculated the invariant measures $\mu_N^{f_i}$ of dissipative homeomorphisms $f_3$ and $f_4$ as defined on page \pageref{defbis}. Our aim is to test whether theorem \ref{convmesdissip} apply in practice or if there are technical constraints such that this behaviour can not be observed on these examples. For a presentation of the representations of the measures, see page~\pageref{pagealgo}.

\begin{figure}[H]
\begin{bigcenter}
\makebox[.75\textwidth]{\parbox{.75\textwidth}{%
\begin{minipage}[c]{.49\linewidth}
	\includegraphics[width=5cm,trim = .8cm .3cm .8cm .1cm,clip]{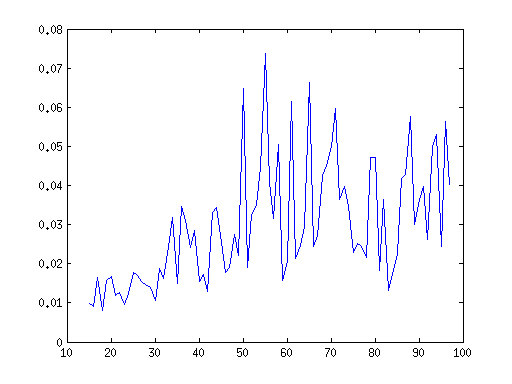}
\end{minipage}\hfill
\begin{minipage}[c]{.49\linewidth}
	\includegraphics[width=5cm,trim = .8cm .3cm .8cm .1cm,clip]{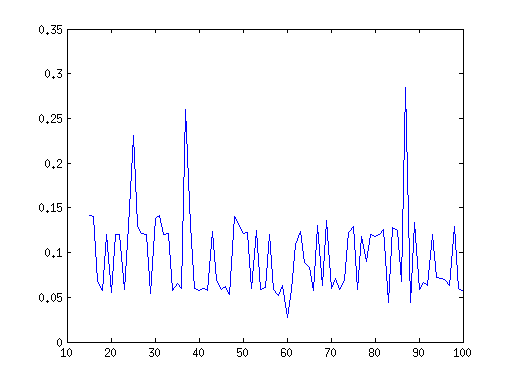}
\end{minipage}
}}
\end{bigcenter}
\caption{Maximal value of $\mu_N$ on $E_N$ depending on $N$ for $f_3$ (left) and $f_4$ (right), on the grids $E_N$ with $N=128k$, $k=15,\dots,100$}\label{GrafMaxmeasDissip}
\end{figure}

There is no clear theoretical results concerning the behaviour of the maximal value of the measure $\mu_N^f$ (figure \ref{GrafMaxmeasDissip}), but we can at least observe that a convergence of this amount would be a good indication that the considered homeomorphism is dissipative. As in the conservative case (in the neighborhood of identity), the behaviour of $\mu_N^{f_3}$ seems very chaotic. While $\mu_N^{f_3}$ is always smaller than $0.07$, the maximum value of $\mu_N^{f_4}$ is very high and fluctuates between $0.02$ and $0.3$, with many values around $0.06$ and $0.12$ and some values around $0.25$. The obvious arithmetic relationships between these values can be explained by the fact that some attractors of discretizations merge for some values of $N$. Thus, there are orbits which attract many points; it seems to reflect the typical behaviour of generic homeomorphisms.

\newpage
\changepage{100pt}{}{}{}{}{-50pt}{}{}{}

\pagestyle{empty}

\begin{figure}[H]
\begin{bigcenter}
\makebox[1.25\textwidth]{\parbox{1.25\textwidth}{%

\begin{minipage}[c]{.33\linewidth}
	\includegraphics[width=6cm,trim = 1.2cm .3cm .6cm .8cm,clip]{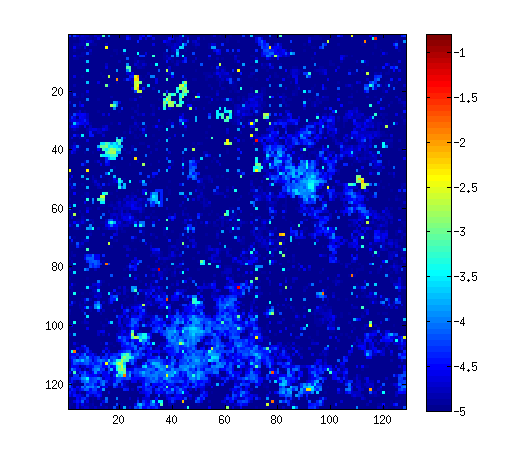}
\end{minipage}\hfill
\begin{minipage}[c]{.33\linewidth}
	\includegraphics[width=6cm,trim = 1.2cm .3cm .6cm .8cm,clip]{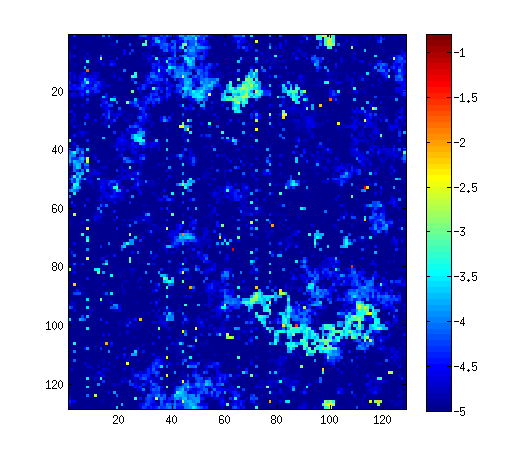}
\end{minipage}\hfill
\begin{minipage}[c]{.33\linewidth}
	\includegraphics[width=6cm,trim = 1.2cm .3cm .6cm .8cm,clip]{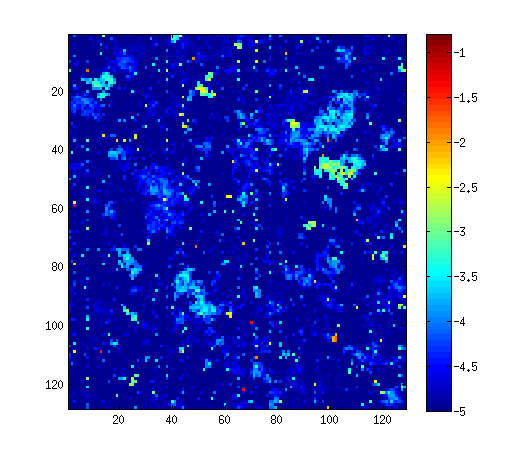}
\end{minipage}

\begin{minipage}[c]{.33\linewidth}
	\includegraphics[width=6cm,trim = 1.2cm .3cm .6cm .8cm,clip]{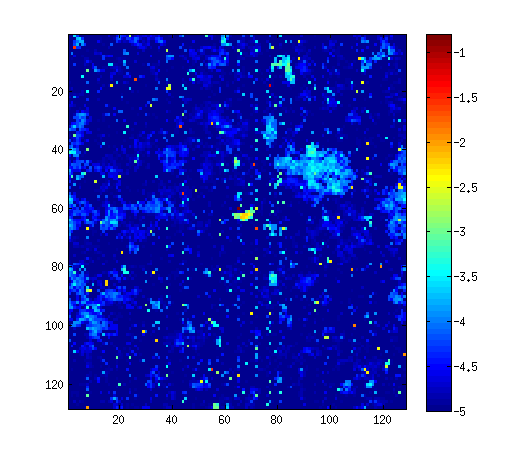}
\end{minipage}\hfill
\begin{minipage}[c]{.33\linewidth}
	\includegraphics[width=6cm,trim = 1.2cm .3cm .6cm .8cm,clip]{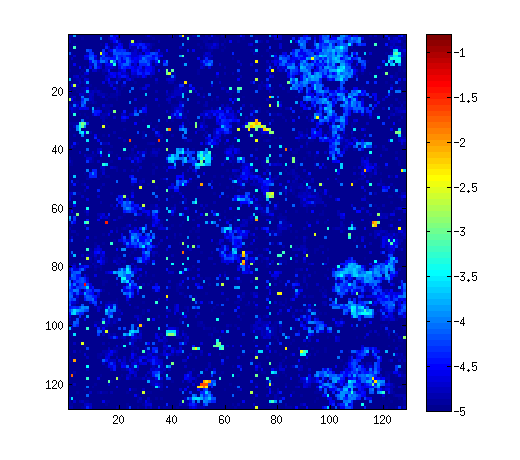}
\end{minipage}\hfill
\begin{minipage}[c]{.33\linewidth}
	\includegraphics[width=6cm,trim = 1.2cm .3cm .6cm .8cm,clip]{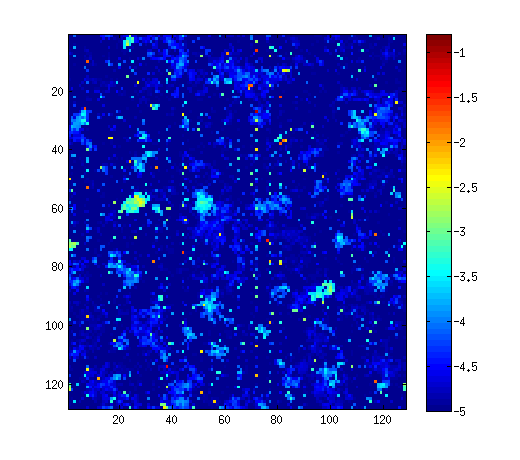}
\end{minipage}

\begin{minipage}[c]{.33\linewidth}
	\includegraphics[width=6cm,trim = 1.2cm .3cm .6cm .8cm,clip]{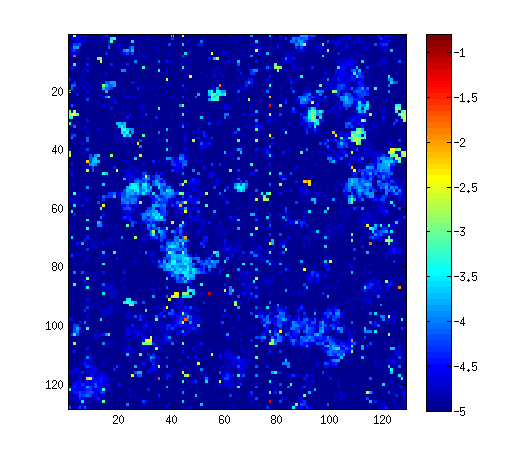}
\end{minipage}\hfill
\begin{minipage}[c]{.33\linewidth}
	\includegraphics[width=6cm,trim = 1.2cm .3cm .6cm .8cm,clip]{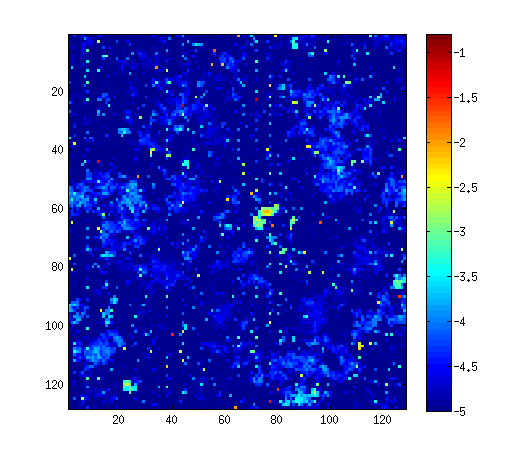}
\end{minipage}\hfill
\begin{minipage}[c]{.33\linewidth}
	\includegraphics[width=6cm,trim = 1.2cm .3cm .6cm .8cm,clip]{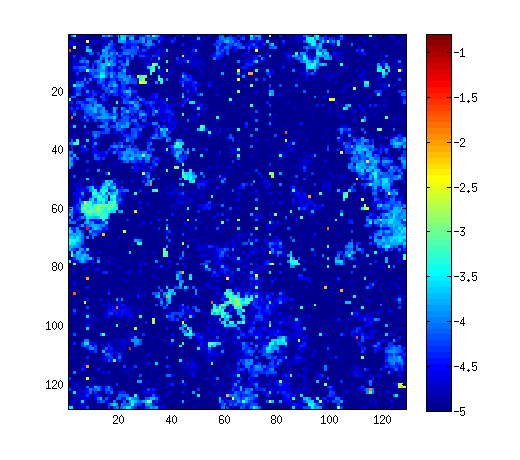}
\end{minipage}
}}
\caption{Simulations of $\mu_N^{f_3}$ on grids $E_N$, with $N=20000,\dots,20008$ (from left to right and top to bottom)}\label{MesC0IdDissipSer}
\end{bigcenter}
\end{figure}
\vspace{-5pt}

The behaviour of invariant measures for $f_3$, which is a small $C^0$ dissipative perturbation of identity, is relatively similar to that of invariant measures for $f_1$ i.e. the corresponding conservative case: when the order discretization is large enough, there is a strong variation of the measure $\mu_N^{f_3}$. Moreover this measure has a significant absolutely continuous component with respect to Lebesgue measure. Nevertheless, there are differences with the conservative case: the maximum value of $\mu_N^{f_3}$ is very large, much more than for $f_1$, and there are orbits which attract many points. On the other hand, one can observe the following phenomenon (we hope that it is anecdotal): these orbits are located on vertical stripes. This is probably related to specific arithmetic phenomena due to the specific form of the homeomorphism $f_3$.

It is normal not to be able to observe all the attractors of $f_3$ on these simulations: as noted by J.-M. Gambaudo and C. Tresser in \cite{Gamb-dif}, the size of these attractors can be very small compared to the numbers involved in the definition of $f_3$. So even in orders discretization such as $2^{15}$, the dissipative nature of the homeomorphism can not be detected on discretizations. For this reason, it looks as if the discretizations of $f_3$ are very similar to those of $f_1$.

\newpage

\begin{figure}[H]
\begin{bigcenter}
\makebox[1.25\textwidth]{\parbox{1.25\textwidth}{%

\begin{minipage}[c]{.33\linewidth}
	\includegraphics[width=6cm,trim = 1.2cm .3cm .6cm .8cm,clip]{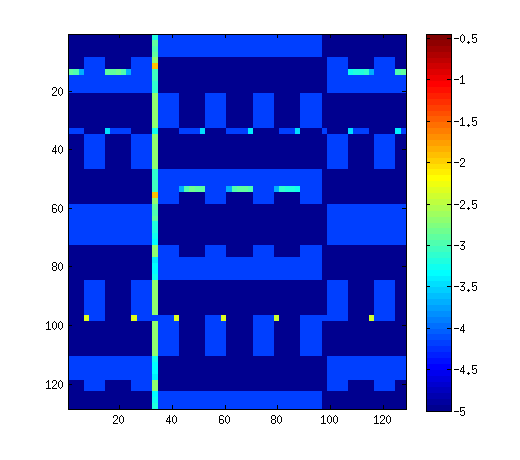}
\end{minipage}\hfill
\begin{minipage}[c]{.33\linewidth}
	\includegraphics[width=6cm,trim = 1.2cm .3cm .6cm .8cm,clip]{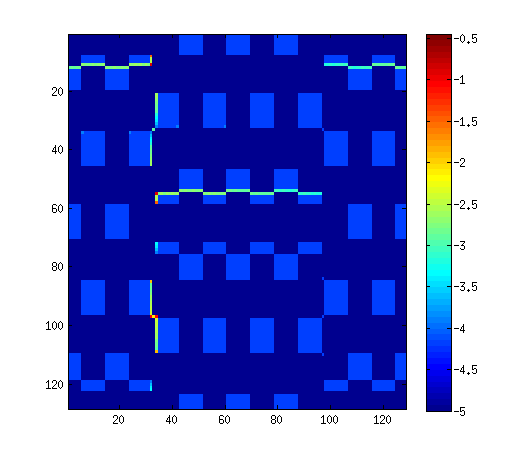}
\end{minipage}\hfill
\begin{minipage}[c]{.33\linewidth}
	\includegraphics[width=6cm,trim = 1.2cm .3cm .6cm .8cm,clip]{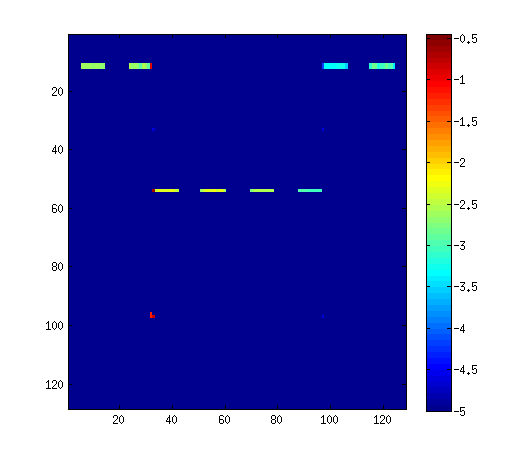}
\end{minipage}

\begin{minipage}[c]{.33\linewidth}
	\includegraphics[width=6cm,trim = 1.2cm .3cm .6cm .8cm,clip]{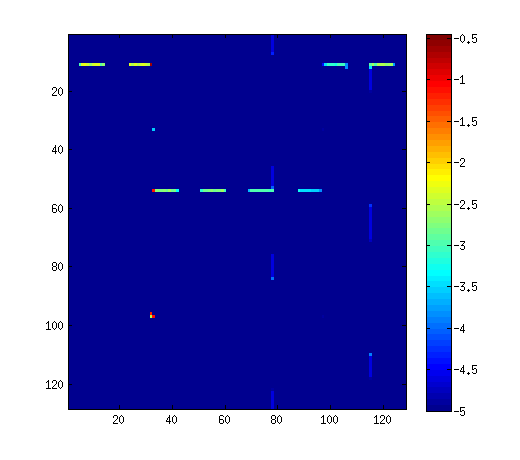}
\end{minipage}\hfill
\begin{minipage}[c]{.33\linewidth}
	\includegraphics[width=6cm,trim = 1.2cm .3cm .6cm .8cm,clip]{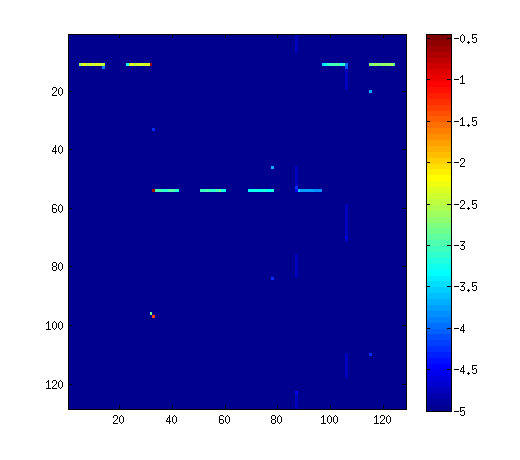}
\end{minipage}\hfill
\begin{minipage}[c]{.33\linewidth}
	\includegraphics[width=6cm,trim = 1.2cm .3cm .6cm .8cm,clip]{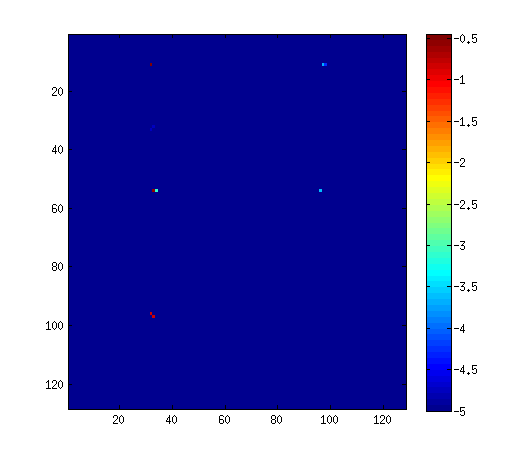}
\end{minipage}

\begin{minipage}[c]{.33\linewidth}
	\includegraphics[width=6cm,trim = 1.2cm .3cm .6cm .8cm,clip]{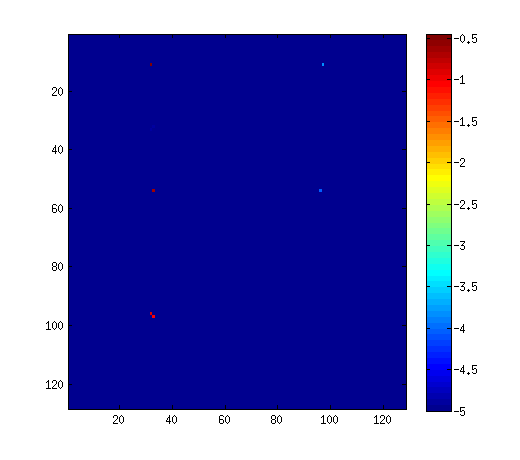}
\end{minipage}\hfill
\begin{minipage}[c]{.33\linewidth}
	\includegraphics[width=6cm,trim = 1.2cm .3cm .6cm .8cm,clip]{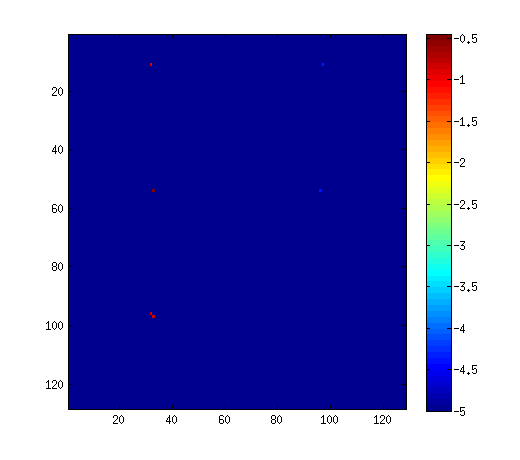}
\end{minipage}\hfill
\begin{minipage}[c]{.33\linewidth}
	\includegraphics[width=6cm,trim = 1.2cm .3cm .6cm .8cm,clip]{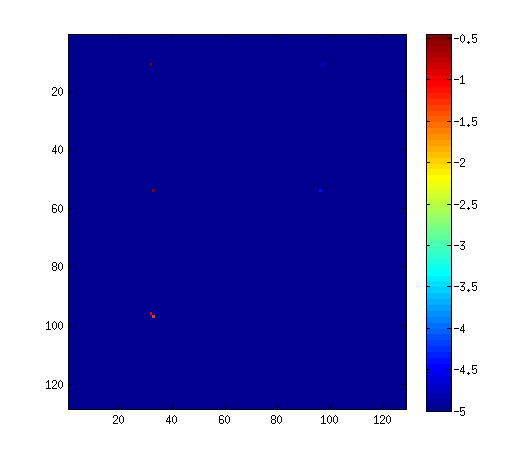}
\end{minipage}
}}
\caption{Simulations af $\mu_N^{f_4}$ on the grids $E_N$, with $N = 2^k$ and $k=6,\dots,14$ (from left to right and top to bottom)}\label{MesC0C1IdDissip}
\end{bigcenter}
\end{figure}

Recall that what happens for $f_3$ is rather close to what happens for $f_1$. For their part, simulations of invariant measures for $f_4$ on grids of size $2^k \times 2^k$ highlight that we expect from a generic dissipative homeomorphism: the measures $\mu^{f_4}_N$ tend quickly to a single measure (that is also observed on a series of simulations), which is carried by the attractors of $f_4$. The fact that $f_4$ has few attractors allows the discretizations of reasonable orders (typically $2^{11}$) to find the actual attractors of the initial homeomorphism, contrary to what we had observed for $f_3$.

\newpage

\changepage{-100pt}{}{}{}{}{50pt}{}{}{}
\pagestyle{plain}

\addtocontents{toc}{\SkipTocEntry}\subsection*{Acknowledgments} The author warmly thanks François Béguin for his sound advises and constant support throughout the achievement of this work, as well as the two anonymous reviewers for their insightful comments.

\bibliographystyle{amsalpha}
\bibliography{../../../Biblio}

\end{document}